\documentclass{amsart}
\usepackage{amsmath,bm,xcolor}
\usepackage{amsthm}
\usepackage{amssymb}
\usepackage{amsfonts}
\usepackage{graphicx}
\usepackage{mathrsfs,bm,amscd}
\usepackage{latexsym}
\usepackage{color}
\usepackage{marginnote}
\usepackage{hyperref}
\usepackage{mathtools}
\usepackage{enumerate}
\usepackage{MnSymbol}
\usepackage{float}

\usepackage[all]{xy}

\numberwithin{equation}{section}
\theoremstyle{plain}
\newtheorem{thm}{Theorem}[section]
\newtheorem{prop}[thm]{Proposition}
\newtheorem{cor}[thm]{Corollary}
\newtheorem{lem}[thm]{Lemma}

\theoremstyle{definition}
\newtheorem{definition}[thm]{Definition}
\newtheorem{rem}[thm]{Remark}
\newtheorem{quest}[thm]{Question}

\newtheorem*{bprob*}{Bonus problem}

\def\R{\mathbb{R}}

\def\N{\mathbb{N}}

\def\W{\mathcal{W}}

\def\e{\varepsilon}
\def\g{\gamma}
\def\d{\delta}
\def\D{\Delta}

\def\G{\Gamma}

\def\V{\mathbb{V}}
\def\t{{\bf T}}

\DeclareMathOperator{\dist}{dist}

\DeclareMathOperator{\post}{post}
\DeclareMathOperator{\gen}{gen}
\DeclareMathOperator{\diam}{diam}

\DeclareMathOperator{\val}{Val}

\DeclareMathOperator{\level}{\mathtt{L}}

\newcommand{\cW}{\mathcal{W}}

\newcommand{\cT}{\mathcal{T}}

\newcommand{\card}{\operatorname{card}}

\newcommand{\interior}{\mathrm{int}}

\title{Quasiconformal uniformization and extension for planar trees}

\author{Fotis Ioannidis}
\author{Vyron Vellis}

\thanks{V.V. was partially supported by NSF DMS grant 2154918.}
\date{\today}

\address{Department of Mathematics\\ The University of Tennessee\\ Knoxville, TN 37966}
\email{fioannid@vols.utk.edu}
\address{Department of Mathematics\\ The University of Tennessee\\ Knoxville, TN 37966}
\email{vvellis@utk.edu}

\begin{document}

\begin{abstract}
A quasiconformal uniformization theorem of Rickman \cite{Rickman} asserts that every bounded turning planar arc is the image of the unit segment under a quasiconformal self-map of the plane. We prove an analogous uniformization theorem for planar trees by showing that every planar quasiconformal tree with uniform branch separation is quasiconformally equivalent to a quasiconvex planar tree with uniform branch separation and Hausdorff dimension arbitrarily close to 1. Along the way, we show that every quasisymmetric (resp. bi-Lipschitz) homeomorphism between planar quasiconformal trees that respects the counterclockwise orientation of the branches at every branch point can be extended quasisymmetrically (resp. bi-Lipschitz) to the whole plane. As a corollary of the extension theorem we construct a class of quadrivalent quasiconformal trees, which includes the Vicsek fractal, so that each tree in this class is the Julia set of a uniformly quasiregular map of $\R^2$. 
\end{abstract}

\maketitle


\section{Introduction}
Given a set $E\subset \R^2$ (potentially without rectifiable curves), when can we apply a ``nearly-conformal'' deformation of $\R^2$ that turns it into a convex, or almost convex set? Questions of this nature fall under the umbrella of the \emph{uniformization problem} and originate from the Riemann mapping theorem, which states that every simply connected proper subdomain of $\R^2$, no matter how complicated boundary it may have, is always conformally equivalent to the unit disk.

In order to obtain a global homeomorphism for non-smooth sets, one has to give up conformality and replace it with a weaker notion. Roughly speaking, a homeomorphism of $\R^2$ is \emph{quasiconformal} if it maps infinitesimal disks to infinitesimal ellipses of bounded eccentricity. While quasiconformal maps are not $C^1$, they preserve many geometric properties, so any set which is globally quasiconformally homeomorphic to a convex set must itself possess some nice geometry. Ahlfors \cite{Ahl63} and later Tukia and V\"ais\"al\"a \cite{TV80} introduced the concept of bounded turning sets. Formally, a set $X\subset \mathbb R^n$ (or a general metric space) is called \emph{bounded turning} if for every pair of points $x,y\in X$ there exists a continuum $K\subset X$ containing $x$ and $y$ with diameter at most a fixed multiple of $|x-y|$. Bounded turning sets do not have cusps but can exhibit fractal behavior (e.g. the von Koch curve). Building on a uniformization theorem of Ahlfors \cite{Ahl63} for Jordan curves, Rickman \cite{Rickman} proved the following.

\begin{thm}[{\cite[Theorem 1]{Rickman}}]\label{thm:Rick}
A set $X\subset\R^2$ is the image of a line segment under a quasiconformal homeomorphism of the plane if and only if it is a bounded turning arc.
\end{thm}

Recall that line segments are the only convex arcs. Later, Tukia and V\"ais\"al\"a \cite{TV80} gave a metric version of Rickman's theorem: a metric space is quasisymmetrically equivalent to a line segment if and only if it is doubling and bounded turning arc. Quasisymmetry is considered the analogue of quasiconformality in the metric space setting; see Section \ref{sec:prelim} for definitions.

Beyond arcs, the next most topologically simple class is that of trees. A metric space $X$ is called a \emph{metric tree} (or \emph{dendrite}) if it is a Peano continuum (i.e., compact, connected, and locally  connected) that does not contain simple closed curves. On the one hand, trees are among the most simple 1-dimensional connected metric spaces, but on the other hand they have infinitely many topological equivalence classes. Bounded turning and doubling trees are known as \textit{quasiconformal trees} (in the planar case they are also known as \emph{Gehring trees \cite{RohdeLin}}). These trees were introduced by Kinneberg \cite{Kinneberg} and have been the focus of extensive research in the last decade \cite{BM20,BM22,FreeGart1,FreeGart2,CGIV}. Quasiconformal trees have appeared in complex dynamics \cite{CJY,CG-dynamics}, dimension theory \cite{BishopTyson,Azzam}, geometric group theory \cite{McMullen}, harmonic analysis \cite{Bishop}, random processes \cite{RohdeLin}, and geometric function theory \cite{NV91,CGV}. 

Bonk and Meyer \cite{BM20} recently proved the analogue of Tukia-V\"ais\"al\"a theorem for metric trees: a metric space is quasisymmetrically equivalent to a doubling geodesic tree of Hausdorff dimension arbitrarily close to 1 if and only if it is a quasiconformal tree. Note that geodesicity is the analogue of convexity in the metric setting: a set is geodesic if every two points can be joined by an arc of length equal to the distance of the two points. In this paper, we are interested in the quasiconformal version of the Bonk-Meyer theorem and the generalization of Theorem \ref{thm:Rick}. Note that there exist no convex trees in $\R^2$ (or any Euclidean space) other than line segments. Therefore, we need to consider a slightly weaker notion. 

\begin{definition}\label{def:semigeodesic}
We say that a metric tree $T$ is \emph{semi-geodesic} if 
\begin{enumerate}
\item there exists $C\geq 1$ such that for every two points $x,y \in T$ there exists an arc that joins them of length at most $C d(x,y)$,
\item for each $x\in T\setminus \mathcal{L}(T)$ there exists a line segment $I\subset T$ that contains $x$ in its interior.
\end{enumerate}
\end{definition}
Here and for the rest of the paper, $\mathcal{B}(T)$ and $\mathcal{L}(T)$ denote the set of branch points and leaves, respectively, of $T$. Trees that satisfy the first property of Definition \ref{def:semigeodesic} are known in literature as \emph{quasiconvex} trees. Note that if an arc is semi-geodesic, then it is a line segment.

\begin{quest}\label{quest1}
Is every planar quasiconformal tree $T$ quasiconformally homeomorphic to a semi-geodesic $T'$?
\end{quest}

The motivation behind Question \ref{quest1} comes from two well-known questions in analysis. The first is to understand which planar trees are Julia sets of uniformly quasiregular mappings; see \textsection\ref{sec:intro-UQR}. Having a classification as in Question \ref{quest1} means that only semi-geodesic trees need to be studied. The second question is to classify those planar sets that arise as hulls of the \emph{Lowner equation}. Based on Theorem \ref{thm:Rick}, Marshall and Rohde \cite{MR} found an intimate connection between bounded turning arcs and solutions of the Loewner equation with $1/2$-H\"older driving functions; see also \cite{LR}. Recently, the Loewner equation has been generalized to Loewner processes where the driving function is replaced by driving measures which, instead of arcs, their solutions yield metric trees \cite{HeaM}. Currently, there is no theory between quasiconformal trees and regularity of the driving measures as in \cite{MR}, but a classification as in Question \ref{quest1} would allow us to study this problem only for semi-geodesic trees. 

Question \ref{quest1} requires overcoming two major challenges, both having to do with how one constructs $T'$. The result of Bonk and Meyer \cite{BM20} guarantees the existence of an abstract geodesic tree $T''$ that is quasisymmetric to $T$, but whether $T''$ (or any doubling geodesic tree) can be quasisymmetrically embedded in $\R^2$ is the context of an open question of Bonk and Meyer \cite[Section 10]{BM20}. Furthermore, even if such an embedding of $T''$ exists, another topological obstacle arises: is the embedded image of $T''$ ambiently homeomorphic to $T$? (See \textsection\ref{sec:intro-ext} for further discussion.) To overcome the first obstacle, we consider quasiconformal trees with branch separation introduced and studied by Bonk and Meyer \cite{BM22} which are known to admit quasisymmetric embeddings into $\R^2$ \cite{CGIV}.

\begin{definition}[{Uniform relative separation of branch points}]\label{def:unifsep}
A metric tree $(T,d)$ is said to have \textit{uniformly (relatively) separated branch points} if there exists $C\geq 1$ with the property: if $p,q\in T$ are distinct branch points, and if $B_{T}(p)$ (resp.~$B_{T}(q)$) is the third largest component of $T\setminus \{p\}$ (resp.~$T\setminus\{q\}$) in diameter, then
$$d(p,q)\geq C^{-1}\min\{\diam{B_{T}(p)},\diam{B_{T}(q)}\}.$$ 
\end{definition}

Most examples of quasiconformal trees appearing in analysis \cite{CG-dynamics,BishopTyson, Bishop} are of uniform branch separation making this class of trees a very natural class. Note that uniform branch separation is invariant under global quasiconformal maps (\cite[Lemma 4.3]{BM22}). The following is the first result of this paper.

\begin{thm}\label{thm:main thm}
Let $T$ be a planar tree with bounded valence, and let $\epsilon>0$.
\begin{enumerate}
\item There exists a homeomorphism $f:\R^2\to\R^2$ such that $f(T)$ is a semi-geodesic tree with $\dim_H{f(T)}<1+\epsilon$.
\item If $T$ is bounded turning and has uniform branch separation, then there exists a quasiconformal homeomorphism $f:\R^2\to\R^2$ such that $f(T)$ is a semi-geodesic tree with $\dim_H{f(T)}<1+\epsilon$.
\end{enumerate}
\end{thm}

The theorem is quantitative in that the quasiconformality constant and the semi-geodesicity constant depend at most on the bounded turning constant of $T$, the valence of $T$, the constant of branch separation, and $\epsilon$; see Theorem \ref{thm:quantitative version}. Moreover, if $T$ is an arc, Theorem \ref{thm:main thm} immediately yields Rickman’s Theorem \ref{thm:Rick}.

Closely related to the uniformization problem, is the notion of \emph{conformal dimension} which measures the extent to which a set can be dimensionally simplified by quasisymmetric deformations. Specifically, the conformal dimension of a metric space $X$ is defined by
$$\operatorname{Cdim}(X)= \inf \{\dim_H(Y) : F:X\to Y \text{ is a quasisymmetric homeomorphism}\}.$$
We refer the reader to \cite{MT10} for a general overview and to \cite{Pan89,Bourdon97} for foundational results in conformal dimension theory. Kinneberg \cite{Kinneberg} showed that the conformal dimension of every quasiconformal tree is 1, though it is known that the dimension 1 is not attained in general; see \cite{BishopTyson}, \cite{Azzam}.

Mackay and Tyson \cite[Chapter 8]{MT10} introduced the notion of \textit{global quasiconformal dimension}, which measures the extent to which a set can be dimensionally simplified by quasiconformal deformations. For a set $S\subset \mathbb R^n$, its global quasiconformal dimension is defined by
$$\operatorname{GQC^{n}dim}(S)= \inf \{ \dim_H(F(S)) : F:\mathbb R^n\to\mathbb R^n \text{ is quasiconformal}\}.$$
The following corollary is an immediate consequence of our theorem. 

\begin{cor}
If $T\subset \R^2$ is a quasiconformal tree with uniform branch separation, then $\operatorname{GQC^{2}dim}(T)=1$.
\end{cor}

\subsection{Extension properties of planar trees}\label{sec:intro-ext}
At the heart of the proof of Theorem \ref{thm:main thm} lies an extension result which is concerned with the second obstacle mentioned above. The classical Sch\"onflies Theorem states that any embedding of a line segment into $\R^2$ can be extended to a homeomorphism of $\R^2$. The quasiconformal version of this result comes from a combination of Theorem \ref{thm:Rick} and a celebrated extension result of Beurling and Ahlfors \cite{BA56}. The bi-Lipschitz version of the Sch\"onflies Theorem is due to Tukia \cite{Tukia}. 


The extension problem for trees is more complicated. For example, let $T_1$ be the interval $[-1,1]$ together with vertical segments of length $2^{-n}$ attached at the points $\pm 2^{-n}$, all lying in the upper half-plane, and let $T_2$ be obtained by alternating the vertical segments lying in the upper and lower half-planes; see Figure \ref{fig:not globally homeo trees}. Trees $T_1$ and $T_2$ are homeomorphic (even bi-Lipschitz homeomorphic) but not ambiently homeomorphic as every sufficiently small neighborhood of the origin in $T_1$ contains a topological half-ball, whereas no sufficiently small neighborhood of the origin in $T_2$ has this property. 

\begin{figure}[H]
\centering
\includegraphics[width=1\linewidth]{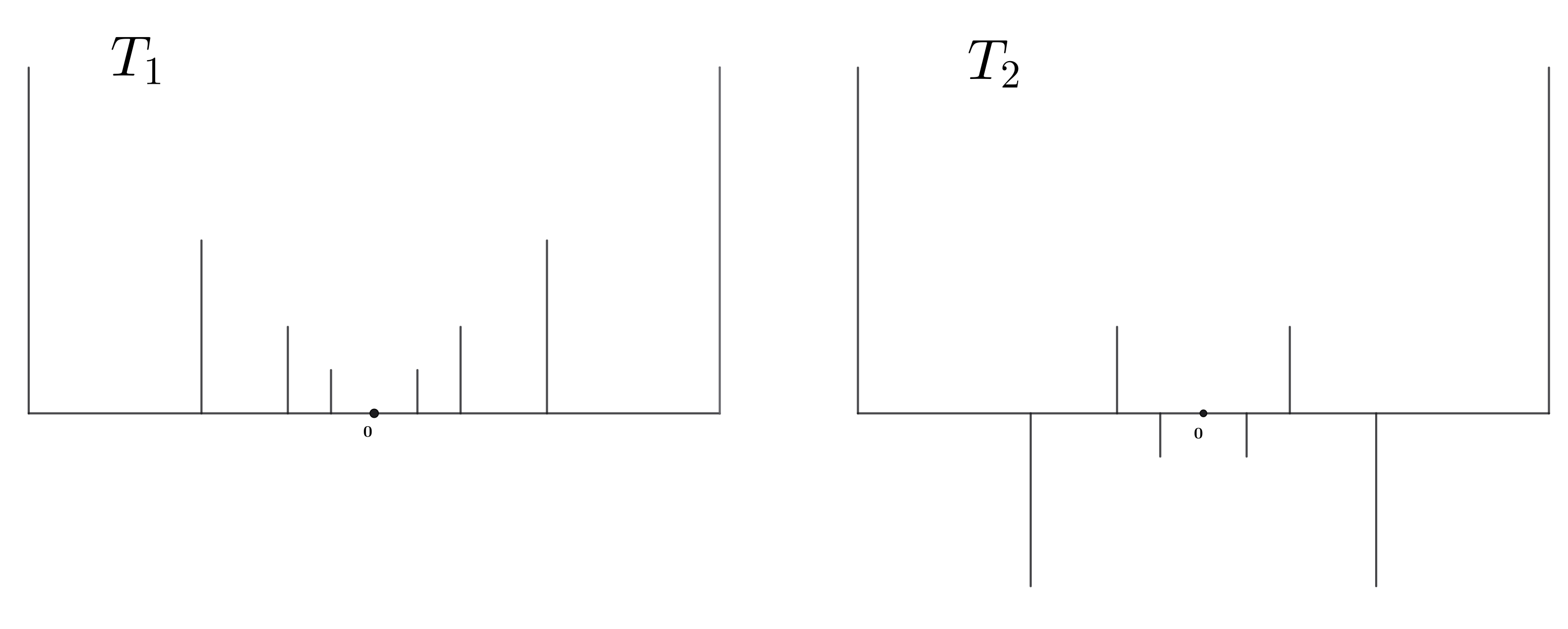}
\caption{The trees $T_1,T_2$ are not globally homeomorphic.}
\label{fig:not globally homeo trees}
\end{figure}

An extension theorem of Adkisson and MacLane \cite{ADMaclane} gives the precise conditions under which a homeomorphism between two planar trees extends to a homeomorphism of $\R^2$. To state their result, we need to first discuss a notion of tree orientability. Avoiding a rigorous definition for now, we only mention that given a planar tree $T$ and a branch point $p\in T$, there is a reasonable way to define counterclockwise enumeration for its branches at $p$. A homeomorphism $f:T\to T'$ between planar trees is \emph{branch-orientation-preserving} if, at every branch point, it preserves the counterclockwise cyclic order of the incident branches; see Definition \ref{defn:BOP} for the rigorous definition. Using this notion, the main theorem in \cite{ADMaclane} states that branch-orientation-preserving homeomorphisms are exactly those homeomorphisms that admit orientation preserving homeomorphic extensions on $\R^2$. Our second main theorem gives a geometric version of this theorem.

\begin{thm}\label{thm:Schoenflies-trees}
Let $T$ be a planar quasiconformal tree and let $f: T \to \R^2$ be a branch-orientation-preserving embedding. 
\begin{enumerate}
\item If $f$ is quasisymmetric, then there exists a quasiconformal orientation preserving extension $F:\R^2\to\R^2$.
\item If $f$ is bi-Lipschitz, then there exists a bi-Lipschitz orientation preserving extension $F:\R^2\to\R^2$.
\end{enumerate}
\end{thm}

Theorem \ref{thm:Schoenflies-trees} is quantitative in the sense that the quasiconformality (resp. bi-Lipschitz) constant of $F$ depends only on the bounded turning constant of $T$, and on the distortion function (resp. bi-Lipschitz constant) of $f$; see Theorem \ref{thm: intro 2 quant.}

It is worth noting that Theorem \ref{thm:Schoenflies-trees} is purely planar and cannot be generalized to higher dimensions for topological reasons. In particular, the Fox-Artin arc \cite{FA} in $\R^3$ is bounded turning (in fact, quasiconvex) but also \emph{wild}, i.e., there exists no homeomorphism of $\R^3$ (let alone quasiconformal homeomorphism) that maps it to a line segment.

\subsection{The Vicsek tree and UQR dynamics}\label{sec:intro-UQR}
Theorem \ref{thm:Schoenflies-trees} has some interesting applications to complex dynamics which may be of independent interest. Given a map $f:\R^2 \to \R^2$ and $n\in\N$, we denote by $f^n = f\circ\cdots\circ f$ the $n$-th iterate of $f$.

A central problem in complex dynamics is the classification of Julia sets of rational maps of $\mathbb{S}^2$. Given a rational map $f$ of $\mathbb{S}^2$, the \emph{Fatou set} of $f$ is defined as
\begin{align*}
F(f) = \{x\in \mathbb{S}^2 : \text{there exists }&\text{neighborhood $x\in U$}\\ 
&\text{such that $(f^{n}|U)$ is a normal family}\}
\end{align*}
and the \emph{Julia set} $J(f)$ is defined as the complement of the Fatou set. 

Uniformly quasiregular (UQR) maps can be thought as the generalizations of rational maps (as quasiconformality is the generalization of conformality). By Stoilow's Factorization Theorem \cite[\textsection 2.2]{MaPe}, one way of defining UQR maps on $\mathbb{S}^2$ are as the quasiconformal conjugates of rational maps. For an introduction to UQR dynamics see \cite{Berg} and \cite[Ch. 21]{IM}. 

By Theorem \ref{thm:Rick}, every bounded turning arc is a Julia set of a UQR map of $\mathbb{S}^2$. The situation for metric trees is far less understood. The Julia set of $z^2+i$ is a quasiconformal tree and, more generally, if the Julia set of a semihyperbolic polynomial is a tree, then it has to be a quasiconformal tree \cite{CJY}. On the other hand, it is not very difficult to show that a quasiconformal tree with exactly one (or finitely many) branch points can not be the Julia set of a UQR map. It is unknown which quasiconformal trees can be realized as Julia sets of rational or UQR maps.

As an application of Theorem \ref{thm:Schoenflies-trees}, we show that the \emph{Vicsek tree} $\mathbb{V}$, one of the most well-known quasiconformal trees, is a Julia set of a UQR map. The Vicsek tree is a 4-valent self-similar semi-geodesic tree of great importance in analysis, probability, and physics; see for instance \cite{Vicsek,Metz,HamblyMetz,Zhou,CSW,BaudoinChen1,BaudoinChen2}. Interestingly, the Vicsek tree is not the Julia set of any holomorphic function (see Lemma \ref{lem:VicseknotJulia}), but it is the Julia set of a UQR map.

\begin{figure}[ht]
    \centering
    \begin{minipage}{0.3\textwidth}
        \centering
        \includegraphics[width=0.9\textwidth]{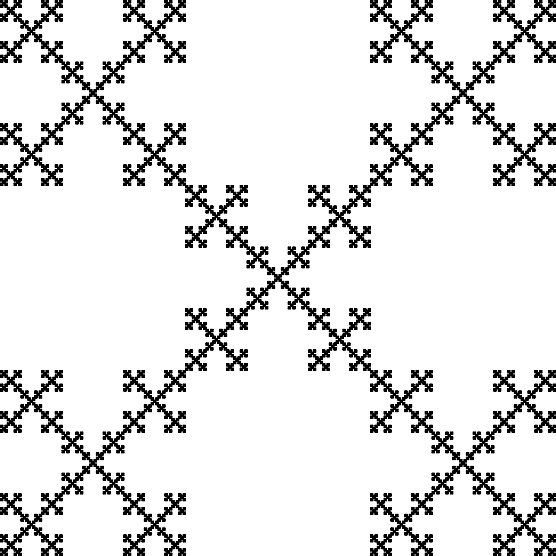} 
    \end{minipage}\hfill
    \begin{minipage}{0.7\textwidth}
        \centering
        \includegraphics[width=0.9\textwidth]{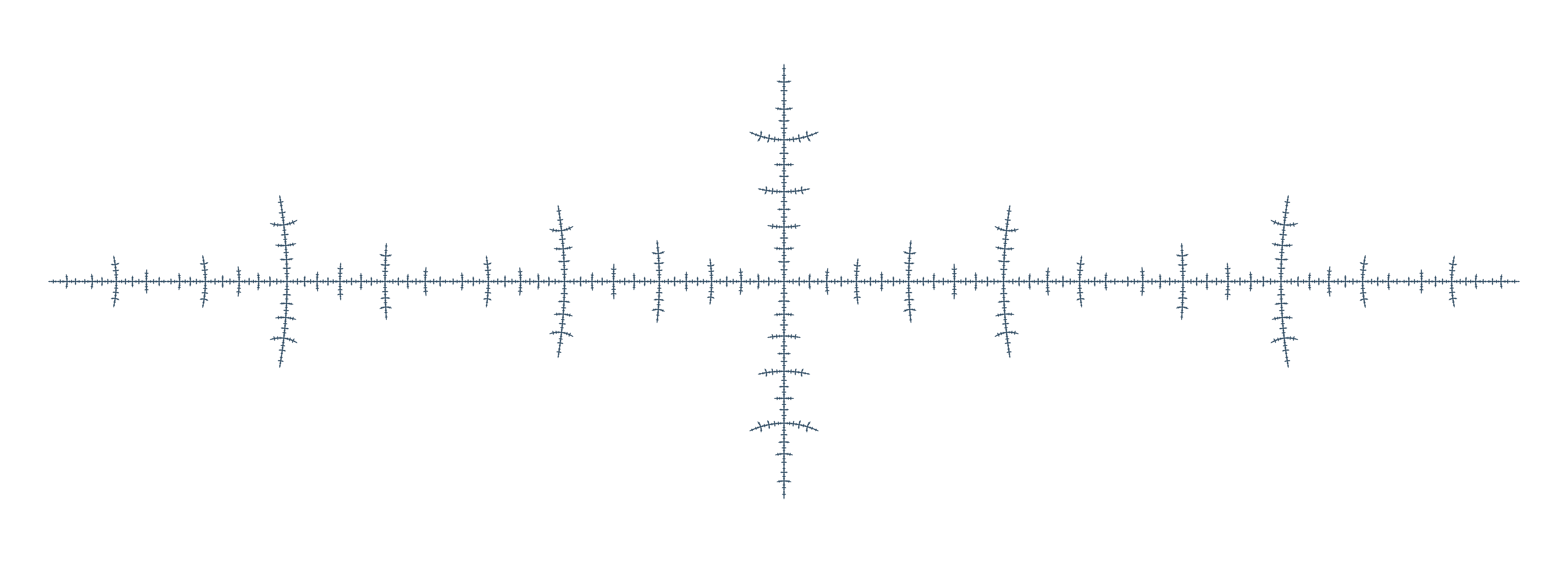} 
    \end{minipage}
    \caption{The Vicsek tree (left) and the Julia set of a quadratic polynomial (right) that is globally quasiconformally homeomorphic to the Vicsek tree.}\label{fig:intro2}
\end{figure}

\begin{thm}\label{thm:Vicsek}
There exists a quadratic UQR map $f$ of $\mathbb{S}^2$ whose Julia set is $\mathbb{V}$.
\end{thm}

Here, by quadratic UQR map $f$, we mean that $f = F^{-1}\circ P \circ F$ for some quasiconformal map $F$ of $\mathbb{S}^2$ and some quadratic polynomial $P$. In fact we show in Proposition \ref{prop: J(f) is QC with T^4} that there exist infinitely many distinct $c\in (-2,-1)$ such that $P$ can be taken to be $z^2+c$. The proof of Theorem \ref{thm:Vicsek} relies on a quasiconformal uniformization of $\V$ (see Proposition \ref{prop: uniformization of T^4}) and the construction of a quadratic polynomial whose Julia set is quasiconformally equivalent to $\mathbb{V}$ (see Proposition \ref{prop: J(f) is QC with T^4}).

\subsection{Overview of the paper}

Section \ref{sec:prelim} is devoted to preliminary definitions and background results and Section \ref{sec: metric trees} gives useful results on trees and quasiconformal trees.

The proof of Theorem \ref{thm:Schoenflies-trees} relies on Carleson’s decomposition technique, introduced in his work on the corona problem \cite{Car62,Car67}, which has been used in several extension results \cite{Ca,Tukia, TukVais}. Specifically, using MacManus' square thickenings \cite{Mac99} we create in Section \ref{sec: Decomposition of R^2 minus T} a \emph{Whitney-like decomposition} of $\R^2\setminus T$. This decomposition consists of countably many mutually disjoint (except their boundaries) simply connected polygons, which are finitely many up to similarities. Each of these polygons has diameter comparable to its distance from $T_1$. In a similar fashion, we construct a Whitney-type decomposition for $\R^2\setminus f(T)$. This is the most technical part of the proof. 

Then, using the fact that $f$ is branch-orientation-preserving, we show in \textsection\ref{sec:compatibility of qc trees} that the two decompositions are combinatorially equivalent, that is, there exists a homeomorphism of the plane that maps each element of one decomposition onto an element of the other. Having this equivalence, in Section \ref{sec: proof of main theorem} we define a map from the skeletons of the decomposition of $\R^2\setminus T$ to the skeleton of the decomposition of $\R^2\setminus f(T)$, extend to the polygons, and then extend to $T$, and then prove that the map is quasiconformal. The finiteness of polygons up to similarities plays the most crucial role.

The proof of Theorem \ref{thm:main thm} involves building a short catalog of planar quasiconvex trees of Hausdorff dimensions arbitrarily close to 1 that uniformizes all planar quasiconformal trees with uniform branch separation. For simplicity, we describe the construction for trees of valence at most 4. Start with a rhombus as in the left of Figure \ref{fig:intropic}. Replace the rhombus by a combination of at most 4 rhombuses as in the right of Figure \ref{fig:intropic}. Each of the new rhombuses is replaced by a combination of at most 4 rhombuses as in the right of Figure \ref{fig:intropic}, and so on so forth, until we obtain a ``quasi-self-similar'' fractal. We make sure that at every step of this process, the union of rhombuses is connected. Let $\mathfrak{T}_4$ be the collection of all such fractal sets produced through this process. Arranging the sizes of the small rhombuses to be very small, we can make the Hausdorff dimension of the final fractals as close to 1 as we want. Similarly, one builds $\mathfrak{T}_m$ for any $m\geq 2$.

\begin{figure}[h]
\centering
\includegraphics[width=\textwidth]{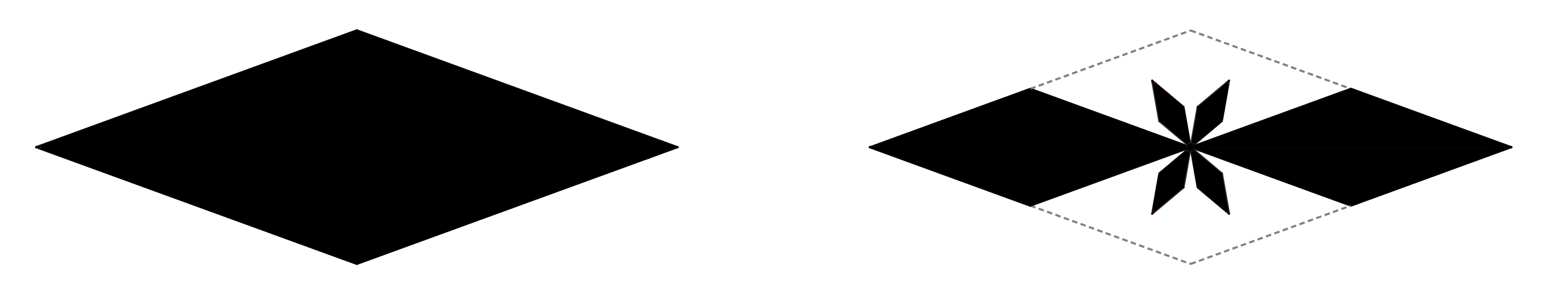}
\caption{}
\label{fig:intropic}
\end{figure}

We show in Proposition \ref{prop: enumeration of T=enumeration of Tmeta} that each element in $\mathfrak{T}_m$ is a quasiconvex tree of uniform branch separation. Conversely, we also show that if $T$ is a planar quasiconformal tree with uniform branch separation and valence at most $m$, then there exists a semi-geodesic $T'\in \mathfrak{T}_m$ and a quasisymmetric branch-orientation-preserving homeomorphism $f:T \to T'$. The rigorous construction of $\mathfrak{T}_m$ and the proof of Proposition \ref{prop: enumeration of T=enumeration of Tmeta} are given in Section \ref{sec: model space}. Then applying Theorem \ref{thm:Schoenflies-trees}, one obtains the desired global map $f$.

In Section \ref{sec:Vicsekunif} we classify all planar trees that are quasiconformally equivalent to the Vicsek tree $\V$. Finally, in Section \ref{sec:Julia} we find a value $c\in\R$ such that the Julia set of the polynomial $z^2+c$ is quasiconformally equivalent to $\V$. These two sections together give the proof of Theorem \ref{thm:Vicsek}.

\subsection*{AI Declaration} No AI has been used during the work and preparation of this manuscript other than Figure \ref{fig:intro2} and Figure \ref{fig:iterations of T_n} which were done using GPT-6 Astra.

\subsection*{Acknowledgments} The authors would like to thank Alastair Fletcher, Dimitrios Ntalampekos, and Efstathios-Konstantinos Chrontsios-Garitsis for helpful discussions and valuable insights related to this project. Lemma \ref{lem:VicseknotJulia} and its proof were suggested by Alastair Fletcher.

\section{Preliminaries}\label{sec:prelim}
We write $\N_0=\N\cup \{0\}.$
Given two non-negative quantities $A$ and $B$, we write $A\lesssim_C B$ if there is a constant $C'>0$ depending on $C$ such that $A\leq C' B$. If $A\lesssim_C B$ and $A\gtrsim_C B$ we write $A\simeq_C B$. We also denote by $\overline{\R^n}=\R^n\cup\{\infty\}$ the one point compactification of $\R^n$. We call $\Omega\subset \R^n$ \textit{domain} if it is open and connected.

A continuous embedding $\g:[a,b]\to \R^2$ is called \textit{curve}. If $\g$ is a homeomorphic embedding, then $\g$ is called \textit{arc}. By a slight abuse of notation we conflate $\g$ with its image $\g([a,b])$. If $x,y\in \g$, then we denote the unique arc between them along $\g$ as $\g[x,y]$. If we want to exclude one or two enpdoints, then we write $\g[x,y),\g(x,y],\g(x,y)$. Last, we denote by $[x,y]\subset \R^2$ the line segment that connects $x$ to $y$.

A simple closed curve $\G\subset \R^2$ is called \textit{Jordan curve}. The domain enclosed by it is called \textit{Jordan domain}. Given a finite set $V\subset \G$ with $\card(V)\geq 3$ and given $x,y\in V$ distinct we denote by $\G[x,y]$ the arc in $\G$ that connects $x$ to $y$ and $\G[x,y]\cap V=\{x,y\}$. Accordingly we say that $x,y$ are \textit{adjacent in $\G$}. If we want to exclude one or two endpoints, then we write $\G[x,y),\G(x,y],\G(x,y)$.

We show the following elementary lemma for future reference. 
\begin{lem}\label{lem:Jordan domain prerequisite}
Let $O\subset \R^2$ be a Jordan domain with boundary $\G$. Let $x\in \G$ and $r>0$ such that $\partial B(x,r)\cap O\neq\emptyset$. Let $p,q\in \partial B(x,r)\cap \G$ such that $x\in \G(p,q)\subset B(x,r)$. Denote by $W_1,W_2$ the connected components of $B(x,r)\setminus\G[p,q]$. There exist arcs $\g_1,\g_2$ with one common endpoint $p'\in \G[p,x)$, the other common endpoint $q'\in \G[q,x)$ such that $\g_1\setminus\{p',q'\}\subset W_1\cap O$ and $\g_2\setminus\{p',q'\}\subset W_2\cap (\R^2\setminus \overline{O})$.
\end{lem}

\begin{proof}
Let $p,q,B(x,r)$ as in the statement. Let $h:\R^2\to \R^2$ be a homeomorphism that sends $\G[p,q]$ to a line segment $l$. Let $r'>0$ such that $B(h(x),r')\subset h(B(x,r))$. Then $B(h(x),r')\setminus l$ has two connected components that lie inside $h(B(x,r))$ and do not intersect $h(\G)$ other than $l$. By taking inverse images and by the fact that $h$ is a homeomorphism the result follows. 
\end{proof}

\subsection{Bounded turning spaces and John domains} A \textit{continuum} $K$ is a non-empty compact connected set. A non-degenerate continuum is a continuum that is not a singleton. 

\begin{definition}\label{def:bounded turning}
A set $X\subset \R^n$ is of $C$-\textit{bounded turning}, for some $C\geq 1$, if for any $x,y\in X$ there exists a continuum $K\subset X$ containing $x,y$ such that $\diam K\leq C|x-y|$.
\end{definition}

\begin{definition}\label{def:John domain}
A domain $\Omega\subset \R^n$ is $C$-\textit{John} if there exists a constant $C\geq 1$ such that for all $x,y\in\Omega$ there exists an arc $\g\subset \Omega$ connecting $x$ to $y$ and for all $z\in\g[x,y]$
\begin{align}\label{eq:John}
\dist(z,\partial\Omega)\geq C^{-1}\min\{\ell(\g[z,x]),\ell(\g[z,y])\}.
\end{align}
\end{definition}
We also say that a curve or an arc is $C$-\textit{John} if it satisfies \eqref{eq:John} for some uniform constant $C$. 

The following theorem of N\"{a}kki and V\"{a}is\"{a}l\"{a} connects the notions of bounded turning sets and John domains in the plane.

\begin{lem}[{\cite[Theorem 4.5]{NV91}}]\label{lem: Nakki vaisalla}
Let $X\subset \overline{\R^2}$ be a non-degenerate continuum. Then, $X$ is $C$-bounded turning if and only if $\overline{\R^2}\setminus X$ is a $C$-John domain.
\end{lem}

\subsection{Words}\label{sec:words} For each $n\in \N$, denote by $A_n=\{1,\dots,n\}$. We call $A_n$ \textit{alphabet of order $n$}. Denote by $\e$ the empty word and by $|w|$ the length of a word, i.e., the number of letters with the convention that $|\e|=0$. Denote also by $w^{\uparrow}$ the \textit{predecessor} of $w$, i.e., if $w=i_1\cdots i_n$, then $w^{\uparrow}=i_1\cdots i_{n-1}$. If $w\in A_n$, then we set $w^{\uparrow}=\e$. Let $A_n^0=\{\e\}$ and for each $m\geq 1$, let $A_n^m$ be the set of all words made from the alphabet $A_n$ of length exactly $m$. Define the set of finite words 
$$A_n^*=\bigcup_{m=0}^{\infty}A_n^m.$$
Denote also by $A^{\N}_n$ the set of infinite words formed by the alphabet $A_n$. Given two finite words $w = i_1\cdots i_k$ and $v =j_1\cdots j_m$ in $A_n^{*}$, we denote by $wu=i_1\cdots i_kj_1\cdots j_m$ their concatenation. Given $w\in A^{\N}_n$ denote by $w\{m\}$ the unique word $u \in A^m_n$ such that $w=uw'$
for some $w'\in A^{\N}_n$. Similarly, if $w\in A_n^*$, $w\{m\}$ denotes the initial subword of
$w$ of length $m$, and we set $w\{m\} = w$ if $m\geq |w|$. Given $i\in A_n$ and $k\in\N$, we denote by $i^{(k)}$ the word $i\cdots i\in A_n^k$ and by $i^{(\infty)}$ the infinite word $ii\cdots \in A^{\N}$. 

For each $n,m\in\N$ let $\emptyset\neq W\subset A_n^m$. Let $w,u\in W$ such that $w=vi_1\cdots i_{m-|v|}$, $u=vj_1\cdots j_{m-|v|}$ where $v\in A_n^*$, $|i_k|=|j_k|$ for $k\in \{1,\dots,m-|v|\}$ and $i_1\neq j_1$. If $i_1<j_1$ we say that $w$ is \textit{to the right} of $u$ (or $u$ is \textit{to the left} of $w$). Accordingly we say that $W$ is oriented according to the \textit{cyclic lexicographic order}. 
If $W=\{w_1,\dots,w_M\}$, then we make the convention that $w_1$ is to the right of $w_M$, and $w_M$ is to the left of $w_1$. Last, if $w_i=w_M$, then we make the convention that $w_{i+1}=w_1$ and if $w_i=w_1$, then we make the convention that $w_{i-1}=w_M$.

\subsection{Maps}
A homeomorphism $f:(X,d_X) \to (Y,d_Y)$ between metric spaces is said to be \emph{quasisymmetric} (or $\eta$-quasisymmetric) if there exists a homeomorphism $\eta \colon [0,\infty) \to [0,\infty)$ such that for all $x,a,b \in X$ with $x\neq b$ 
\[ \frac{d_Y(f(x),f(a))}{d_Y(f(x),f(b))} \leq \eta \left ( \frac{d_X(x,a)}{d_X(x,b)} \right ). \]

The composition of two quasisymmetric maps and the inverse of a quasisymmetric map are quasisymmetric. Two spaces $X,Y$ are quasisymmetrically equivalent if there exists a quasisymmetric map between them. 

Quasisymmetric maps were introduced by Tukia and V\"ais\"al\"a \cite{TV80} as the analogue of quasiconformal maps in the abstract metric setting. For more background see \cite[Chapters 10-12]{Heinonen}. It is well know that the properties ``doubling'' and ``bounded turning'' are preserved under quasisymmetric maps quantitatively. Bonk and Meyer showed that the property of ``uniformly branch separation'' is also preserved \cite[Lemma 4.3]{BM22}. 

A homeomorphism $f:(X,d_X)\to(Y,d_Y)$ is said to be \emph{weakly quasisymmetric} (or weakly $H$-quasisymmetric) if there exists $H\geq 1$ such that for all $x,a,b\in X$,
\[ d_X(x,a)\leq d_X(x,b)\quad\text{implies}\quad d_Y(f(x),f(a)) \leq H\,d_Y(f(x),f(b)).\]
It follows from \cite[Theorem 10.19]{Heinonen} that if $X,Y$ are connected doubling spaces, then quasisymmetry is equivalent to weak quasisymmetry quantitatively.

A homeomorphism $f:(X,d_X)\to (Y,d_Y)$ is called \emph{(metrically) $H$-quasiconformal} if there exists $H\geq 1$ such that for all $x\in X$,
\[H_f(x):=\limsup_{r\to 0}\frac{\sup_{d_X(x,y)\leq r} d_Y(f(x),f(y))}{\inf_{d_X(x,y)\geq r} d_Y(f(x),f(y))}\leq H.\]

It is not hard to see that quasisymmetric maps are quasiconformal. On the other hand, if $\Omega\subset \R^n$ is a domain and $f:\Omega\to \R^n$ is quasiconformal, then it is quasisymmetric on any ball $B(x,\frac{1}{2}\dist(x,\partial\Omega))$ \cite[Theorem 11.14]{Heinonen}. In particular, if $f:\R^n \to \R^n$ is quasiconformal, then it is quasisymmetric.

The following result of Jones and Smirnov gives a quasiconformal removability criterion
in terms of John domains.

\begin{lem}[{\cite[Corollary 1]{JonesPeter}}]\label{lem:qc removable}
Let $X\subset \R^2$ be a closed set and let $F:\R^2\to\R^2$ be a homeomorphism. If $F|\R^2\setminus X$ is quasiconformal and $\R^2\setminus X$ is a John domain, then $F$ is quasiconformal.
\end{lem}

A map $f:(X,d_X)\to (Y,d_Y)$ is an $(L,\lambda)$-\textit{quasi-similarity}, for some $\lambda>0$, $L\geq 1$, if for all $x_1,x_2\in X$ 
$$L^{-1}\lambda d_X(x_1,x_2)\leq d_Y(f(x_1),f(x_2))\leq L\lambda d_X(x_1,x_2).$$
An $(L,1)$-quasisimilarity is \emph{$L$-bi-Lipschitz}, a $(1,\lambda)$-quasisimilarity is a \emph{similarity} with scaling factor $\lambda$, and a $(1,1)$-quasisimilarity is an \emph{isometry}. 

The next remark can be easily verified from definitions.

\begin{rem}
If $f:(X,d_X)\to (Y,d_Y)$ is an $(L,\lambda)$-quasi-similarity, then it is $L^2$-quasiconformal and $K$-bi-Lipschitz with $K=L\max\{\lambda ,\lambda^{-1}\}$. Moreover, if $f$ is $L$-bi-Lipschitz, then $f$ is $\eta(L^2t)$-quasisymmetric.
\end{rem}

A map $f:(X,d_X)\to(Y,d_Y)$ is an $L$-\emph{bounded length distortion} map, (abbv. $L$-\emph{BLD}), if
\[ L^{-1}\ell(\gamma) \leq \ell(f\circ\gamma) \leq L\ell(\gamma)\]
for every rectifiable path $\gamma$ in $X$.

Every $L$-bi-Lipschitz map is $L$-BLD. Conversely, it is not hard to show that if $f:\mathbb R^2\to\mathbb R^2$ is a homeomorphism and is $L$-BLD, then $f$ is $L$-bi-Lipschitz.

The following extension theorem of Tukia allows us to extend bi-Lipschitz maps between boundaries of bi-Lipschitz disks to the whole plane.



\begin{lem}\label{lem:Tukia quasi-similarity}
Let $\Gamma,\Gamma'\subset\mathbb R^2$ be Jordan curves such that $\Gamma$ is $(L_1,\lambda_1)$-quasi-similar to $\mathbb{S}^1$. If $f:\Gamma\to\Gamma'$ is an $(L_2,\lambda_2)$ quasi-similarity, then $f$ extends to an $(L,\lambda_2)$-quasi-similarity $F:\mathbb R^2\to\mathbb R^2$, where $L$ depends only on $L_1,L_2$.
\end{lem}

\begin{proof}
Let $h:\mathbb{S}^1\to\Gamma$ be an $(L_1,\lambda_1)$-quasi-similarity. Then $h_1=\lambda_1^{-1}h$ is $L_1$-bi-Lipschitz and $h_2=(\lambda_1\lambda_2)^{-1}f\circ h$ is $L_1L_2$-bi-Lipschitz. By Tukia's extension theorem \cite[Theorem A]{Tukia}, the maps $h_1,h_2$ extend to global bi-Lipschitz homeomorphisms $H_1,H_2$ of $\mathbb R^2$, with constants depending only on $L_1,L_2$. The map
\[F(z)=\lambda_1\lambda_2 H_2\left(H_1^{-1}\left(\frac{z}{\lambda_1}\right)\right) \]
extends $f$ and is an $(L,\lambda_2)$-quasi-similarity, where $L$ depends only on $L_1,L_2$.
\end{proof}

\subsection{Square thickening} 
Given $\d>0$,  denote by
\[\Sigma_\delta=\{ [m\delta,(m+1)\delta]\times [k\delta,(k+1)\delta]:(m,k)\in\mathbb{Z}^2\}\]
the collection of all closed dyadic squares of side length $\delta$ and define the \emph{$\delta$-square grid} by
\[ \mathcal G_\delta = \bigcup_{Q\in \Sigma_{\d}} \partial Q. \]
Following MacManus \cite{Mac99}, given a bounded set $S\subset \R^2$, define
\[S^\delta=\bigcup_{\substack{Q\in \Sigma_\d \\ Q\cap S\neq\emptyset}} Q.\]
and define the \emph{square thickening} of $S$ by
\[\mathscr{T}_\delta(S)=(S^{4\delta})^{\delta} .\]

In the following lemma, we gather some properties about square thickening.  

\begin{lem}\label{lem:properties of thickening}
Let $S\subset\R^2$ be a bounded set and let $\d>0$.
\begin{enumerate}[(i)]
\item The set $\partial\mathscr{T}_{\delta}(S)$ consists of finitely many mutually disjoint Jordan curves, each of which is a subset of $\mathcal{G}_\d$. Moreover, if $S$ is connected, then there exists a Jordan domain $\Omega$ such that $\partial\Omega\subset \partial\mathscr{T}_{\delta}(S)$ and $S \subset \Omega$.
\item For all $y\in \partial\mathscr{T}_{\delta}(S)$, $\d<\dist(y,S)<8\d$. In particular, 
$$\diam\mathscr{T}_{\d}(S)\leq \diam S+16\delta.$$
\item If $S'\subset S$, then $\mathscr{T}_{\d}(S')\subset\mathscr{T}_{\d}(S)$.
\item If $\{A_i\}_{i\in I}$ are sets in $\R^2$ with $S=\bigcup_{i\in I} A_i$, then $\mathscr{T}_{\d}(S)=\bigcup_{i\in I}\mathscr{T}_{\d}(A_i)$.
\item If $\d'\in (0,\d]$, then $\mathscr{T}_{\d'}(S)\subset \mathscr{T}_{\delta}(S)$.
\item For each $z\in\mathscr{T}_{\d}(S)$ there exists $w\in S$ such that $|z-w| < 9\d$ and $[z,w]\subset \mathscr{T}_{\d}(S)$.
    \end{enumerate}
\end{lem}

\begin{proof}
Claims (i)-(v) follow directly from the definitions and from \cite[Lemma 2.1]{Mac99}. To show (vi), fix $z\in\mathscr{T}_{\d}(S)$. Then $z$ belongs to a square $Q_1\in \Sigma_{\d}$ that intersects a square $Q_2\in \Sigma_{4\d}$ that intersects $S$. Let $w\in Q_{2}\cap S$. Since $Q_1\cap Q_{2}\neq\emptyset$ the squares $Q_1,Q_{2}$ share at least a side or a corner. Let 
\[Q' = Q_{2}\cup \bigcup_{\substack{Q \subset \Sigma_\delta\\ Q \cap Q_{2} \neq \emptyset}} Q.\]
Note that $Q'$ is a square, $Q'\subset \mathscr{T}_{\d}(S)$, and $z,w\in Q'$. We conclude that $[z,w]\subset \mathscr{T}_{\d}(S)$ and that $|z-w|\leq \diam Q'<9\d$.
\end{proof}

\section{Metric trees}\label{sec: metric trees} 

We say that a metric space $T$ is a \textit{(metric) tree} if it is  compact, connected, and locally connected, with at least two distinct points, and such that for any $x,y\in T$  there is a unique arc in $T$ with endpoints $x, y$. We denote this unique arc by $T[x,y]$. As with arcs, we denote $T[x,y) = T[x,y]\setminus\{y\}$ and analogously we define $T(x,y]$ and $T(x,y)$.

We say that a point $p\in T$ has \textit{valence} equal to $m\in \N\cup \{\infty\}$ and write $\val(p,T)=m$ if the complement $T\setminus\{p\}$ has $m$ connected components, called \textit{branches of $T$ at $p$}. We denote the branches of $p$ in $T$ by $B_T^i(p)$ for $i\in\{1,\dots,m\}$. If $\val(p,T)=1$ we say that $p$ is a \textit{leaf} of $T$, if $\val(p,T)=2$ we say that $p$ is a \textit{double point} of $T$, and if $\val(p,T)\geq 3$ we say $p$ is a \textit{branch point} of $T$. We denote the set of branch points of a tree with $\mathcal{B}(T)$ and the set of leaves of a tree with $\mathcal{L}(T)$. We define the \textit{valence} of a tree $T$ by
\[ \val(T):=\sup \{\val(p,T): \, p \in T\}.\] 
If $\mathcal{B}(T)=\emptyset$, then $T$ is a metric arc. If $T$ has at least one branch point and there exists $m\in\{3,4,\dots\}$ such that $\val(p,T)=m$ for all branch points $p$ of $T$, then we say that $T$ is \textit{$m$-valent}. 

A subset $S$ of a tree $T$ is called \textit{subtree} of $T$ if $S$ equipped with the restricted metric of $T$ is also a tree. By \cite[Lemma 3.3]{BT21}, $S$ is a subtree of $T$ if and only if $S$ contains at least 2 points and it is closed and connected.

We also gather some properties about metric trees from \cite{BM20}.
\begin{lem}\label{lem: metric trees properties of BM}
Let $(T,d)$ be a metric tree.
\begin{enumerate}[(i)]
\item Let $V\subset T$ be a finite set. Two points $x, y \in T\setminus V$ lie in the same component of $T \setminus V$ if and only $T[x,y]\cap V=\emptyset$.
\item Let  $b, x_1, x_2, x_3\in T$ with $b \neq x_1, x_2, x_3$ and suppose that the sets $T[x_1, b)$, $T[x_2, b)$, $T[x_3, b)$ are pairwise disjoint. Then the points $x_1, x_2, x_3$ lie in different components of $T \setminus\{b\}$, and $b$ is a branch point of $T$.
    \end{enumerate}
\end{lem}
\begin{proof}
    Claim (i) follows from \cite[Lemma 2.3(i)]{BM20}. Claim (ii) follows from \cite[Lemma 2.4]{BM20}.
\end{proof}

\subsection{Branch-orientation-preserving maps}

The next lemma gives a way to define counterclockwise enumeration for the branches of a tree at some branch point $b\in T$.

\begin{lem}\label{lem: maximal arcs}
Let $T\subset \R^2$ be a metric tree, let $b\in T$ be a branch point of valence $k\geq 3$, and let $B_1,\dots,B_k$ be the components of $T\setminus\{b\}$. Let $D,D'$ be Jordan domains that contain $b$ and intersect $B_i$ for all $i\in \{1,\dots,k\}$.
\begin{enumerate}[(i)]
\item For each $i\in \{1,\dots,k\}$ there exists $p_i\in B_i\cap \partial D$ and $q_i\in B_i\cap \partial D'$ such that $T[b,p_i)\subset D$ and $T[b,q_i)\subset D'$.
\item If $p_{i_1},\dots,p_{i_k}$ and $q_{j_1},\dots,q_{j_k}$ are the counterclockwise enumerations along $\partial D$ and $\partial D'$, respectively, and if $i_1=j_1$, then $i_l=j_l$ for all $l\in \{2,\dots,k\}$.
\end{enumerate}
\end{lem}

\begin{proof}
To show (i), fix $r>0$ such that $\overline{B}(b,r)\subset D$. For each $i\in \{1,\dots,k\}$ let $y_i\in \partial D\cap B_i$ and note that  $T[b,y_i]\subset \overline{B_i}$. Let $p_i\in B_i$ be the first time that the arc $T[b,y_i]$ meets $\partial D$ as we travel from $b$ to $y_i$. It is then guaranteed that $T[b,p_i)\subset D$ since there exists a subarc of $T[b,p_i)$ lying entirely in $\overline{B}(b,r)$. The existence of $\{q_i\}$ on $\partial D'$ is similar.

To show (ii), suppose for a contradiction that $i_1=j_1$ and that $j_2=i_{l}$ for some $l\in\{3,\dots,k\}$. We start with a special case and then prove (ii) in full generality.

\emph{Case 1:} Suppose that $\overline{D}\subset D'$. By Schoenflies and Annulus theorem applying an orientation preserving homeomorphism $h$ we may assume that $b=0$, $D = B(0,1)$, and $D'=B(0,2)$. Let $\sigma_i=T[0,p_i]$, $\g_i=T[0,q_i]$. For any $x,y\in \partial B(0,2)$ (resp. $x,y\in \partial B(0,1)$) denote also by $P[x,y]$ (resp. $Q[x,y]$) the counterclockwise arc in $\partial B(0,2)$ (resp. in $\partial B(0,1)$) from $x$ to $y$. Write $B(0,2)=S_1\cup S_2\cup S_3$ where 
\begin{align*}
        \partial S_1&=P[q_{i_1},q_{i_l}]\cup  \g_{i_l} \cup \g_{i_1} \\
        \partial S_2&= P[q_{i_l},q_{i_2}]\cup  \g_{i_2} \cup \g_{i_l} \\
        \partial S_3&=P[q_{i_2},q_{i_1}]\cup  \g_k \cup \g_{i_2}.
\end{align*}
\begin{figure}[H]
\centering
\includegraphics[width=0.7\linewidth]{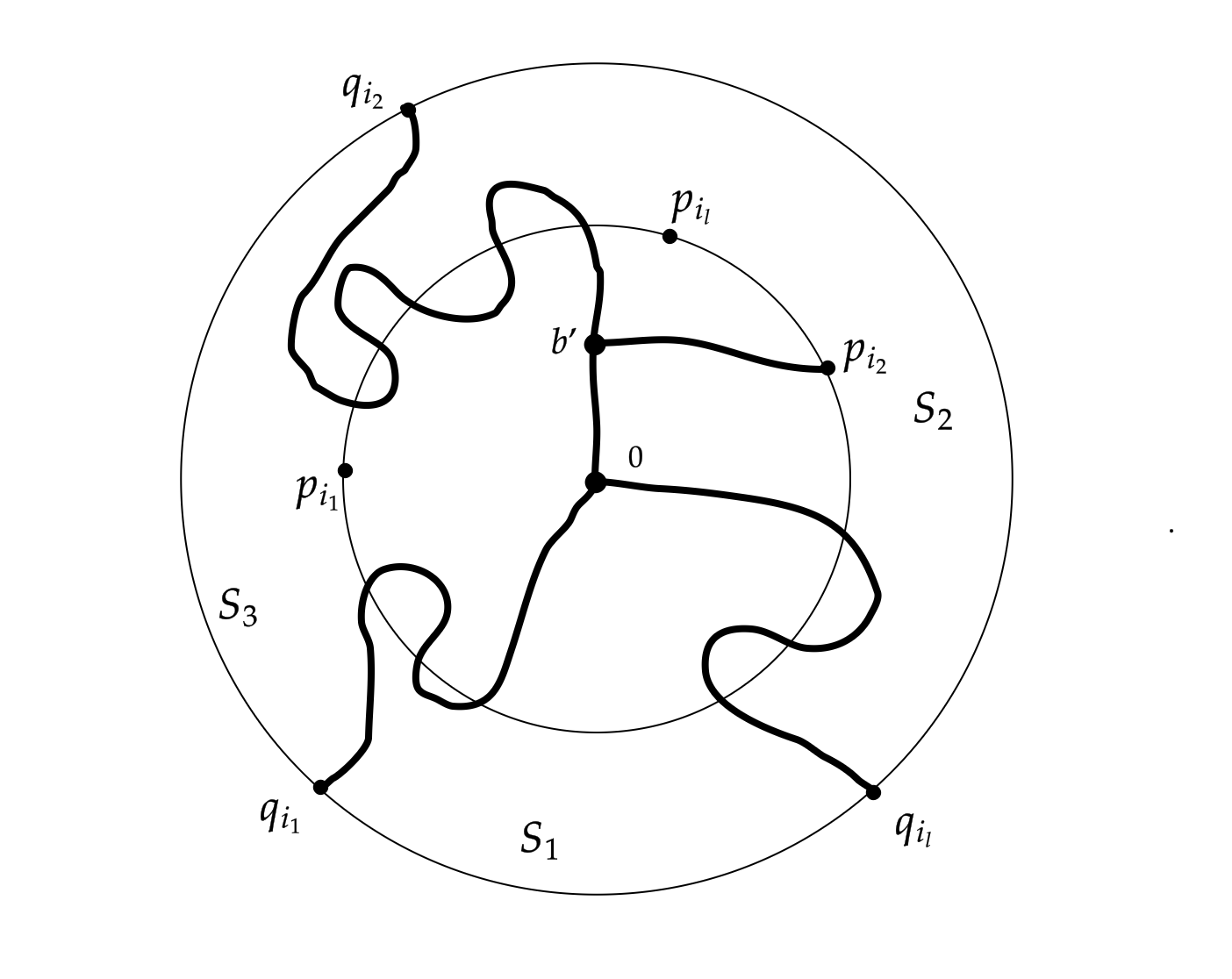}
\label{fig:maximal_arcs 1}
\end{figure}
We claim that $\sigma_{i_1}\subset S_1\cup S_3$. If not, then $\sigma_{i_1}\subset S_2$ due to the fact that $T$ is a tree, $h$ is a homeomorphism, and that $\sigma_{i_3}\subset B(0,2)$. But then $\{0\}= \gamma_{i_1}\cap \sigma_{i_1}$ which implies that $0\in T[p_{i_1},q_{i_1}]$ which is impossible since $p_{i_1},q_{i_1}\in B_{i_1}$ by Lemma \ref{lem: metric trees properties of BM}(i). Similarly, $\sigma_{i_l}\subset S_1\cup S_2$ and $\sigma_{i_2}\subset S_2\cup S_3$. Suppose that $p_{i_1}\in S_3$; if $p_{i_1}\in S_1$ similar arguments work. Since $p_{i_2}\in Q[p_{i_1},p_{i_l}]$ it follows that either $\sigma_{i_l}\cup \sigma_{i_2}\subset S_2$ or that $\sigma_{i_2}\subset S_3$ and $\sigma_{i_l}\subset S_1$. We show the former and the latter follows with similar arguments. Note that $T[0,q_{i_2}]\cap T[0,p_{i_2}]= T[0,b']$ where $b'\in B_{i_2}$ is a branch point from Lemma \ref{lem: metric trees properties of BM}(ii). Since $\sigma_{i_l}\setminus\{p_{i_l}\}\subset B(0,1)$ it follows that 
$$\emptyset\neq\sigma_{i_l}\cap (\g_{i_2}[b',q_{i_2}]\cup\sigma_{i_2}[b',p_{i_2}])\subset B_{i_l}\cap B_{i_2}$$
which is contradiction due to the fact that $T$ is a tree.

We work similarly if $\overline{D'}\subset D$.

\emph{Case 2:} Suppose that $\partial D\cap \partial D'\neq\emptyset$. Let $r>0$ so that $\overline{B}(b,r)\subset D\cap D'$. For each $t\in\{1,\dots,k\}$ let $p_{i_t}'$ (resp. $q_{j_t}'$) be the first time we meet $\partial B(b,r)$ as we travel from $b$ to $p_{i_t}$ (resp. to $q_{j_t}$). By Case 1 we have the counterclockwise enumerations $p_{i_1}',p_{i_2}',\dots,p_{i_l}'$ and $q_{i_1}',q_{j_l}',\dots,q_{j_2}'$ along $\partial B(b,r)$. Let $\g_i=T[b,p_i]$, $\sigma_i=T[b,q_i]$. For any $x,y\in\partial B(b,r)$ denote also by $P[x,y]$ the counterclockwise arc in $\partial B(b,r)$ from $x$ to $y$. 
Write $\overline{B}(b,r)=S_1\cup S_2\cup S_3$ where 
\begin{align*}
        \partial S_1&=P[p_{i_1}',p'_{i_2}]\cup  \g_{i_2} \cup \g_{i_1} \\
        \partial S_2&= P[p'_{i_2},p'_{i_l}]\cup  \g_{i_l} \cup \g_{i_2} \\
        \partial S_3&=P[p'_{i_l},p'_{i_1}]\cup  \g_{i_1} \cup \g_{i_l}.
    \end{align*}

Similar arguments with Case 1 show that $\sigma_{i_1}\subset S_1\cup S_3$, $\sigma_{i_l}\subset S_2\cup S_3$, and $\sigma_{i_2}\subset S_1\cup S_2$. Suppose that $q_{i_1}'\in \partial S_3$; if $q_{i_1}'\in \partial S_1$ similar arguments work. Since $q_{i_2}'\in P[q_{i_l}',q_{i_1}']$ it follows that either $\sigma_{i_l}\cup \sigma_{i_2}\subset S_2$ or that $\sigma_{i_l}\subset S_3$ and $\sigma_{i_2}\subset S_1$. We show the former and the latter follows with similar arguments. Note that $T[b,p_{i_l}]\cap T[b,q_{i_l}]= T[b,b']$ where $b'\in B_{i_l}$ is a branch point from Lemma \ref{lem: metric trees properties of BM}(ii). Let $S_2= R_1\cup R_2$ where 
    \begin{align*}
        &\partial R_1=P[p_{i_2}',q_{i_l}']\cup T[q_{i_l}',b]\cup T[b,p_{i_2}'] \\ &
        \partial R_2=P[q_{i_l}',p_{i_l}']\cup T[p_{i_l}',b']\cup T[b',q_{i_l}']. 
    \end{align*}
    Since $q_{i_2}'\in P[q_{i_l}',q_{i_1}']$, $T(q_{i_2}',b]\subset B(b,r)$ and $T$ is a tree it follows that $T[q_{i_2}',b]\subset R_2$ which is contradiction.
\end{proof} 

We now give the rigorous definition of branch-orientation-preserving maps.

\begin{definition}\label{defn:BOP}
A homeomorphism $f:T\to T'$ between two planar trees is \textit{branch-orientation-preserving} if it satisfies the following property. Let $b\in T$ be a branch point, let $B_1,\dots,B_k$ be the branches of $T$ at $p$, and let $D,D'$ be Jordan domains, so that
\begin{enumerate}
\item $b\in D$ and $f(b)\in D'$,
\item for each $i\in \{1,\dots,k\}$ there exist $x_i \in \partial D\cap B_i$ and $y_i \in \partial D'\cap f(B_i)$ with $T[b,x_i)\subset D$, and $T'[f(b),y_i)\subset D'$.
\end{enumerate}
If $x_{i_1},\dots, x_{i_{k}}$ is a counterclockwise enumeration along $\partial D$, then $y_{i_1},\dots, y_{i_k}$ is a counterclockwise enumeration along $\partial D'$. 
\end{definition}

Note that thanks to Lemma \ref{lem: maximal arcs}, the choice of $D$ and $D'$ in the above definition is not important.

\subsection{Tile decompositions}

We recall the definition of tile decompositions due to Bonk and Meyer \cite{BM22}.

\begin{definition}\label{def:tile decomposition}
Let $(T,d)$ be a metric tree. We say that $T$ admits a $C$-\emph{tile decomposition}
$\{\mathbf{X}^n\}_{n\in\N_0}$ if there exists an increasing sequence of finite sets
$\emptyset=\mathbf{V}^0\subset \mathbf{V}^1\subset \mathbf{V}^2\subset\cdots,$
each consisting of double points of $T$, and a constant $C>17$, such that for every $n\in\N_0$, the collection $$\mathbf{X}^n=\left\{\overline{U}: U \text{ is a connected component of } T\setminus\mathbf{V}^n\right\}$$
is called the collection of \emph{$n$-tiles}, and the following properties hold.
    \begin{enumerate}[(i)] 
    \item Each $n$-tile contains at least three $(n+1)$-tiles and at most $C$ $(n+1)$-tiles.
    \item If $X,Y\in{\bf X}^n$ with $X\cap Y\neq \emptyset$, then $\frac{1}{C}\diam Y\leq\diam X\leq C\diam Y$. 
    \item If $X,Y\in{\bf X}^n$ with $X\cap Y=\emptyset$, then $\dist(X,Y)\geq\frac{1}{C}\diam Y$. 
    \item If $X\in{\bf X}^n$, $Y\in{\bf X}^{n+1}$ with $Y\subset X$, then $16\diam Y\leq\diam X\leq C\diam Y$. 
    \item For each $k\geq C$ we have
$\diam(Y)\leq \frac12\diam(X)$ for all $X\in {\bf X}^n$ and $Y \in {\bf X}^{n+k}$ with $X\cap Y\neq\emptyset$.
\end{enumerate}
\end{definition}

\begin{prop}\label{prop:tile-decomposition}
Every quasiconformal tree admits a $C$-tile decomposition with $C$ depending on the bounded turning constant. Moreover, if $({\bf X}^n)_{n\in\N}$ is its $C$-tile decomposition of $(T,d)$ and $f:(T,d)\to (T',d')$ is $\eta$-quasisymmetric, then the set
$$\{\mathbb{X}^n\}:=\{f(X):X\in{\bf X}^n\}_{n\in\N}$$
is a $C'$-tile decomposition of $(T',d')$, where $C'$ depends only on $C$, the bounded turning constant of $T$, and $\eta$.
\end{prop}

\begin{proof}
Let $(T,d)$ be a quasiconformal tree and let $C_T$ be the bounded turning constant. Consider a new metric $\texttt{dd}$ on $T$ by $\texttt{dd}(x,y)=\diam_d T[x,y]$. By \cite[Lemma 2.5(iv)]{BM20} $(T,d)$ is $C_{T}$-bi-Lipschitz equivalent to $(T,\texttt{dd})$. Hence, the space $(T,\texttt{dd})$ is doubling and $1$-bounded turning. By \cite[Section 5]{BM20} there exists a sequence of finite sets ${\bf V}^1\subset {\bf V}^2\subset\dots$ consisting of double points of $(T,\texttt{dd})$ that are $\d_1^n$-separated for some constant $\d_1>0$ that depends on the doubling constant of $(T,\texttt{dd})$. Hence, for $\d_0<\d_1/C_\t^2$ the points in ${\bf V}^n$ are $\d^n/C_T$-separated in $\R^2$ for every $\d<\d_0$. The above facts and \cite[page 275 (5.3),(5.4)]{BM20} give that for $X,Y\in{\bf X}^n$ disjoint $\d^n/C_T\le\diam X\le3\beta\d^n$ for some constant $\beta$ that depends only on $C_T$ and that $\dist_d (X,Y)\geq \d^n/C_T$. After rescaling we may assume that $C_T>17$.
Assume further that 
\begin{align*}\label{eq: choice of epsilon}
    \d < \min \{ \d_0, (16\beta C_{\t}^2)^{-1}\}.
\end{align*}
Set $C=3\beta C_T$. It easy to verify now that $({\bf X}^n)_{n\in\N}$ satisfies properties (ii)-(v) of Definition \ref{def:tile decomposition}. Last, Definition \ref{def:tile decomposition}(i) follows from \cite[Lemma 5.4(i)]{BM20} and \cite[Lemma 5.7(ii)]{BM20}. It follows that $(T,d)$ has $C$-tile decomposition $({\bf X}^n)_{n\in\N}$ with $C$ depending only on $C_T$. In particular, the $C$-tile decomposition $({\bf X}^n)_{n\in\N}$ satisfies the properties of \cite[Lemma 2.4]{BM22} and is therefore a \textit{quasi-visual approximation of $T$} according to Definition \cite[Definition 2.1]{BM22}. 

Let now $f:(T,d)\to (T',d')$ be an $\eta$-quasisymmetric map and let $({\mathbb X}^n)_{n\in\N}$ be a $C$-tile decomposition of $(T,d)$. By Proposition \cite[Proposition 2.7]{BM22}, $({\mathbb X}^n)_{n\in\N}$ is a quasi-visual approximation of $T'$. Therefore, by \cite[Definition 2.1, Corollary 2.3]{BM22} and the fact that $f$ is a homeomorphism, it follows that $({\mathbb X}^n)_{n\in\N}$ satisfies the properties of Definition \ref{def:tile decomposition} for some constant $C'\geq 1$ that depends on $C_T$ and $\eta$. We do a small adjustment to show that $C'>17$. 

Note that if $X\in {\bf X}^n$, $Y\in{\bf X}^{n+1}$, then 
$$\frac{\diam Y}{\diam X}\leq 3\beta C_T\d.$$
By \cite[Proposition 10.8]{Heinonen}, 
$$\frac{\diam f(Y)}{\diam f(X)}\leq \eta (6\beta C_T\d)$$
By continuity of $\eta$ there exists $d>0$ such that $\eta(6\beta C_T\d)<1/16$ for all $\d< \frac{d}{6\beta C_T}$. Hence, taking $\d< \frac{d}{6\beta C_T}$ and setting $C'=1/\eta(6\beta C_T\d)$, we have that $({\mathbb X}^n)_{n\in\N}$ is a $C'$-tile decomposition with $C'$ depending only on $C_T$ and $\eta$.
\end{proof}

Given a tile decomposition $({\bf X}^n)_{n\geq 0}$, we denote ${\bf X}^* = \bigcup_{n\in\N}{\bf X}^n$. Moreover, if $X\in {\bf X}^*$, we set
\begin{equation}\label{eq:tileboundary}
\partial X := X \cap \overline{(T\setminus X)} \subset \bigcup_{n\in\N}{\bf V}^n.
\end{equation} 

\begin{lem}\label{lem:properties of BM}
 \begin{enumerate}[(i)]
\item If $U$ is a connected component of $T\setminus{\bf V}^n$, then $\overline{U}$ is a subtree of $T$ and $\emptyset\neq\partial U\subset {\bf V}^n$.
\item Two distinct $n$-tiles $X,Y$ have at most one point in common. Such a common point is in ${\bf V}^n$ and a boundary point of both $X$ and $Y$.
\item Each $v\in{\bf V}^n$ is contained in exactly two $n$-tiles. 
\item If $X$ is an $n$-tile with $x,y\in\partial X$ distinct, then $T(x,y)\subset X\setminus\partial X$ and it intersects at least three $(n+1)$-tiles contained in $X$.
\item  If $X$ is an $n$-tile with $x\in \partial X$, then there exists a unique $(n+1)$-tile $Y$ such that $x\in Y\subset X$. Also, each $(n+1)$-tile $Y$ is contained in a unique $n$-tile $X$.
\item  If $X$ is an $n$-tile and $v\in{\bf V}^n$, then $X$ is contained in the closure of one of the two components of $T \setminus \{v\}$ and disjoint from the other component.
\item If $X$ is an $n$-tile with $x,y\in \partial X$ distinct, then $|x-y|\geq\frac{\diam X}{C^2}$.
\item We have $\max\{\diam (X):X\in{\bf X}^n\}\to 0$ as $n\to\infty$.
\end{enumerate}
\end{lem}
\begin{proof}
Claim (i) follows from \cite[Lemma 2.3(ii)]{BM20} and \cite[Lemma 5.1(i),(iii)]{BM20}. Claims (ii),(iii) follow from \cite[Lemma 5.1(iv),(v)]{BM20}. Claim (iv) follows from \cite[Lemma 5.4(ii)]{BM20} and \cite[Lemma 5.5]{BM20}. Claim (v) follows from \cite[Lemma 5.1(vii),(ix)]{BM20}. Last, Claim (vi) follows from \cite[Lemma 5.1(ii)]{BM20}. 

We show (vii). By (v), $x\in X_1,y\in X_2$, where $X_1,X_2\subset X$ are $(n+1)$-tiles and by (ii),(iv) it follows that $X_1\cap X_2=\emptyset$. By Definition \ref{def:tile decomposition}(iii),(iv) we have
   $$|x-y|\geq\dist (X_1,X_2)\geq \frac{\diam X_1}{C}\geq \frac{\diam X}{C^2}.$$

Finally, (viii) follows from Definition \ref{def:tile decomposition}(v) and \cite[Corollary 2.3]{BM22}. 
\end{proof}

\section{A catalogue of quasiconvex trees}\label{sec: model space}
Fix for the rest of this section an integer $m\geq 3$ and a constant $\epsilon\in (0,1]$. Let $\theta = \frac{\pi}{3m-3}$, let $a\in (0,1)$ such that 
\begin{equation}\label{eq:Hausdim}
2^{-\epsilon} + (2m-4)a^{1+\epsilon} = 1,
\end{equation}
and let $R\subset \mathbb{R}^2$ be the convex hull of the 4 points $\{\pm\frac{1}{2}, \pm\frac{i}{2}\tan\theta\}$. Define the following contracting similarities given (in complex notation) by
\begin{equation}\label{eq:psi}
\psi_{j}(z) = \lambda_je^{3(j-1)\theta \textbf{i}}\left(z + \tfrac{1}{2}\right), \quad \text{with}\quad \lambda_j = 
\begin{cases}
1/2, &\text{if $j=1,m$}\\
(a/2)\sin{\theta}, &\text{otherwise.}
\end{cases}  
\end{equation}

\begin{figure}[h]
\centering
\includegraphics[width=\textwidth]{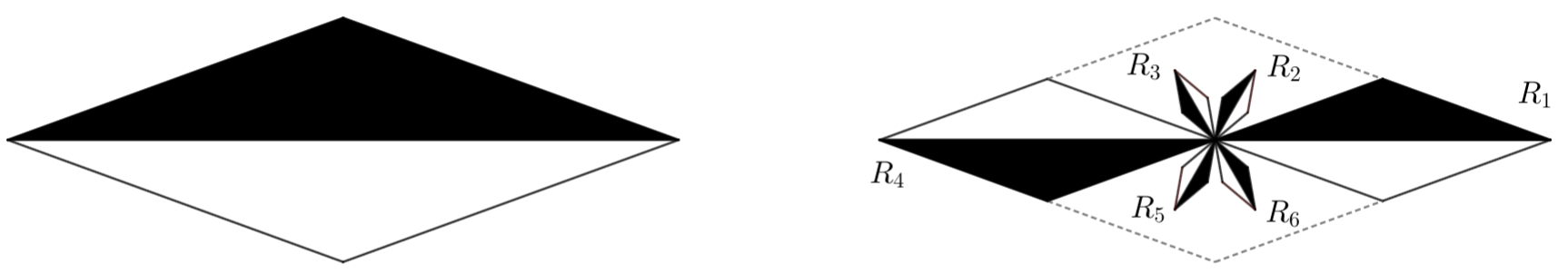}
\caption{The images of $R$ under similarities $\psi_{1},\dots,\psi_{6}$. Here we have chosen $m=4$.}
\label{fig:IFS-m=4}
\end{figure}

For each $w=j_1\cdots j_n\in A_{2m-2}^{*}$, define 
\begin{equation}\label{eq:psi2}
\psi_w:=\psi_{j_1}\circ \psi_{j_2}\circ \cdots \circ \psi_{j_n} \quad\text{and}\quad R_w=\psi_w(R)
\end{equation}
with the convention that $\psi_{\e}$ is the identity map on $\mathbb{C}$.

The following lemma is elementary and the proof is omitted. 
\begin{lem}\label{lem: geometry of rhombuses}
For each $w\in A_{2m-2}^*$ and distinct $i,j\in A_{2m-2}$ we have that $R_{wi}\subset R_w$ and that $R_{wi}\cap R_{wj}=\{\psi_w(0)\}$. Moreover, there exists a constant $c>0$ that depends only on $m$ such that for each $x\in R_{wi}$ and $y\in R_{wj}$ distinct, we have
$$|x-y|\geq c(|x-\psi_w(0)|+|y-\psi_w(0)|).$$
\end{lem}

The next lemma identifies the intersection pattern of any two rhombuses $R_w$.

\begin{lem}\label{lem: properties of R_m}
For all $n\in\N$ and $w\in A_{2m-2}^n$, define the set 
$$R_w^*= R_w \cap \bigcup_{u\in A_{2m-2}^n\setminus\{w\}} R_u.$$
Then $\card(R_w^*)\in\{1,2\}$ for each $w\in A_{2m-2}^*\setminus\{\e\}$. In particular, if $w=w'j$ for some $j\in A_{2m-2}\setminus \{1,m\}$, then $\card(R_w^*)=1$. Moreover, if $R_u\cap R_v\neq\emptyset$ for some $u,v\in A_{2m-2}^n$ distinct, then there exist $i,j\in\{1,m\}$, and $w'\in A_{2m-2}^*$ such that $R_u\cap R_v=R_{ui}\cap R_{vj}=\{\psi_{w'}(0)\}$.
\end{lem}

\begin{proof}
That $\card(R_w^*)\in\{1,2\}$ for any $w\in A_{2m-2}^*\setminus\{\e\}$, and that $\card(R_{w'j})=1$ for $j\in A_{2m-2} \setminus \{1,m\}$ follow as in \cite[Lemma 5.2(i)]{BM22}.

Fix now distinct $u,v\in A_{2m-2}^n$ and assume that $R_u\cap R_v\neq\emptyset$. Let $w\in A_{2m-2}^*$ be the longest word such that $u=wi_1\cdots i_{n-|w|}$ and $v=wj_1\cdots j_{n-|w|}$. By Lemma \ref{lem: geometry of rhombuses}, $R_{wi_1}\cap R_{wj_1}=\{\psi_w(0)\}$, $R_u\subset R_{wi_1}$, and $R_v\subset R_{wj_1}$. Hence,     
$$\emptyset \neq \psi_w^{-1}(R_u\cap R_v) = R_{i_1\dots i_{n-|w|}}\cap R_{j_1\dots j_{n-|w|}}\subset \psi_w^{-1}(\{\psi_{w}(0)\})=\{0\}.$$

If $|w|=n-1$, then $u=wi_1$, and $v=wj_1$. The above equation and the fact that $\psi_m(\frac{1}{2})=-\frac{1}{2}$, $\psi_{i_1}(-\frac{1}{2})=\psi_{j_1}(-\frac{1}{2})=0$ shows that $R_{i_1m}\cap R_{j_1m}=\{0\}.$ Thus, 
$$R_{wi_1m}\cap R_{wj_1m}=\{\psi_w(0)\}.$$
    
If $|w|<n-1$, then we claim that $i_1\cdots i_{n-|w|}=i_1m1^{(n-|w|-2)}$ and that $j_1\cdots j_{n-|w|}=j_1m1^{(n-|w|-2)}$. First, note that $0\in R_{i_1m1^{(n-|w|-2)}}\cap R_{j_1m1^{(n-|w|-2)}}$ since $\psi_{i_1}(-\frac{1}{2})=\psi_{j_1}(-\frac{1}{2})=0$, $\psi_{1}(\frac{1}{2})=\frac{1}{2}$, and $\psi_{m}(\frac{1}{2})=-\frac{1}{2}$. Moreover, a simple geometric observation shows that $\frac{1}{2}\in R_i$ exactly when $i=1$, and that $-\frac{1}{2}\in R_j$ exactly when $j=m$. Hence, $0\notin R_{i_1k}\cup R_{j_1k}$, for $k\in A_{2m-1}\setminus\{m\}$, and $0\notin R_{i_1m1^{(l)}k'}\cup R_{j_1m1^{(l)}k'}$ for any integer $l\geq 0$, and $k'\in A_{2m-2}\setminus\{1\}$. This completes the proof of the claim.

Setting now $i=j=1$, we get 
    \[ R_{ui}\cap R_{vj}=R_{wi_1m1^{(n-|w|+1)}}\cap R_{wj_1m1^{(n-|w|+1)}}=\{\psi_w(0)\}. \qedhere \]
\end{proof}

Let $\mathcal{T}^{2m-2}$ be the attractor of $\{\psi_1,\dots,\psi_{2m-2}\}$, that is,
\begin{equation}\label{eq:def of T^2m-2}
    \mathcal{T}^{2m-2}=\bigcap_{n\in\N}\bigcup_{w\in A_{2m-2}^n}R_w.
\end{equation}
Set $\mathcal{T}^{2m-2}_w=\psi_w(\mathcal{T}^{2m-2})$ for each $w\in A_{2m-2}^*$. Arguments similar to those in \cite[Proposition 6.1, Corollary 10.2, Lemma 10.6]{CGIV}, show that $\mathcal{T}^{2m-2}$ is a quasiconvex tree with uniform branch separation and valence exactly $2m-2$. Moreover, by \cite{Hu81} and by \eqref{eq:Hausdim}, the Hausdorff dimension of $\mathcal{T}^{2m-2}$ is at most $1+\epsilon$.

By Lemma \ref{lem: geometry of rhombuses} and the fact that $\psi_w(0)\in \mathcal{T}^{2m-2}_w$ it follows that for $i,j\in A_{2m-2}$ distinct, $\mathcal{T}^{2m-2}_{wi}\cap \mathcal{T}^{2m-2}_{wj}=\{\psi_w(0)\}$. Thus, $\mathcal{T}_w^{2m-2}\setminus\{\psi_w(0)\}=\bigcup_{i\in A_{2m-2}}\mathcal{T}^{2m-2}_{wi}$.

\subsection{Dictionaries} Here we construct a combinatorial catalog of quasi-self-similar fractals.

\begin{definition}\label{def: choice function}
An \emph{$m$-dictionary} is a set $\mathcal{W}\subset A_{2m-2}^*$, so that the following properties are satisfied.
\begin{enumerate}
\item[(D1)] We have that $\{1, m\} \subset \mathcal{W}$. 
\item[(D2)] If $w\in \mathcal{W}\setminus\{\e\}$, then $w^{\uparrow} \in \W$. 
\item[(D3)]  If $w \in \W$, then $\card{\{j\in A_{2m-2}:wj\in \mathcal{W}\}} \leq m$.
\item[(D4)] If $v,u\in \W$ are distinct with $|v|=|u|$ and $R_{vi}\cap R_{uj}\neq\emptyset$ for some $i,j\in \{1,m\}$, then $vi,uj \in\W$.
\end{enumerate}
\end{definition} 

\begin{lem}\label{lem: m-dictionary non-empty}
Given an $m$-dictionary $\mathcal{W}$ and an integer $n\ge 0$ define 
\[\mathcal{W}_n = \{ w\in \mathcal{W} : |w|=n\} \quad\text{and}\quad \mathcal{W}_{\infty} :=\{w\in A_{2m-2}^{\N}: w\{n\}\in\W_n \text{ for all }n\in\N\}. \]
For all integers $n\geq 0$, and $i\in \W_1$, $im1^{(n)}\in \W_{n+2}$, and $\bigcap_{n\in\N}R_{im1^{(n)}}=\{0\}$. In particular, $\W_n\neq\emptyset$ and $\W_{\infty}\neq\emptyset$.
\end{lem}

\begin{proof}
By (D1),(D2) we have $\e\in \W_0, \{1,m\}\subset\W_1$, and hence $\W_0,\W_1\neq\emptyset$. Moreover, note that for all integers $n\geq 0$, and distinct $i,j\in \W_1$, $0\in R_{im1^{(n)}}\cap R_{jm1^{(n)}}$ since $\psi_1(\frac{1}{2})=\frac{1}{2}$, $\psi_m(\frac{1}{2})=-\frac{1}{2}$, and $\psi_i(-\frac{1}{2})=\psi_j(-\frac{1}{2})=0$. Hence, by (D4), for all integers $n\geq 0$, $im1^{(n)}\in \W_{n+2}$. We conclude that $im1^{(\infty)}\in \W_{\infty}$. Last, since the sets $R_{im1^{(n)}}$ are compact, nested, and $\diam R_{im1^{(n)}}\to 0$, we have that $\{0\}=\bigcap_{n\in\N}R_{im1^{(n)}}$.
\end{proof}

\begin{lem}\label{lem:consistency of eta}
Let $\mathcal{W}$ be an $m$-dictionary. For all $w\in \W$ there exists $j\in\{1,m\}$ such that $wj \in \W$.
\end{lem}

\begin{proof}
We first show that for each $n\in\N$ and each $w\in\W_n$ there exists $u\in\W_n\setminus\{w\}$ such that $R_{w}\cap R_{u}\neq\emptyset$. The proof is by induction on $n$. Suppose first that $w\in\W_1$. By (D1), there exists $i\in (\{1,m\} \setminus \{w\})\cap\W$. By Lemma \ref{lem: geometry of rhombuses}, $R_{w}\cap R_{i} \neq \emptyset$. Suppose now that the claim is true for some $n\in\N$. Fix $w\in \mathcal{W}_{n+1}$. By the induction hypothesis, there exists $u \in \mathcal{W}_n\setminus \{w^{\uparrow}\}$ such that $R_{w^{\uparrow}}\cap R_{u}\neq\emptyset$. By Lemma \ref{lem: properties of R_m} there exist $i,j\in \{1,m\}$ such that $R_{w^{\uparrow}i}\cap R_{uj}\neq\emptyset$. Therefore, by (D4) we have $w^{\uparrow}i,uj\in \W$. If $w=w^{\uparrow}i$, then the induction is complete. If $w=w^{\uparrow}k$ for some $k\in A_{2m-2}\setminus\{i\}$, then by Lemma \ref{lem: geometry of rhombuses} we have $R_{w^{\uparrow}i}\cap R_{w}\neq\emptyset$ and the proof of the claim is complete.

For the proof of the lemma, fix $w\in\W$.  If $w=\e$, then by (D1) we have $\{1,m\}\subset \W$. If $w\neq\e$, then by the claim there exists $u\in\W\setminus\{w\}$ such that $|w|=|u|$ and $R_{u}\cap R_{w}\neq\emptyset$. By Lemma \ref{lem: properties of R_m} there exist $i,j\in\{1,m\}$ such that $R_{wj}\cap R_{ui}\neq\emptyset$. By (D4) we conclude that $wj\in \W$.
\end{proof}

Given an $m$-dictionary $\mathcal{W}$, define the set
\begin{align}\label{eq:def of T^W}
    \mathcal{T}^{\mathcal{W}} := \bigcap_{n \in \mathbb{N}} 
\bigcup_{w\in \W_n}R_{w}.
\end{align}
Note that $\mathcal{T}^\W$ implicitly depends on $\epsilon$ and $m$, although we omit these dependencies from the notation for simplicity. Among all $m$-dictionaries, a key role for the rest of the paper plays the $m$-dictionary $\mathcal{W} = A^*_{m}$ and the associated generated set
\begin{equation}\label{eq: T^m rhombus def}
\mathcal{T}^{A^*_{m}}=\bigcap_{n\in\N}\bigcup_{w\in A_m^n}R_w.
\end{equation}

In the next lemma, we show some topological properties of  $\mathcal{T}^{\W}$. 
\begin{lem}\label{lem: Tm,eta connected}
For every $m$-dictionary $\mathcal{W}$, the set $\mathcal{T}^{\W}$ is non-empty, compact and connected.
\end{lem}

\begin{proof}
Fix an $m$-dictionary $\mathcal{W}$. For each $n\in\N$, let 
\[ J_{n}=\bigcup_{w\in \W_n}R_{w}.\]
    
By Lemma \ref{lem: geometry of rhombuses}, if $w,wi\in \mathcal{W}$, then $R_{wi}\subset R_{w}$. Hence, by (D2),
$$J_{n+1} = \bigcup_{w\in \W_{n+1}}R_{w}\subset \bigcup_{w\in\W_n}R_{w}=J_{n}.$$
Since $\mathcal{W}_n\neq\emptyset$ for each $n\in\N$, the set $J_{n}$ is nonempty and compact as a finite union of compact nonempty sets. Therefore, $\mathcal{T}^{\W}$ is compact and by the nestedness of the sets $J_n$ it follows that $\mathcal{T}^{\W}$ is non-empty. 

We show now by induction on $n$ that each $J_{n}$ is connected. Assuming connectedness, by compactness and nestedness of the sets $J_{n}$, it follows that $\mathcal{T}^{\W}$ is connected. For $n=1$, by Lemma \ref{lem: geometry of rhombuses}, $J_{1}$ is the union of a finite number of rhombuses $R_i$ that all intersect at the same point, so $J_{1}$ is connected. 

Assume now that $J_{n}$ is connected for some $n\in\N$ and let $u\in\W_n$. By Lemma \ref{lem: properties of R_m} and the induction hypothesis,
$$\card (\{v\in\W_n\setminus\{u\}: R_{u}\cap R_{v}\neq\emptyset\})\in\{1,2\}.$$
Moreover, by Lemma \ref{lem: geometry of rhombuses}, if $ui,ui'\in\W_{n+1}$, then $R_{ui}\cap R_{ui'}\neq\emptyset$. Therefore, since $J_{n+1}\subset J_{n}$ and $J_{n}$ is connected, it is enough to show that there exists $ui'\in\W_{n+1}$ such that $R_{ui'}\cap R_{u}^*\neq\emptyset$, where $R_{u}^*$ is as in Lemma \ref{lem: properties of R_m}. Fix $v\in\W_n$ distinct from $u$ such that $R_{v}\cap R_{u}\neq\emptyset$. By Lemma \ref{lem: properties of R_m} there exist $j,j'\in \{1,m\}$ such that 
$$R_{u}\cap R_{v}=R_{uj}\cap R_{vj'}\neq\emptyset.$$ 
By (D4) it follows that $uj,vj'\in\W_{n+1}$ and we conclude that $R_{uj}\cap R_{u}^*\neq\emptyset$.
\end{proof}

\begin{lem}\label{lem: pi functions}
Let $\mathcal{W}$ be an $m$-dictionary. The ``projection'' map
\begin{equation}\label{eq:proj} 
\pi_{\W} : \W_{\infty} \to \mathcal{T}^{\W},\quad \text{by}\quad 
\bigcap_{n\in\N}R_{w\{n\}} = \{\pi_{\W}(w)\} \quad \text{for each $w\in \W_{\infty}$}
\end{equation}
is a well-defined surjection. Moreover, $\pi_{\W}(v)=\pi_{\W}(w)$ for distinct $v,w\in \W_{\infty}$ if and only if there exists $u\in\W$ and distinct $i,j \in A_{2m-2}$ such that $v=uim1^{(\infty)}$ and $w=ujm1^{(\infty)}$.
\end{lem}

\begin{proof}
To show that $\pi_{\W}$ is well-defined, fix $w\in \W_{\infty}$. Since $w\{n\}\in \W_n$ for each $n\in\N$, the sets $R_{w\{n\}}$ are compact and nested and $\diam R_{w\{n\}}\to 0$. Hence, $\card(\bigcap_{n\in\N} R_{w\{n\}})=1$, so $\pi_{\W}$ is well-defined. 

Before proving the surjectivity of $\pi_\W$, we show that 
\begin{equation}\label{eq:pi-1(0)}
\pi_{\W}^{-1}(0)=\{im1^{(\infty)} : i\in \W_1\}.
\end{equation}
By Lemma \ref{lem: m-dictionary non-empty} $0\in \mathcal{T}^{\W}$ and $\pi_{\W}(im1^{(\infty)})=0$.  Suppose now that $\pi_{\W}(w)=0$ for some $w=i_1i_2\cdots\in\W_{\infty}$ where $i_j\in A_{2m-2}$. Then, $\psi_{i_1}(\bigcap_{n\in\N}R_{i_2\cdots i_n}) = 0$, which gives that $\bigcap_{n\in\N}R_{i_2\cdots i_n} = -1/2$. A simple geometric observation shows now that $i_2=m$. Hence, $\bigcap_{n\in\N}R_{i_3\cdots i_n} = 1/2$. Notice that $1/2\in R_i$ exactly when $i=1$, and that $1/2$ is a fix point for $\psi_1$. Thus, $i_3=i_4=\cdots=1$. Therefore, $w=i_1m1^{(\infty)}$ which shows \eqref{eq:pi-1(0)}.

To show surjectivity, fix $y\in \mathcal{T}^{\W}$. For each $n\in\N$, there exists $w_n \in \W_n$ such that $y\in R_{w_n}$. If for all $n\in\N$, $R_{w_{n+1}}\subset R_{w_n}$, then there exists $i_{n+1} \in A_{2m-2}$ such that $w_{n+1} = w_ni_{n+1}$. Setting $w = w_1i_2i_3\cdots$, we have that $\pi_{\W}(w)=y$. Suppose now that $n\in\N$ is maximal such that $R_{w_{n+1}}\not\subset R_{w_n}$. By (D2) we may assume $w_{n+1}=ui$ for some $u\in\W_{n}$ and $i\in A_{2m-2}$. Let $v\in \W$ be the longest word such that $w_n=vi_1\cdots i_{n-|v|}$ and $u=vj_1\cdots j_{n-|v|}$. By Lemma \ref{lem: geometry of rhombuses}, $R_{vi_1}\cap R_{vj_1}=\{\psi_{v}(0)\}$, $R_{w_n}\subset R_{vi_1}$, and $R_{u}\subset R_{vj_1}$. Hence, 
$$y\in R_{w_n}\cap R_{w_{n+1}}\subset  R_{w_n}\cap R_{u}=R_{vi_1\dots i_{n-|v|}}\cap R_{vj_1\dots j_{n-|v|}}\subset\{\psi_{v}(0)\}.$$
We conclude that $y=\psi_{v}(0)$. Thus, by \eqref{eq:pi-1(0)}, $y=\pi_{\W}(v1m1^{(\infty)})$.

To prove the second claim, suppose that $\pi_{\W}(v)=\pi_{\W}(w)$ for some distinct $v,w\in\W_{\infty}$. Let $u\in\mathcal{\W}_n$ be the longest word such that $v=ui_1i_2\cdots$ and $w=uj_1j_2\cdots$ with $i_1,i_2,\dots\in A_{2m-2}$ and $j_1,j_2,\dots \in A_{2m-2}$. Then, 
$$\{\pi_{\W}(v)\}=\psi_{u}(\bigcap_{n\in\N}R_{i_1\cdots i_n})=\psi_{u}(\bigcap_{n\in\N}R_{j_1\cdots j_n})=\{\pi_{\W}(w)\}.$$
By maximality of $u$, $j_1\neq i_1$. Moreover, $\psi_{u}^{-1}(\pi_{\W}(v))$ and $\psi_{u}^{-1}(\pi_{\W}(w))$ are both in $R_{i_1}\cap R_{j_1}$. Therefore, by \eqref{eq:pi-1(0)} and Lemma \ref{lem: geometry of rhombuses} the words $i_1i_2\cdots$ and $j_1j_2\cdots$ are in 
$$\{im1^{(\infty)}: i\in A_{2m-2}\}.$$
The reverse implication follows easily by \eqref{eq:pi-1(0)}.
\end{proof}

\subsection{Dictionary functions} We show here that sets $\cT^{\cW}$ are quasiconvex trees with branch separation. For the proof we require a notion of function between dictionaries. To that end, if $\cW$ is an $m$-dictionary and $w\in\mathcal W\setminus\{\varepsilon\}$, set
\[F_w:=\left\{ i\in\{1,m\}: R_{wi}\cap \bigcup_{{u\in\mathcal W_{|w|}\setminus\{w\}}}R_u \neq\emptyset \right\},\]
and set $F_\varepsilon:=\{1,m\}$.

\begin{definition}\label{def:dictfun}
Given an $m$-dictionary $\mathcal{W}$, a function $\eta: \mathcal{W}\cup\W_{\infty} \to A_{2m-2}^*\cup A_{2m-2}^{\N}$ is a \emph{dictionary function} if it is injective and satisfies the following rules.
\begin{enumerate}
\item We have that {$\eta(\mathcal{W})\subset A_{2m-2}^*$} and $\eta(\W_{\infty}) \subset A_{2m-2}^{\N}$.
\item We have that $\eta(\varepsilon) = \varepsilon$.
\item For every $w\in\mathcal W$, there exists a permutation $\sigma_w$ of $A_{2m-2}$ such that $\sigma_w(i)=i$ for all $i\in F_w$, and $\eta(wi)=\eta(w)\sigma_w(i)$ for all $wi\in\mathcal W$. 
\item If $w\in \W_{\infty}$ and $n\in\N$, then $\eta(w)\{n\}  = \eta(w\{n\})$.
\end{enumerate}
\end{definition}

\begin{lem}\label{lem:dictfun2}
Let $\eta: \mathcal{W}\cup\W_{\infty} \to A_{2m-2}^*\cup A_{2m-2}^{\N}$ be a dictionary function.
\begin{enumerate}[(i)]
\item For each $w\in \mathcal{W}$, $|\eta(w)|=|w|$. 
\item The image $\eta(\mathcal{W})$ is an $m$-dictionary.
\end{enumerate}
\end{lem}

\begin{proof}\label{lem:dictfun}
The first claim follows easily from (2) and (3) of the definition.

We show now that $\eta(\W)$ is an $m$-dictionary. Property (D1) follows from (2) and (3) of the definition along with (D1) for $\W$. For (D2) we simply observe that by (3) of the definition and the injectivity of $\eta$, if $w\in \W\setminus\{\e\}$, then $\eta(w)^{\uparrow} = \eta(w^{\uparrow})$. For property (D3), we observe that, by injectivity of $\eta$, we have $\eta(w)j =\eta(u)$ if and only if $u=wi$ for some $i$. Therefore, since $\W$ satisfies (D3),
\[\card\{j\in A_{2m-2} : \eta(w)j \in \eta(\W)\} = \card\{i\in A_{2m-2} : wi \in \W\} \leq m.\]

To show (D4), we first show that if $R_{\eta(u)}\cap R_{\eta(w)}\neq\emptyset$ for distinct $u,w\in\mathcal W_n$, then $R_u\cap R_w\neq\emptyset$. Suppose that $w,u\in\cW_n$, and let $v\in\mathcal W$ be their longest common initial word. Write $w=vi_1\cdots i_{n-|v|}$ and $u=vj_1\cdots j_{n-|v|}$ with $i_1\neq j_1$. By (3) and injectivity of $\eta$, the longest common initial word of $\eta(w)$ and $\eta(u)$ is $\eta(v)$. Set $d=n-|v|$. If $d=1$, then Lemma \ref{lem: geometry of rhombuses} gives
\[R_w\cap R_u=R_{vi_1}\cap R_{vj_1}=\{\psi_v(0)\}.\]
Suppose now that $d\ge2$. Since $R_{\eta(w)}\cap R_{\eta(u)}\neq\emptyset$, following the proof of Lemma \ref{lem: properties of R_m} we have $\eta(w)=\eta(v)i_1'm1^{(d-2)}$ and $\eta(u)=\eta(v)j_1'm1^{(d-2)}$ where $i_1',j_1'\in A_{2m-2}$. We claim that $i_2\cdots i_{d}=j_2\cdots j_{d}=m1^{(d-2)}$. Since $vi_1,vj_1\in\mathcal W$ are distinct, since $R_{vi_1}\cap R_{vj_1}=\{\psi_v(0)\}$, and since 
\[ \psi_v(0)=\psi_{vi_1}(-1/2)=\psi_{vj_1}(-1/2),\]
we have $m\in F_{vi_1}\cap F_{vj_1}$. By (3), we have $i_2=j_2=m$. If $d=2$, this proves the claim. Suppose that $d\geq3$ and note that $R_{vi_1m}\cap R_{vj_1m}=\{\psi_v(0)\}$ and
\[\psi_v(0)=\psi_{vi_1m}(1/2) =\psi_{vj_1m}(1/2).\]
Hence, $1\in F_{vi_1m}\cap F_{vj_1m}.$ Again, by (3), it follows that $i_3=j_3=1$. Continuing in similar fashion, we obtain $w=vi_1m1^{(d-2)}$ and $u=vj_1m1^{(d-2)}$. A simple inspection of the similarities now gives $R_w\cap R_u=\{\psi_v(0)\}$ which proves the claim.

Assume now that $R_{\eta(w)i}\cap R_{\eta(u)j}\neq\emptyset$ for some $i,j\in\{1,m\}$. Since 
\[ R_{\eta(w)i}\cap R_{\eta(u)j}\subset R_{\eta(w)}\cap R_{\eta(u)},\] 
the claim gives $R_w\cap R_u\neq\emptyset$. Let $v$ be the longest common initial word of $w,u$. The preceding argument shows that there exists $t\in\{-1/2,1/2\}$ such that
\[ \psi_w(t)=\psi_u(t)=\psi_v(0), \qquad \psi_{\eta(w)}(t)=\psi_{\eta(u)}(t) =\psi_{\eta(v)}(0).\]
Since $R_{\eta(w)i}\cap R_{\eta(u)j}\neq\emptyset$, it follows that $t\in R_i\cap R_j$. Applying $\psi_w$ and $\psi_u$, respectively, gives $\psi_v(0)\in R_{wi}\cap R_{uj}.$ Thus, $i\in F_w$ and $j\in F_u$. By (D4) for $\W$, $wi,uj\in\mathcal W$, and by (3), $\eta(wi)=\eta(w)i$, $\eta(uj)=\eta(u)j.$ Consequently, $\eta(w)i,\eta(u)j\in\eta(\mathcal W)$, which proves (D4).
\end{proof}

\begin{lem}\label{lem: global homeo}
For any $m$-dictionary $\W$, the space $\mathcal{T}^{\W}$ $L$-bi-Lipschitz embeds into $\mathcal{T}^{A^*_{m}}$ with $L\geq 1$ depending only on $m$ and $\epsilon$. In particular, there exists a constant $C\geq 1$ depending only on $m$ such that $\mathcal{T}^{\W}$ is a $c_m$-quasiconvex tree with $C_m$-uniform branch separation, valence at most $m$, and Hausdorff dimension at most $1+\epsilon$. Here $C_m,c_m$ depend only on $m$.
\end{lem}

\begin{proof}
Fix an $m$-dictionary $\W$. We construct a dictionary function $\eta$ such that $\eta(\W) \subset A_m^*$ and $\eta(\W_{\infty}) \subset A^{\N}_{2m-2}$. The construction of $\eta$ is in an inductive fashion. Set $\eta(\varepsilon)=\varepsilon$. Assume now that for some integer $n\geq 0$ we have defined an injective function
\[ \eta : \bigcup_{k=0}^n\W_k \to \bigcup_{k=0}^nA_{m}^k\]
so that property (2) holds. Fix $w\in \W_n$ and write 
\[ \{i_1,\dots,i_l\} = \{i\in A_{2m-2} : wi \in \W\}.\]
By (D3) we have $l\leq m$. If $i_r\in \{1,m\}$, then set $\eta(wi_r)=\eta(w)i_r$. Otherwise, assign $\eta(wi_r)=\eta(w)j_r$ with $j_r \in A_{m}\setminus \{1,m\}$ so that $j_1,\dots, j_l$ are distinct. 
Proceeding inductively, we define $\eta$ on all of $\W$. Finally, define $\eta$ on $\W_{\infty}$ by $\eta(i_1i_2\cdots) = j_1j_2\cdots$ so that for all $n\in\N$, $\eta(i_1\cdots i_n) = j_1\cdots j_n$. Applying an induction, it is easy to see that $\eta$ is a dictionary function. By properties (1),(2), and (3) of Definition \ref{def:dictfun}, it is easy to see that $\diam{R_w} = \diam{R_{\eta(w)}}$ for all $w\in\W$.

Let $\pi_{\W},\pi_{A^*_{m}}$ be as in Lemma \ref{lem: pi functions}. Define 
\[ g:\mathcal{T}^{\W}\to \mathcal{T}^{A^*_{m}}, \quad\text{by}\quad g\circ \pi_{\W} = \pi_{A^*_{m}}\circ (\eta| \mathcal{W}_{\infty}).\]
We first show that $g$ is a well-defined injection. Let $x\in\mathcal{T}^{\W}$ and let $w,v\in \mathcal{W}_{\infty}$ such that $w\neq v$ and $x=\pi_{\W}(w)=\pi_{\W}(v)$. By Lemma \ref{lem: pi functions} there exists $u\in\W$ such that $v,w\in\{uim1^{(\infty)}:i\in A_{2m-2}\}$ which, by design of $\eta$, gives 
\[\eta(v),\eta(w) \in \{\eta(ui)m1^{(\infty)}: i\in A_m\}\]
and it follows again from Lemma \ref{lem: pi functions} that $\pi_{A^*_{m}}(\eta(v))=\pi_{A^*_{m}}(\eta(w))$. Hence, $g$ is well-defined. Interchanging $\pi_{\W}$ and $\pi_{A^*_{m}}$, we get that $g^{-1}$ is well-defined so $g$ is injective.

We show now that $g$ is bi-Lipschitz. Let $x,y\in\mathcal{T}^{\W}$ distinct. Let $u,v\in\W_{\infty}$ such that $x=\pi_{\W}(u)$ and $y=\pi_{\W}(v)$. Since $x\neq y$, there exists $w\in\W$ such that $u=wi_1i_2\cdots$, $v=w j_1 j_2\cdots$, and $i_1\neq j_1$. Moreover, by Lemma \ref{lem: geometry of rhombuses} there exists maximal $n\in\N$ such that $R_{wi_1\cdots i_n} \cap R_{wj_1\cdots j_n} \neq \emptyset$. We may assume that $n\geq 2$; the case where $n=1$ is similar. Following the proof of Lemma \ref{lem: properties of R_m} we have $wi_1\cdots i_n = wi_1m1^{(n-2)}$ and $wj_1\cdots j_n = wj_1m1^{(n-2)}$ with either $i_{n+1} \neq 1$ or $j_{n+1} \neq 1$. Thus,
\[ |x-y| \simeq \diam{R_{wi_1m1^{(n-2)}}} \simeq (1/2)^{n}\diam{R_w}\]
with the implied constant depending only on $m$ and $\epsilon$.

By design of $\eta$, $\eta(u) = \eta(w)i_1'm1^{(n-2)}i_{n+1}'\cdots$ and $\eta(v) = \eta(w)j_1'm1^{(n-2)}j_{n+1}'\cdots$ with either $i_{n+1}' \neq 1$ or $j_{n+1}' \neq 1$. Therefore,
\[ |g(x)-g(y)| \simeq \diam{R_{\eta(wi_1m1^{(n-2)})}} \simeq (1/2)^{n}\diam{R_{\eta(w)}} = (1/2)^{n}\diam{R_{w}}\]
with the implied constant depending only on $m$ and $\epsilon$.

By Lemma \ref{lem: Tm,eta connected} $\mathcal{T}^{\W}$ is compact and connected. Arguments similar to those of \cite[Proposition 6.1, Corollary 10.2, Lemma 10.6]{CGIV} show that $\mathcal{T}^{A^*_{m}}$ is a quasiconvex tree with uniform branch separation and valence exactly $m$. Therefore, $g(\mathcal{T}^{\W})$ is a quasiconvex tree with uniform branch separation and valence at most $m$ as a subtree of $\mathcal{T}^{A^*_{m}}$. Since these properties are bi-Lipschitz invariant, $\mathcal{T}^{\W}$ is also a quasiconvex tree with uniform branch separation and valence at most $m$. Finally, the Hausdorff dimension of $\mathcal{T}^{\W}$ is at most that of $\mathcal{T}^{A^*_{m}}$ which by \eqref{eq:Hausdim} and by Hutchinson \cite{Hu81}, is at most $1+\epsilon$.
\end{proof}

\begin{lem}\label{lem:treearc}
If $n\in\N$, and $w,u\in\mathcal{W}_n$ are distinct with $R_u\cap R_w\neq\emptyset$, then $\mathcal{T}^{\mathcal{W}}[x,y]\subset R_u\cup R_w$ for each $x\in \mathcal{T}^{\mathcal{W}}\cap R_w$ and $y\in \mathcal{T}^{\mathcal{W}}\cap R_u$.
\end{lem}

\begin{proof}
We show the claim by induction on $n$. If $n=1$, then the claim follows from Lemma \ref{lem: geometry of rhombuses} and the fact that $\mathcal{T}^{\mathcal{W}}$ is a tree. Suppose now that the claim is true for some $n\in\N$ and let $w,u\in \mathcal{W}_{n+1}$ with $R_w\cap R_u\neq\emptyset$, and let $x\in R_w$ and $y\in R_u$. Write $w=w'i$, $u=u'j$ where $w',u'\in \mathcal{W}_n$ and $i,j\in A_{2m-2}$. By Lemma \ref{lem: geometry of rhombuses}, $\emptyset\neq R_{w}\cap R_u\subset R_{w'}\cap R_{u'}$. By the induction hypothesis, $\mathcal{T}^{\mathcal{W}}[x,y]\subset R_{w'}\cup R_{u'}$. Since $x\in R_{w'i}$ and $y\in R_{u'j}$, by Lemma \ref{lem: properties of R_m} and the fact that $\mathcal{T}^{\mathcal{W}}$ is a tree, we get $\mathcal{T}^{\mathcal{W}}[x,y]\subset R_w\cup R_u$ and the proof of the claim is complete.
\end{proof}

\subsection{A uniformization of planar trees}\label{sec:orientation of T} 
The following proposition is the main result of Section \ref{sec: model space}.

\begin{prop}\label{prop: enumeration of T=enumeration of Tmeta}
Let $m\in\{2,3,\dots,\}$ and $\epsilon>0$.
\begin{enumerate}[(i)]
\item There exist $C_m,c_m \geq$ depending only on $m$ such that if $\mathcal{W}$ is an $m$-dictionary, then $\mathcal{T}^{\mathcal{W}}$ is a $c_m$-quasiconvex tree, with $C_m$-uniform branch separation, Hausdorff dimension at most $1+\epsilon$, and valence at most $m$.
\item If $T\subset \R^2$ is a tree with valence at most $m$, then there exist an $m$-dictionary $\W$ and a branch-orientation-preserving homeomorphism $h:T\to \mathcal{T}^\W$ such that $\mathcal{T}^\W$ is semi-geodesic.
\item If $T\subset \R^2$ is a quasiconformal tree with uniform branch separation and valence at most $m$, then there exist an $m$-dictionary $\W$ and a branch-orientation-preserving $\theta$-quasisymmetric homeomorphism $g:T\to\mathcal{T}^{\W}$ such that $\mathcal{T}^{\mathcal{W}}$ is  semi-geodesic. The distortion $\theta$ depends only on $m$, $\epsilon$, the bounded turning constant of $T$, and the uniform branch separation constant of $T$.
\end{enumerate}
\end{prop}

Before we continue with the proof of Proposition \ref{prop: enumeration of T=enumeration of Tmeta}, we record an immediate corollary which states that $m$-dictionaries essentially classify all metric trees of valence at most $m$ and all quasiconformal trees with uniform branch separation.

\begin{cor}\label{thm:f(T) is Tmeta} 
Let $(T,d)$ be a metric tree with valence at most $m$. There exists an $m$-dictionary $\W$ and a homeomorphism $f:T\to \mathcal{T}^\W$.
Moreover, if $T$ is a quasiconformal tree with uniform branch separation, then $f$ can be assumed quasisymmetric.
\end{cor}

\begin{proof}
By \cite{Charatonik} there exists an embedding $f:T\to \mathcal{T}^{A_m^*}$. Moreover, if $T$ is a quasiconformal tree with uniform branch separation, one can take $f$ to be quasisymmetric \cite[Theorem 1.3(i),(iii)]{CGIV}. Clearly $f(T)$ is a planar quasiconformal tree with uniform branch separation and valence at most $m$ as a connected closed subset of $\mathcal{T}^{A_m^*}$. Proposition \ref{prop: enumeration of T=enumeration of Tmeta} now gives the result.
\end{proof}


For the proof of Proposition \ref{prop: enumeration of T=enumeration of Tmeta}, we require two lemmas. For the first lemma, recall the functions $\psi_i$ from \eqref{eq:psi} and their compositions $\psi_w$ from \eqref{eq:psi2}.

\begin{lem}\label{lem:orientation of T2m-2}
Let $w\in A_{2m-2}^*$ and let $D$ be a Jordan domain with $\psi_w(0)\in D$ and $\partial D\cap \mathcal{T}^{2m-2}_{wi}\neq\emptyset$ for all $i\in A_{2m-2}$. Then for each $i\in A_{2m-2}$ there exists $p_{wi}\in \mathcal{T}^{2m-2}_{wi}\cap\partial D$ such that $\mathcal{T}^{2m-2}_{wi}[\psi_w(0),p_{wi})\subset D$. The counterclockwise enumeration of the points $\{p_{wi}\}_{i\in A_{2m-2}}$ along $\partial D$, starting from $p_{w1}$, is $p_{w1},\dots, p_{w(2m-2)}.$
\end{lem}

\begin{proof}
For each $i\in A_{2m-2}$ denote by $B_{i}$ the unique component of $\mathcal{T}^{2m-2}\setminus \{\psi_w(0)\}$ such that $\mathcal{T}^{2m-2}_{wi} \subset \overline{B_i}$. Then for each $i\in A_{2m-2}$, $\mathcal{T}^{2m-2}_{wi}$ is a subtree of $\overline{B_{i}}$. By Lemma \ref{lem: maximal arcs}(i) for each $i\in A_{2m-2}$ there exists $p_{wi}\in \mathcal{T}^{2m-2}_{wi}\cap\partial D$ such that $\mathcal{T}^{2m-2}_{wi}[\psi_w(0),p_{wi})\subset D$. It remains to find the counterclockwise enumeration starting from $p_{w1}$. By Lemma \ref{lem: maximal arcs}(ii), we may assume that $D=B(\psi_w(0),r)$ as the enumeration is independent of the Jordan domain. Moreover, since $D\cap \mathcal{T}^{2m-2} = D\cap \mathcal{T}^{2m-2}_w$, we may assume that $w=\varepsilon$. Recall that $\mathcal{T}^{2m-2}_i \subset R_i$ for each $i=1,\dots,2m-2$. The proof now follows from the simple geometric observation that rhombuses are counterclockwise enumerated as $R_1,R_2,\dots,R_{2m-2}$.
\end{proof}

The second lemma states that each dictionary function induces a quasisymmetric homeomorphism.

\begin{lem}\label{lem:dictqs}
Let $\eta: \mathcal{W}\cup\W_{\infty} \to A_{2m-2}^*\cup A_{2m-2}^{\N}$ be a dictionary function. There exists a quasisymmetric homeomorphism $h:\mathcal{T}^{\mathcal{W}}\to \mathcal{T}^{\eta(\mathcal{W})}$ such that $h(\mathcal{T}^{\mathcal{W}}_w) = \mathcal{T}^{\eta(\mathcal{W})}_{\eta(w)}$ for all $w\in \mathcal{W}$. The distortion of $h$ depends only on $m$ and $\epsilon$.
\end{lem}

\begin{proof}
Define $h:\mathcal{T}^{\mathcal{W}}\to \mathcal{T}^{\eta(\mathcal{W})}$ by $$h\circ\pi_{\mathcal{W}}=\pi_{\eta(\mathcal{W})}\circ(\eta| \mathcal{W}_{\infty})$$
where $\pi_{\mathcal{W}}$ and $\pi_{\eta(\mathcal{W})}$ are as in Lemma \ref{lem: pi functions}. Since $\mathcal{T}^{\mathcal W}$ and $\mathcal{T}^{\eta(\mathcal W)}$ are connected doubling spaces, by \cite[Theorem 10.19]{Heinonen} it suffices to show that $h$ is $H$-weakly quasisymmetric with $H$ depending only on $m,\epsilon$. Consequently, $g:=h\circ f$ is a $\theta$-quasisymmetric map with $\theta$ depending on $m,\epsilon$, the bounded turning constant of $T$, and the uniform branch separation constant.

We first show that if $u,v\in\mathcal W$ with $R_u\cap R_v\neq\emptyset$ and sets $R_u,R_v$ are not subsets of each other, then
\begin{equation}\label{eq:comparability of intersecting rhombuses}
\frac{\diam R_u}{\diam R_v}\simeq_{m,\epsilon}\frac{\diam R_{\eta(u)}}{\diam R_{\eta(v)}}.
\end{equation}
Recall the contraction ratios $\lambda_i$ from \eqref{eq:psi} and set $r:=\min_{i\in \{1,\dots,m\}} \lambda_i$. Let $u,v\in\mathcal W$ such that $R_u\cap R_v\neq\emptyset$ and neither $u$ nor $v$ is initial subword of each other. Write $u=wi\alpha$ and $v=wj\beta$, where $i\neq j$ and $\alpha,\beta\in A_{2m-2}^*$, and write $\eta(wi)=\eta(w)i'$ and $\eta(wj)=\eta(w)j'$. By the geometry of the rhombuses, it follows that each of $\alpha,\beta$ is either $\e$ or of the form $m1^{(l)}$, for some $l\geq 0$. Consequently, by construction of $\eta$,
\[\frac{\diam R_u}{\diam R_v} =\frac{\lambda_i}{\lambda_j}\,2^{|\beta|-|\alpha|} \simeq_{m,\epsilon} \frac{\lambda_{i'}}{\lambda_{j'}}\,2^{|\beta|-|\alpha|} = \frac{\diam R_{\eta(u)}}{\diam R_{\eta(v)}}.\]

Fix now $x,y,z\in \mathcal{T}^{\mathcal{W}}$ such that $|x-y|\leq|x-z|$. We may assume that $x,y,z$ are distinct. Choose $u,v\in\mathcal W_\infty$ with $x=\pi_{\mathcal W}(u)$ and $z=\pi_{\mathcal W}(v)$. Write $u=wi_1i_2\cdots$ and $v=wj_1j_2\cdots$, where $i_1\neq j_1$. Let $n$ be maximal such that $R_{wi_1\cdots i_n}\cap R_{wj_1\cdots j_n}\neq\emptyset$, and set $u':=wi_1\cdots i_n$, $v':=wj_1\cdots j_n$. Write $R_{u'}\cap R_{v'}=\{b\}$, and by Lemma \ref{lem: properties of R_m} there exist $e,f\in\{1,m\}$ such that $b\in R_{u'e}\cap R_{v'f}$. By Lemma \ref{lem:treearc}, $\mathcal{T}^{\mathcal{W}}[x,z]\subset R_{u'}\cup R_{v'}$. More precisely,
\begin{equation*}
\mathcal{T}^{\mathcal{W}}[x,z] \subset R_{u'i_{n+1}}\cup R_{u'e}\cup R_{v'f} \cup R_{v'j_{n+1}}.
\end{equation*}
By maximality of $n$ either $i_{n+1}\neq e$ or $j_{n+1}\neq f$. Without loss of generality, assume the former and set $s=\diam R_{u'e}$. It follows that $\ell(\mathcal{T}^{\mathcal{W}}[x,z])\geq s$. In conjuction with quasiconvexity of $\mathcal{T}^{\mathcal{{W}}}$ and the above equation we have $|x-z|\simeq_{m,\epsilon}s$. Moreover, by definition of $h$ and quasiconvexity of $\mathcal{T}^{{\eta (\mathcal{W})}}$
\begin{equation}\label{eq: quasiconvexity of T^eta(W)}
\diam R_{\eta(u'e)} \leq \ell\bigl(\mathcal{T}^{\eta(\mathcal W)}[h(x),h(z)]\bigr) \lesssim_{m} |h(x)-h(z)|.
\end{equation}

Consider now the connected set $\Gamma:=\mathcal{T}^{\mathcal W}[x,y]\cup \mathcal{T}^{\mathcal W}[x,z]$. By quasiconvexity and the assumption $|x-y|\leq|x-z|$, there is a constant $C$, depending only on $m$, such that $\Gamma\subset\overline{B}(x,C|x-z|).$ Define
\[ \mathcal M_s:= \left\{ v\in\mathcal W\setminus\{\varepsilon\}: \diam R_v\leq s < \diam R_{v^\uparrow}\text{ and }R_v\cap\Gamma\neq\emptyset \right\}.\]
The rhombuses $\{R_v\}_{v\in\mathcal{M}_s}$ cover $\G$, have pairwise disjoint interiors, satisfy $rs<\diam R_v\leq s$, lie in $\overline{B}(x,C|x-z|+s)$, and satisfy $\operatorname{area}(R_v)\simeq_{m,\epsilon} s^2$. Hence,
\[\card(\mathcal{M}_s)\lesssim_{m,\epsilon}s^{-2}\sum_{v\in\mathcal{M}_s}\operatorname{area}(R_v) \lesssim s^{-2}(C|x-z|+s)^2\lesssim_{m,\epsilon} 1.\]

The sets $\{R_v\cap\Gamma\}_{v\in\mathcal{M}_s}$, form a finite closed cover of the connected set $\Gamma$. Their intersection is therefore connected. It follows from \eqref{eq:comparability of intersecting rhombuses} that there exists $L>0$ depending only on $m,\epsilon$ such that $\diam R_{\eta(v)} \leq L^{\card(\mathcal{M}_s)}\diam R_{\eta(u'e)}.$ By definition of $h$ we have that $h(\Gamma)\subset \bigcup_{v\in\mathcal{M}_s}R_{\eta(v)}.$ This union is connected, because intersections along $\Gamma$ are preserved by $h$. It follows that
\begin{align*} 
|h(x)-h(y)| \leq \diam \left(\bigcup_{v\in\mathcal{M}_s}R_{\eta(v)}\right) \leq \sum_{v\in\mathcal{M}_s}\diam R_{\eta(v)}&\leq NL^{\card(\mathcal{M}_s)}\diam R_{\eta(u'e)}\\
&\lesssim_{m,\epsilon} |h(x)-h(z)|,
\end{align*}
where the last inequality follows from \eqref{eq: quasiconvexity of T^eta(W)}. Thus $h$ is $H$-quasisymmetric with $H$ depending only on $m,\epsilon$.
\end{proof}

We are now ready to show Proposition \ref{prop: enumeration of T=enumeration of Tmeta}.

\begin{proof}[{Proof of Proposition \ref{prop: enumeration of T=enumeration of Tmeta}}]
Fix $m\in \{2,3,\dots\}$ and $\epsilon>0$. Claim (i) follows from Lemma \ref{lem: global homeo}. 

We show Claim (iii). Let $T$ as in the statement and let $f:T\to \mathcal{T}^{A_m^*}$ be a quasisymmetric embedding given from \cite[Theorem 1.3(i),(iii)]{CGIV}. The distortion of $f$ depends on the bounded turning constant of $T$, its uniform branch separation and $\epsilon$. Clearly, $T':=f(T)$ is a quasiconformal tree with uniform branch separation and valence at most $m$ as a connected closed subset of $\mathcal{T}^{A^*_m}$. 

There exists a longest word $w\in A_m^*$ such that $T' \subset \psi_w(\mathcal{T}^{A^*_m})$. Replacing $f$ by $f\circ \psi_w$, we may assume that there exist two distinct $i_1,i_2 \in A_m$ such that $\diam(T'\cap R_{i_1})>0$ and $\diam(T'\cap R_{i_2})>0$.
Let $h:\mathcal{T}^{A_m^*}\to \mathcal{T}^{A^*_m}$ be the bi-Lipschitz homeomorphism defined by
\begin{align*}
&h|\psi_{i_1}(\mathcal{T}^{A^*_m}) = \psi_1\circ\psi_{i_1}^{-1},\quad h|\psi_{1}(\mathcal{T}^{A^*_m}) = \psi_{i_1}\circ\psi_{1}^{-1},\\
&h|\psi_{i_2}(\mathcal{T}^{A^*_m}) = \psi_m\circ\psi_{i_2}^{-1},\quad h|\psi_{m}(\mathcal{T}^{A^*_m}) = \psi_{i_2}\circ\psi_{m}^{-1}\\
&h|\psi_i(\mathcal{T}^{A^*_m}) = \text{Id}\quad\text{for $i\not\in\{1,m,i_1,i_2\}$}.
\end{align*}
Composing $f$ with $h$, we may assume that $i_1=1$ and $i_2=m$.

Define now the set
\begin{equation}\label{eq:def of V}
    \mathcal{V}=\{w\in A_m^*: \diam(T'\cap R_w)>0\}\subset A^*_m.
\end{equation}
We show that $\mathcal{V}$ is an $m$-dictionary. 

By design of $f$, $\diam(T'\cap R_{1})>0$ and $\diam(T'\cap R_{m})>0$. Also, if $\diam(T'\cap R_w)>0$, then clearly $\diam(T'\cap R_{w^{\uparrow}})>0$. Therefore, properties (D1) and (D2) hold. Property (D3) is clear since $\mathcal{V} \subset A^*_m$. 

We now show (D4). Suppose now that $n\in\N$, that $w,u\in\mathcal{W}_n$ are distinct, and let $i,j\in \{1,m\}$ such that $R_{wi}\cap R_{uj}\neq \emptyset$. Then there exist $x\in (R_w\cap T')\setminus R_w^*$ and $y\in (R_u\cap T')\setminus R_u^*$. Moreover, by Lemma \ref{lem: properties of R_m},
\[ \emptyset \neq R_{wi}\cap R_{uj} \subset R_w \cap R_u \subset R_w^*\cap R_u^* \]
with the latter set containing only a single point $z$. In conjunction with Lemma \ref{lem:treearc}, we have that $z\in T'[x,y]\subset R_w\cup R_u$. Since $z\notin\{x,y\}$ there exists a subarc of $T'[x,y]$ that lies in $R_{wi}\cup R_{uj}$. We conclude that $wi,uj\in \mathcal{V}$ which gives (D4). 

We show that $T' =\mathcal{T}^{\mathcal{V}}$. Let $x,y\in T'$ distinct. Since $T'\subset \mathcal{T}^{A_m^*}$, \eqref{eq: T^m rhombus def} and the design of $f$ give $x\in(R_{wi}\cap T')\setminus R_{wi}^*$ or $y\in (R_{wj}\cap T')\setminus R_{wj}^*$ for some $i,j\in A_m$ distinct. By Lemma \ref{lem: geometry of rhombuses} and the claim above, we have $T'[x,y]\subset R_{wi}\cup R_{wj}$. Thus, $wi,wj\in \mathcal{V}$ and $x,y\in \mathcal{T}^{\mathcal{V}}$ which gives $T'\subset\mathcal{T}^\mathcal{V}$. For the reverse inclusion, let $y\in \mathcal{T}^\mathcal{V}$. By Lemma \ref{lem: pi functions}, $y=\bigcap_{n\in\N}R_{w\{n\}}$ where $w\{n\}\in\mathcal{V}_n$ for all $n\in\N$. Let $z_n\in T'\cap R_{w\{n\}}$. It follows that $z_n\to y$ and since $T'$ is closed, $y\in T'$. Hence, $T'=\mathcal{T}^{\mathcal{V}}$.

We now construct a dictionary function $\eta$ on $\mathcal{V}\cup\mathcal{V}_{\infty}$ in an inductive manner. Let $\eta(\e)=\e$. Assume now that for some integer $n\geq 0$ we have defined an injective function
\[ \eta : \bigcup_{k=0}^n\mathcal{V}_k \to \bigcup_{k=0}^nA_{m}^k\]
so that property (3) of Definition \ref{def:dictfun} holds. Fix $w\in \mathcal{V}_{n}$, let $x_w=f^{-1}(\psi_w(0))$ and let $\val(x_w)=p$. By Lemma \ref{lem:consistency of eta}, one of $w1,wm$ is in $\mathcal{V}_{n+1}$. 

\emph{Case 1}. Suppose first that both $w1,wm\in\mathcal{V}_{n+1}$. Let $D$ be a Jordan domain such that $D\cap f^{-1}(R_{wi})\neq\emptyset$ for all $i\in A_m$ such that $wi\in \mathcal{V}$. Recalling that $T'=\mathcal{T}^{\mathcal{V}}$ we have that $\bigcup_{i\in A_m}f^{-1}(R_{wi})$ is a tree. Hence, Lemma \ref{lem: maximal arcs}(i) asserts the existence of $y_{wi}\in f^{-1}(R_{wi})$ such that $f^{-1}(R_{wi})[x_w,y_{wi})\subset D$. Let now $y_{w{i_1}},\dots, y_{w{i_p}}$ be a counterclockwise enumeration along $\partial D$, starting from $y_{wi_1}$. By Lemma \ref{lem: maximal arcs}(ii) this enumeration is independent of the domain $D$. Without loss of generality, we may assume that $i_1=1$. Let $j\in \{1,\dots,p\}$ such that $i_{j}=m$. Let $i_{l_1}<\cdots<i_{l_{j}}$ and $i_{l_{j+1}}<\cdots<i_{l_p}$ be rearrangements of the indices $i_1,\dots,i_{j}$ and $i_{j+1},\dots,i_p$ in increasing order, respectively. For each $k\in \{1,\dots, j\}$ set $\eta(w{i_k})=\eta(w)i_{l_k}$ and for each $k\in \{j+1,\dots, p\}$ set $\eta(w{i_k})=\eta(w)({i_{l_k}}+m-2)$.

\emph{Case 2}. Suppose now that $w1\in\mathcal V_{n+1}$ but $wm\notin\mathcal V_{n+1}$. The construction is similar to Case 1, with the following changes. The counterclockwise enumeration starts with $i_1=1$, but choose $j$ so that $i_j=\max\{i:wi\in\mathcal V_{n+1}\}.$ We make the same assignments as above, except that we set $\eta(wi_j)=\eta(w)m.$

\emph{Case 3}. Suppose that $wm\in\mathcal V_{n+1}$ but $w1\notin\mathcal V_{n+1}$. The construction is similar to Case 1, except that the counterclockwise enumeration starts with $i_1=\min\{i:wi\in\mathcal V_{n+1}\}$. Choose $j$ so that $i_j=m$, and make the same assignments as above, except that $\eta(wi_1)=\eta(w)1.$

In both Case 2 and Case 3, these modifications fix every letter in $F_w$. Proceeding inductively, we define $\eta$ on all of $\mathcal{V}$. Finally, define $\eta$ on $\mathcal{V}_{\infty}$ by $\eta(i_1i_2\cdots) = j_1j_2\cdots$ so that for all $n\in\N$, $\eta(i_1\cdots i_n) = j_1\cdots j_n$. Applying an induction, it is easy to see that $\eta$ is a dictionary function. By Lemma \ref{lem:dictfun}(ii) $\eta(\mathcal{V})$ is an $m$-dictionary and by Lemma \ref{lem:dictqs}, there exists a quasisymmetric map $h:\mathcal{T}^{\mathcal{V}} \to \mathcal{T}^{\eta(\mathcal{V})}$ such that $h(\mathcal{T}^{\mathcal{V}}_w) = \mathcal{T}^{\eta(\mathcal{V})}_{\eta(w)}$ for each $w\in \mathcal{V}$. Set 
\[ g= h\circ f : T \to \mathcal{T}^{\eta(\mathcal{V})}.\]

For each $w\in \mathcal{V}$, let $T_w=g^{-1}(\cT^{\eta(\mathcal{V})}_{\eta(w)})$. Let $y_{wi_1},\dots,y_{wi_{\val(x_w)}}$ as in the statement. Without loss of generality, we may assume that $i_1=1$. Note that, by the construction of $\eta$, the words $\{\eta(wi)\}_{i\in A_{2m-2},\,wi\in\mathcal W}$ are enumerated according to increasing values of $i$, beginning with $\eta(w1)$. On the other hand, by Lemma \ref{lem:orientation of T2m-2}, the points $\{p_{wi}\}_{i\in A_{2m-2}}\subset \mathcal{T}^{2m-2}$ are also enumerated according to increasing values of $i$, beginning with $p_{w1}$. Hence, the corresponding points $\{z_{\eta(wi)}\}_{wi\in\mathcal W}$ inherit the same ordering: $z_{\eta(w1)},\dots, z_{\eta(wi_{\val(x_w)})}$. We conclude that $g$ is branch-orientation-preserving map. 

Finally, let $\mathcal W=\eta(\mathcal V)$, and define
\[V_n:=\bigcup_{k=0}^n\bigcup_{u\in A_{2m-2}^k}\psi_u([-1/2,1/2])\quad\text{and}\quad V:=\bigcup_{n\geq1}V_n.\]
Note that each $V_n$ is a tree consisting of finitely many line segments and that $\mathcal{T}^{2m-2}=\overline{V}$. 

We observe first that $\mathcal T^{2m-2}\setminus \mathcal{L}(\mathcal T^{2m-2})\subset V$. To see this fix $p\in\mathcal T^{2m-2}\setminus \mathcal{L}(\mathcal T^{2m-2})$. Two distinct components of $\mathcal T^{2m-2}\setminus\{p\}$ are open in $\mathcal T^{2m-2}$ and therefore contain points of $V$. The arc joining these points lies in some $V_n$ and contains $p$. 

Since $\mathcal{T}^\W\subset \mathcal{T}^{2m-2}$, we conclude that $\mathcal T^{\mathcal W}\setminus \mathcal L(\mathcal T^{\mathcal W})\subset V$. Let now $b\in \mathcal T^{\mathcal W}\setminus \mathcal L(\mathcal T^{\mathcal W})$ and let $B_1,B_2$ be the biggest (in diameter) component of $\mathcal T^{\mathcal W} \setminus\{b\}$. Then for $i\in \{1,2\}$ there exist $x_i\in B_{i}\cap V$. By the uniqueness of arcs and Lemma \ref{lem: metric trees properties of BM}(i), $b\in \mathcal{T}^\W(x_1,x_2)=V_n(x_1,x_2)$ for some $n\in\N$. In particular $\mathcal{T}^\W[x_1,x_2]$ is a polygonal arc. If $b$ is not a corner of this arc, a sufficiently short subarc containing $b$ is a line segment with $b$ in its interior. 

Suppose now that $b$ is a corner which means that $b=\psi_u(0)$ for some $u\in \W$. By the construction of $\eta$, $u1,um\in\mathcal W$, and both $\mathcal T^{\mathcal W}\cap R_{u1}$ and $\mathcal T^{\mathcal W}\cap R_{um}$ are nondegenerate. Choose non-leaf points in these two sets different from $b$. The arcs joining them to $b$ are polygonal and lie in the respective rhombuses. Since $b=\psi_{ui}(-1/2)$ for $i\in\{1,m\}$, their initial edges from $b$ lie in $\psi_{ui}([-1/2,1/2])$. Moreover, 
$$\psi_{u1}([-1/2,1/2]) \cup \psi_{um}([-1/2,1/2]) = \psi_u([-1/2,1/2])$$
with the latter being a line segment. Thus, choosing $x,y$ sufficiently close to $b$ on these respective arcs, $\mathcal T^{\mathcal W}[x,y]$ is a line segment containing $b$ in its interior.

Last, Claim (ii) follows identical arguments with claim (iii) with the difference that the topological embedding $f$ is given from \cite{Charatonik}.
\end{proof}

\section{Whitney-type decompositions for quasiconformal trees}\label{sec: Decomposition of R^2 minus T}
It is well-known that every open set $\Omega \subset \R^2$ (or generally in $\R^n$) admits a \emph{Whitney decomposition}: $\Omega = \bigcup_{n\in\N}Q_n$ where $\{Q_n : n\in\N\}$ is a collection of closed squares inside $\Omega$ with mutually disjoint interiors such that each square satisfies $\diam{Q_n} \simeq \dist(Q_n,\partial\Omega)$. See for example \cite[Theorem VI.1.1,Proposition VI.1.1]{Stein} for a rigorous construction.

The goal of this section is to prove the following proposition which allows us to decompose the exterior of planar quasiconformal trees into Whitney-like domains; these domains are not exactly squares but polygons which are finitely many up to similarities. Recall the \emph{tile decompositions} from Definition \ref{def:tile decomposition}.

\begin{prop}\label{prop:polygonal curves}
Let $T\subset \R^2$ be a $C_T$-bounded turning tree with $\diam{T}=1$ and let $({\bf X}^n)_{n\in\N}$ be a $C$-tile decomposition. There exist
\begin{enumerate}
\item an integer $N\in\N$ depending only on $C$, and a nonempty set $\mathscr{W} \subset\{1,\dots,N\}^*$,
\item positive integers $(k_X)_{X\in{\bf X}^*}$ such that $0<k_Y-k_X\lesssim_{C,C_T} 1$ for any $X\in{\bf X}^n$, $Y\in{\bf X}^{n+1}$, with $Y\subset X$,
\item a collection $\{x_w : w\in\mathscr{W}\setminus\{\varepsilon\}\}$ of double points of $T$ and a collection $\{\zeta_w:w\in\mathscr{W}\setminus\{\varepsilon\}\}$ of arcs,
\item a nested sequence of Jordan domains $\Omega_1\supset\Omega_2\supset\cdots\supset T$, 
\end{enumerate}
so that 
if $\d_X:=2^{-k_X}$ and $\mathscr{W}_n:=\{w\in \mathscr{W}:|w|=n\}$, then the next properties hold.
\begin{enumerate}[(i)]
\item (Properties of $\mathscr{W}$.)
\begin{enumerate}
\item For each $w\in \mathscr{W}$ there exists $N_w \in \{2,\dots,N\}$ such that $wi\in \mathscr{W}$ if and only if $i\in\{1,\dots,N_w\}$.
\item If $w\in\mathscr{W}\setminus\{\varepsilon\}$, then $w^{\uparrow} \in \mathscr{W}$.
\end{enumerate}      
\item (Properties of points $x_w$.)
\begin{enumerate}
\item For all $w\in \mathscr{W}\setminus\{\varepsilon\}$ and $n\in\N$, $x_w=x_{w1^{(n)}}$.
\item For all $w\in \mathscr{W}\setminus\{\varepsilon\}$ there exist unique distinct $X_w^1,X_w^2\in {\bf X}^{|w|}$ such that $\{x_w\}=X_w^1\cap X_w^2$.
\item For all $w\in \mathscr{W}\setminus\{\varepsilon\}$ there exists unique $u\in\mathscr{W}_{|w|} \setminus\{w\}$ such that $x_w=x_u$.
\item If $v,u\in \mathscr{W}$ with $|v|=|u|$ and $x_v\neq x_u$,  
$$|x_v-x_u| \gtrsim_{C_T,C} \max\{\d_{X_v^1},\d_{X_v^2},\d_{X_u^1},\d_{X_u^2}\}.$$
\item For each $x\in T$ and $n\in\N$, there exists $w\in \mathscr{W}_n$, and $X_w\in{\bf X}^n$ such that $|x-x_w| \lesssim_{C_T,C} \d_{X_w}$. 
\end{enumerate}

\item (Properties of arcs $\zeta_w$.)
\begin{enumerate}
\item For each $w\in \mathscr{W}\setminus\{\e\}$, $x_w$ is an endpoint of $\zeta_w$ and $\zeta_w \subset \overline{\Omega_{|w|}}$.
\item If $w\in \mathscr{W}\setminus\{\e\}$ and $m\geq |w|$, then $\zeta_w$ intersects $\partial\Omega_m$ at exactly one point which we denote by $p_{w1^{(m-|w|)}}$.
\item If $n\in\N$ and $w \in \mathscr{W}_n$, then $\zeta_{w1} = \zeta_w \cap \overline{\Omega_{n+1}}$, $\diam(\zeta_w\setminus\zeta_{w1}) \simeq_{C_T,C} 2^{-t_w}$ and $\zeta_w\setminus\zeta_{w1}\subset \mathcal{G}_{2^{-t_w}}$, where $t_w\in\N$ is chosen such that
\[2^{-t_w}<(16C)^{-1}\min\{\d_{X_{w1}^1},\d_{X^2_{w1}}\}\leq 2^{-t_w+1}.\]
\item For each $n\in\N$, the points $\{p_w:w\in\mathscr{W}_n\}$ are oriented on $\partial\Omega_n$ according to the cyclic lexicographic order of $\mathscr{W}_n$.
\item Let $w,u \in \mathscr{W}$ with $x_w\neq x_u$. If $v\in \mathscr{W}_t$ with $t=\max\{|w|+1,|u|+1\}$, then $$\dist(\zeta_w,\zeta_u)\gtrsim_{C_T,C} \max\{\d_{X_{v}^1},\d_{X_{v}^2},\d_{X_{v}^1},\d_{X_{v}^2}\}.$$
\end{enumerate}
\item (Properties of domains $\Omega_n$.) For each $n\in\N$, let $\G_n:=\partial\Omega_n$.
\begin{enumerate}
\item For each $z\in \G_n$ there exists $X\in{\bf X}^n$ such that $\d_X\lesssim_C\dist(z,T)< 8\d_X$.
\item Let $n\in\N$ and $w,u\in \mathscr{W}_n$ such that $u$ is to the left of $w$. There exists unique $j\in \{1,2\}$ such that $x_w,x_u\in \partial X_w^j$, $T(x_w,x_u)\subset X_w^j\setminus\partial X_w^j$, $x_{w2},\dots,x_{wN_w}\in X_w^j\setminus\partial X_w^j$, $\G_n[p_w,p_u]\subset \mathcal{G}_{\d_{Z_w}}$ where $Z_w\in {\bf X}^n$ such that $\d_{Z_w}=\min_{W\in{\bf X}^n,W\cap X_w^j\neq\emptyset}\d_W$,
and $\diam \G_n[p_w,p_u]\simeq_{C_T,C}\d_{Z_w}$.
\end{enumerate} 
\item (2-sidedness property.) Suppose that $w,u \in \mathscr{W}$ are distinct with ${|u|}\leq {|w|}$
and $x_w=x_u$. Let $0<r\leq \min\{\d_{X_w^1},\d_{X_w^2}\}$, and let $p_{r},q_{r}\in\partial B(x_w,r)\cap T$ such that $x_w\in T(p_{r},q_{r})\subset B(x_w,r)$. Denote by $\sigma_w$ and $\sigma_u$ the components of $B(x_w,r)\cap\zeta_{w}$ and $B(x_w,r)\cap\zeta_{u}$, respectively, that contain $x_w$. Then $\sigma_w$, $\sigma_u$ are in different components of $B(x_w,r)\setminus T[p_{r},q_{r}]$. Last, if $w$ is not initial subword of $u$, then
$\zeta_w\cap\zeta_u = \{x_w\}$.
\end{enumerate}
\end{prop}

For the rest of this section, we fix a $C_T$-bounded turning tree $T\subset \R^2$ and a $C$-tile decomposition $({\bf X}^n)_{n\in\N_0}$ with $\diam T=1$. Recall the nested family $({\bf V}^n)_{n\geq 0}$ of double points of $T$ associated with the tile decomposition $({\bf X}^n)_{n\geq 0}$. For each $n\in\N$ set ${\bf W}^n:={\bf V}^n\setminus{\bf V}^{n-1}$ and 
\[ {\bf V}^*:={\bf W}^*:=\bigcup_{n=1}^{\infty}{\bf V}^n=\bigcup_{n=1}^{\infty}{\bf W}^n.\] 
We say that $v\in {\bf V}^*$ is an \textit{$n$-vertex} if $v\in{\bf V}^n$ and we define the \emph{level}
\begin{equation}\label{eq:level}
\level(v) := \min\{n\in \N : v \in {\bf V}^n\},
\end{equation}
i.e., the unique $n\in\N$ such that $v\in {\bf W}^n$. Denote by ${\bf X}^*=\bigcup_{n\in\N_0}{\bf X}^n$. Lastly, if $X\in{\bf X}^n$, $n\in\N$, then by Lemma \ref{lem:properties of BM}(i) we have $\emptyset\neq\partial X\subset {\bf V}^n$. Hence, when $\card(\partial X)=1$, we call $X$ \textit{leaf tile} and when $\card(\partial X)\geq 2$, we call it \textit{branch tile}.

Let us briefly describe the geometric idea behind the proof of Proposition \ref{prop:polygonal curves}. Starting from the tile decomposition of $T$, we thicken the tiles at their corresponding scales.  By Lemma \ref{lem:def of thickening of Tm,eta}, these thickenings can be organized into nested Jordan domains $\Omega_n$ containing $T$.  We then associate to every double point, at each relevant scale, two John arcs of quantitatively controlled size which, roughly speaking, approach the point from two different sides of the tree; see Lemma \ref{lem:arzela curves}.  The main technical issue is to arrange these arcs consistently across different tiles and different scales.  The separation and compatibility estimates needed for this are established in Lemma \ref{lem:diameter of the sides of W.S} and Lemma \ref{lem:prerequisite for polygonal appr}.  After these estimates are in place, the arcs are replaced by suitable polygonal approximations which cross the relevant thickenings only once; see Lemma \ref{lem:polygonal curves}. The resulting polygonal curves, together with pieces of the boundaries of the thickenings, enclose the regions that we call Whitney squares.

\subsection{Square thickenings around $T$}\label{sec:sqthick}
Here we construct the Jordan domains $\Omega_n$. For each $n$-tile $X$, let  
\begin{equation}\label{eq:def of delta X}
 \d_X=2^{-k_{X}} \qquad\text{where} \qquad k_X\in\N\quad\text{and}\quad 2^{-k_X}<\frac{\diam X}{2^{15}{C^3_{T}}C^6}\leq 2^{-k_X+1}.
\end{equation}

\begin{rem}
By (iv) of Definition \ref{def:tile decomposition}, if $X\in{\bf X}^n$ and $Y\in{\bf X}^{n+1}$ with $Y\subset X$, then $2^{-k_X}\simeq_{C,C_T} 2^{-k_Y}$ which gives $0<k_Y-k_X\lesssim_{C,C_T} 1$. Moreover, by (viii) of Lemma \ref{lem:properties of BM}, $\lim_{n\to \infty} \sup_{X\in {\bf X}^n} \d_{X} = 0$.
\end{rem}

The next lemma gives the definition of domains $\Omega_n$ of Proposition \ref{prop:polygonal curves}.

\begin{lem}\label{lem:def of thickening of Tm,eta}
For each $n\in\N$, let $E_n=\bigcup_{X\in{\bf X}^n}\mathscr{T}_{\delta_X}(X)$. 
\begin{enumerate}[(i)]
\item The boundary $\partial E_n$ consists of finitely many mutually disjoint Jordan curves, each of which is a subset of $\mathcal{G}_{\d_n}$ where $\d_n=\min_{X\in{\bf X}^n}\d_X$. 
\item There exists a Jordan domain $\Omega_n$ such that $\partial\Omega_n\subset \partial E_n$ and $T\subset \interior E_n\subset \Omega_n$.
\item If $m>n$, then  $\overline{\Omega_{m}}\subset \Omega_{n}$.
\item If $X,Y$ are disjoint $n$-tiles, then $$\dist(\mathscr{T}_{\d_X}(X),\mathscr{T}_{\d_Y}(Y))\geq C^{-1}\max\{\d_X,\d_Y\}.$$
\item The set $E_n$ is $(2^6C_TC)$-bounded turning.
\end{enumerate}
\end{lem}
\begin{proof}
Fix $n\in\N$ and let
\[D_n=\{\delta_X:X\in\mathbf X^n\},\quad S_\delta=\bigcup_{\substack{X\in\mathbf X^n\\
\delta_X=\delta}}X.\]
By Lemma \ref{lem:properties of thickening}(iv), $E_n=\bigcup_{\delta\in D_n}\mathscr T_\delta(S_\delta).$ For each $\delta\in D_n$, Lemma \ref{lem:properties of thickening}(i) implies that $\partial\mathscr T_\delta(S_\delta)$ consists of finitely many mutually disjoint Jordan curves. Moreover, by the definition of square thickening, every corner of $\partial\mathscr T_\delta(S_\delta)$ has both coordinates equal to odd multiples of $\delta$. Since all the numbers $\delta_X$, $X\in\mathbf X^n$, are integer powers of 2, a corner of $\partial\mathscr T_\delta(S_\delta)$ cannot belong to $\partial\mathscr T_{\delta'}(S_{\delta'})$ when $\delta\neq\delta'$. Consequently, the union $E_n$ has no diagonal contacts between grid squares. Since $E_n$ is a finite union of squares of the $\delta_n$-grid, claim (i) follows.

To prove (ii), note that Lemma \ref{lem:properties of thickening}(ii) gives
\[T=\bigcup_{X\in\mathbf X^n}X\subset\operatorname{int}E_n.\]
Since each tile $X$ is connected, so is its square thickening $\mathscr T_{\delta_X}(X)$. As the tiles cover the connected tree $T$, the set $E_n$ is connected. Therefore, the outer boundary component of $E_n$ is a Jordan curve enclosing $\operatorname{int}E_n$. The Jordan domain $\Omega_n$ bounded by this curve has the required properties.

To show (iii), let $m>n$. By Lemma \ref{lem:properties of thickening}(v) $E_m \subset E_n$ which gives $\Omega_{m}\subset \Omega_{n}$. Let now $x\in \partial \Omega_{n}$
and $y\in \partial \Omega_{m}$. By (i), $x\in\mathscr{T}_{\d_X}(X)$ and $y\in \mathscr{T}_{\d_Y}(Y)$ where $X$ is an $n$-tile and $Y$ is an $m$-tile. By Lemma \ref{lem:properties of thickening}(ii) there exists $w_Y\in Y$ such that $|w_Y-y|<8\d_Y$. Moreover, by Lemma \ref{lem:properties of BM}(v), $Y$ is subset to a unique $n$-tile $Y'$. Since $x\in \partial \Omega_n$ it follows from (i) that $x\notin\interior(\mathscr{T}_{\d_{Y'}}(Y'))$. In conjunction with \eqref{eq:def of delta X} and Lemma \ref{lem:properties of thickening}(ii) we have
$$|x-y|\geq |x-w_Y|-|w_Y-y|\geq \d_{Y'}-8\d_Y>\d_{Y'}-16C^{-1}\d_{Y'}>0.$$
Thus, $\partial \Omega_n\cap \partial\Omega_m=\emptyset$, and hence $\overline{\Omega_m}\subset\Omega_n$.

To show (iv), let $X,Y$ be disjoint $n$-tiles. Let also $x\in\mathscr{T}_{\d_X}(X),y\in \mathscr{T}_{\d_Y}(Y)$. Without loss of generality suppose that $\d_X\geq \d_Y$. Let $w_x\in X, w_y\in Y$ be the points given by Lemma \ref{lem:properties of thickening}(ii) such that $|w_x-x|<8\d_X$ and $|w_y-y|<8\d_Y$. Since $X\cap Y=\emptyset$ it follows that $w_x\neq w_y$. Then, by \eqref{eq:def of delta X} and Definition \ref{def:tile decomposition}(iii) we have
\begin{align*}
|x-y|\geq |w_x-w_y|-|w_x-x|-|w_y-y|\geq \dist(X,Y)-8\d_X-8\d_Y&\geq \frac{\diam X}{C}-16\d_X\\ &>\frac{\d_X}{C}. 
\end{align*}

To show (v), let $x_1,x_2\in E_n$. For each $i\in \{1,2\}$, there exists $X_i\in{\bf X}^n$ such that $x_i\in \mathscr{T}_{\d_{X_i}}(X_i)$ and we may assume that $\diam{X_1}\leq \diam{X_2}$ which implies $\d_{X_1}\leq \d_{X_2}$. By Lemma \ref{lem:properties of thickening}(vi) there exists $x_i'\in X_i$ such that $[x_i,x_i']\subset \mathscr{T}_{\d_{X_i}}(X_i)$ and $|x_i-x_i'|< 9\d_{X_i}$. We consider two cases. 

\emph{Case 1.} Suppose that $|x_1-x_2|<\d_{X_1}$. For $i=1,2$, let $S_{i}\in\Sigma_{\d_{X_{i}}}$ with $x_i\in S_{X_i}\subset \mathscr{T}_{\d_{X_i}}(X_i)$. By \eqref{eq:def of delta X}, $\d_{X_1}$ divides $\d_{X_2}$ and therefore there exists $S'\in \Sigma_{\d_{X_1}}$ such that $x_2\in S'\subset S_{2}$. Since $|x_1-x_2|<\d_{X_1}$, we have that $S_{1}\cap S'\neq\emptyset$. If $S',S_{1}$ share a side, then $S'\cup S_{1}$ is a convex set and $[x_1,x_2]\subset E_n$. If $S',S_{1}$ do not share a side, then they share a corner $w$ and $\g=[x,w]\cup[w,y]\subset E_n$ and $\diam \g=|x_1-x_2|$.

\emph{Case 2.} Suppose that $|x_1-x_2| \geq \d_{X_{1}}$. We consider two subcases.

\emph{Case 2.1.} Suppose that $X_1\cap X_2 \neq \emptyset$. 
Set $\sigma:=[x_1,x_1']\cup T[x_1',x_2']\cup [x_2',x_2]$. By (i), $\sigma\subset E_n$ and by the bounded turning property of $T$ and Definition \ref{def:tile decomposition}(ii), as in Case 2.1,
\begin{align*}
\diam(\sigma)&\leq 9\d_{X_{1}}+C_T|x_1'-x_2'|+9\d_{X_{2}}\\ &\leq 9\d_{X_1}+C_T(|x_1-x_2|+9\d_{X_1}+18C\d_{X_1})+18C\d_{X_1}\\ &\leq 2^6CC_T|x_1-x_2|.
\end{align*}

\emph{Case 2.2}. Suppose that $X_1\cap X_2 = \emptyset$. Then, by Definition \ref{def:tile decomposition}(iii), 
\[ |x_1'-x_2'| \geq C^{-1}\diam{X_2} \geq 36\d_{X_2}\]
and, setting $\sigma:=[x_1,x_1']\cup T[x_1',x_2']\cup [x_2',x_2]$,
\[ |x_1-x_2| \geq |x_1'-x_2'| - 9(\d_{X_1} + \d_{X_2}) \geq \frac12|x_1'-x_2'| \]
while
\[ \diam{\sigma} \leq 9\d_{X_{1}}+C_T|x_1'-x_2'|+9\d_{X_{2}}\leq 2C_T|x_1'-x_2'|. \qedhere\]
\end{proof}

Here and for the rest of this section, for $n\in\N$, we write $\G_n = \partial\Omega_n$.

\subsection{John arcs}
Here we take the first step towards the construction of arcs $\zeta_w$. Recall the definition of $\level(x)$ form \ref{eq:level} and the definition of $\d_X$ from \eqref{eq:def of delta X}. Recall also that for each $x\in {\bf W}^*$ and $k\geq \level(x)$, by Lemma \ref{lem:properties of BM} there exist two unique distinct $k$-tiles $X_{x,k}^1,X_{x,k}^2$ that contain $x$. 

\begin{lem}\label{lem:arzela curves} 
There exists a family of $C_T$-John arcs
\[ \{\g_{x,k}^i : x\in {\bf W}^*, k\geq \level(x), \, i\in\{1,2\}\}\]
so that following are true for all $x\in{\bf W}^*$, $l\geq k\geq \level(x)$, and $i\in\{1,2\}$.
\begin{enumerate}[(i)]
\item One endpoint of $\g_{x,k}^i$ is $x$, the other is on $\G_k$, and the rest is in $\Omega_k \setminus T$.
\item The arc $\g_{x,l}^i$ is the closure of the component of $\g_{x,k}^i\cap\Omega_{l}$ that contains $x$.
\item We have $\ell(\g_{x,k}^i)\leq 16CC_T\min\{\d_{X_{1,x}^{k}},\d_{X_{2,x}^{k}}\}$. In particular, $$\g_{x,k}^i\subset B(x,(2C)^{-1}r_{x,k+1})\subset B(x,r_{x,k})$$
where
\[ r_{x,k}=2^{14}C^4C^2_{T}\min\{\d_{X_{x,k}^{1}},\d_{X_{x,k}^{2}}\}. \]
\item[(iv)] We have
\begin{align*}
\overline{B}(x,r_{x,k})\cap \G_k&\subset \mathscr{T}_{\d_{X_{x,k}^1}}(X_{x,k}^1)\cup \mathscr{T}_{\d_{X_{x,k}^2}}(X_{x,k}^2).
\end{align*}
\item[(v)] Let $0<r\leq r_{x,k}$ and $p_{r},q_{r}\in\partial B(x,r)\cap T$ such that $x\in T(p_{r},q_{r})\subset B(x,r)$. Denote by $U_{x,r}^1,U_{x,r}^2$ the connected components of $B(x,r)\setminus T[p_{r},q_{r}]$. The connected component of $B(x,r)\cap\g_{x,k}^i$ that contains $x$ lies in $U_{x,r}^i\cup\{x\}$. In particular, $\g_{x,k}^i\subset U_{x,r_{x,k}}^i\cup\{x\}$ and $\g_{x,k}^1\cap\g_{x,k}^2=\{x\}$. 
\item[(vi)] If $B_{x}^1,B_{x}^2$ are the components of $T\setminus\{x\}$, then there exist Jordan domains $O_1,O_2 \subset \Omega_k$ such that $B_{x}^i\subset O_i$, 
\[ \g_{x,k}^1\cup \g_{x,k}^2\cup\G_{k}=\partial O_1\cup \partial O_2, \quad\text{and}\quad\partial O_1\cap \partial O_2=\g_{x,k}^1\cup \g_{x,k}^2.\]
\item[(vii)] If $y\in {\bf W}^*\setminus\{x\}$, $m\geq \level(y)$, $j\in\{1,2\}$, and $t=\max\{m+1,k+1\}$ then 
\begin{equation*}
\dist(\g_{x,k}^i,\g_{y,m}^j) \geq(2C)\max\{ \d_{X_{x,t}^{1}},\d_{X_{x,t}^2}, \d_{X_{y,t}^1},\d_{X_{y,t}^2}\}
\end{equation*}
\item[(viii)] If $i'\in\{1,2\}\setminus\{i\}$ and $\omega_{x,k}^{i'}$ is the subarc of $\g_{x,k}^{i'}$ in $\overline{\Omega_{k}}\setminus\Omega_{k+1}$ with one endpoint in $\G_{k}$, then
\begin{equation*}
\dist(\g_{x,k}^i,\omega_{x,k}^{i'})\geq (2C)^{-1}\max\{\d_{X_{x,k+1}^{1}},\d_{X_{x,k+1}^{2}}\}. 
\end{equation*}
\end{enumerate}
\end{lem}

\begin{proof}[{Proof of Lemma \ref{lem:arzela curves}}]
For each $x\in{\bf W}^*$ and $m\geq \level(x)$, write 
\[ R_{x,m}=\frac14\min\{\diam X_{x,m}^1,\diam X_{x,m}^2\}.\] 
By Lemma \ref{lem:properties of BM}(viii), $R_{x,m}\to0$ as $m\to\infty$. 

Fix $x\in{\bf W}^*$, let $n=\level(x)$, and let $B_{x}^1,B_{x}^2$ be the components of $T\setminus \{x\}$. For $i=1,2$
we can find a point $p_{x}^i \in B_{x}^i\cap(\R^2\setminus\overline{B}(x,R_{x,n}))$. Let also $p_{x,m}^i$ be the first point when the arc $T[x,p_{x}^i]$ meets $\partial B(x,R_{x,m})$ as we travel from $x$ to $p_{x}^i$. Then the set $B(x,R_{x,m})\setminus T[p_{x,m}^1,p_{x,m}^2]$ has two connected components $W_{x,m}^1,W_{x,m}^2$. After relabeling, we may assume that $W_{x,m+1}^1\subset W_{x,m}^1$ and $W_{x,m+1}^2\subset W_{x,m}^2$ for all $m\geq n$. Note that $\bigcap_{m\geq n} \overline{W_{x,m}^i}=\{x\}$ for $i\in\{1,2\}$.

Fix, for $m\geq n$ and $i\in\{1,2\}$, points $y_{x,m}^i\in (W_{x,m}^i\setminus T)$. Fix $u\in\G_1$ and set $B:=B(u,25)$. It follows from Lemma \ref{lem:properties of thickening}(ii) and the fact that $\diam T=1$ that $T\subset \overline{B}$ and $y_{x,m}^i\in \overline{B}$ for all $m\geq n$ and $i\in\{1,2\}$. Finally, fix $w_0\in\R^2\setminus\overline{B}$ such that $\dist(w_0,\overline{B})>50C_{T}$. 

By Lemma \ref{lem: Nakki vaisalla}, there exist families of arcs $\{\sigma_{x,m}^i : m\geq n, \, i=1,2\}$ such that each $\sigma_{x,m}^i$ has endpoints $w_0$ and $y_{x,m}^i$ and satisfy \eqref{eq:John} with constant $C_{T}$. Let $\tilde{y}_{x,m}^i$ be the first time that the arc $\sigma_{x,m}^i$ meets $\partial \overline{B}$ as we travel from $y_{x,m}^i$ to $w_0$. Then for all $z\in\sigma_{x,m}^i[y_{x,m}^i,\tilde{y}_{x,m}^i]\subset \overline{B}$ we have that $\ell(\sigma_{x,m}^i[y_{x,m}^i,z])\leq \ell(\sigma_{x,m}^i[w_0,z])$ since otherwise 
$$\dist(w_0,\overline{B})\leq |w_0-z|\leq\ell(\sigma_{x,m}^i[w_0,z])\leq C_{T}\dist(z,T)\leq 50C_{T}$$ 
which is contradiction. Therefore, for each $i\in\{1,2\}$ we have 
$$\ell(\sigma_{x,m}^i[y_{x,m}^i,\tilde{y}_{x,m}^i])\leq C_{T}\dist(\tilde{y}_{x,m}^i,T)\leq 50C_{T}$$ 
which implies that the arcs $\{\sigma_{x,m}^i[y_{x,m}^i,\tilde{y}_{x,m}^i]\}_{m\geq n}$ are Lipschitz arcs with uniform bound on their lengths. By the Arzela-Ascoli Theorem, for $i=1,2$, there exist limit curves $\sigma_{x}^i$ connecting $x$ to a point $y_x^i\in\partial B$. 

Let $\g_{x}^i\subset\sigma_x^i$ be an arc that has $x$ and $y_{x}^i$ as endpoints. We claim that $\g_{x}^i$ is $C_{T}$-John arc. To see that, fix $z\in \g_{x}^i$. For each $m\geq n$, there exist $z_m\in\sigma_{x,m}^i[y_{x,m}^i,\tilde{y}_{x,m}^i]$ such that $z_m \to z$. As $x=\lim y_{x,m}^i$ and $y_x^i=\lim  \tilde{y}_{x,m}^i$, by the semicontinuity of length \cite{golab}
\begin{align*}
\dist(z,T) &= \lim_{m\to\infty}\dist(z_m,T)\\ 
&\geq C_{T}^{-1}\liminf_{m\to\infty}\min\{\ell(\sigma_{x,m}^i[z_m,y_{x,m}^i]),\ell(\sigma_{x,m}^i[z_m,\tilde{y}_{x,m}^i])\} \\
&\geq C_{T}^{-1}\min\{\ell(\g_{x}^i[x,z]),\ell(\g_{x}^i[y_x^i,z])\}.
\end{align*}

For each $x\in {\bf W}^*$, each $k\geq \level(x)$, and each $i\in\{1,2\}$, define $\g_{x,k}^i$ to be the closure of the connected component of $\gamma_{x}^i \cap \Omega_k$ that contains $x$. Let also $q_{x,k}^i$ be the endpoint of $\g_{x,k}^i$ on $\G_k$. Claim (ii) is immediate. 

To show (i), fix $x\in {\bf W}^*$, $k\geq \level(x)$, $i\in \{1,2\}$, and $z\in \g_x^i(x,q_{x,k}^i)$. For each $m\in\N$, there exists $w_m \in \sigma_{x,m}^i$ and $z_m \in \sigma_{x,m}^i[y_{x,m}^i,w_m]$, such that $z_m \to z$ and $w_m \to q_{x,k}^i$ as $m\to \infty$. By the $C_{T}$-John property and semicontinuity of length,  
\begin{align*}
0 < \min\{|z - x|, |z - q_{x,k}^i|\} &\leq \min\{\ell(\g_x^i[z, x]), \ell(\g_x^i[z, q_{x,k}^{i}])\} \\
&\leq \liminf_{m \to \infty} \min\{\ell([y_{m,i}, z_m]), \ell([w_{m,i}, z_m])\} \\
&\leq C_{T} \lim_{m \to \infty} \dist(z_m, T)\\ 
&= C_{T} \dist(z, T)
\end{align*}
which shows that $z\not\in T$. This completes the proof of (i).

To show (iii), fix $x\in {\bf W}^*$, $k\geq \level(x)$, and $i\in \{1,2\}$. By Lemma \ref{lem:properties of thickening}(ii) and the fact that $\diam T=1$ we get $\diam\G_1\leq 2$ and hence $\dist(\partial B,\G_1)>23$. By Lemma \ref{lem:def of thickening of Tm,eta}(i),(ii), $q_{x,k}^i\in\partial \mathscr{T}_{\d_Z}(Z)$ for some $Z\in{\bf X}^k$. Moreover, $\ell(\g_{x}^i[x,q_{x,k}^i]) \leq \ell(\g_x^i[q_{x,k}^i,y_x^i])$ since otherwise Lemma \ref{lem:properties of thickening}(ii) and Lemma \ref{lem:def of thickening of Tm,eta} give
\begin{align*}
23<\dist (\partial B,\G_1)\leq \dist (\partial B,\overline{\Omega_k})\leq |y_{x}^i-q_{x,k}^i|&\leq \ell(\g_x^i[q_{x,k}^i,y_x^i])\\ &\leq C_{T}\dist(q_{x,k}^i,T)\\ &\leq 8C_{T}\delta_Z
\end{align*}
which contradicts \eqref{eq:def of delta X}. Therefore, by the John property of $\g_x^i$ and Lemma \ref{lem:properties of thickening}(ii)
\begin{align*} 
\ell(\g_{x,k}^i) = \ell(\g_x^i[x,q_{x,k}^i]) = \min \{\ell(\g_x^i[x,q_{x,k}^i]), \ell(\g_x^i[y_{x}^i,q_{x,k}^i])\} &\leq C_T\dist(q_{x,k}^i,\t)\\
&\leq 8C_T\d_Z.
\end{align*}
Last, we claim that $Z\in\{X_{x,k}^1,X_{x,k}^2\}$. Assume, for a contradiction, that $Z\notin\{X_{x,k}^1,X_{x,k}^2\}$. By Lemma \ref{lem:properties of thickening}(ii) there exists $z\in Z$ such that $|q_{x,k}^i-z|\leq 8\d_Z$. If $Z\cap (X_{x,k}^1 \cup X_{x,k}^2)=\emptyset$, then by Definition \ref{def:tile decomposition}(iii) we have
\begin{align*}
8C_T\d_Z\geq \ell(\g_{x,k}^i)\geq |x-q_{x,k}^i|\geq |x-z|-|z-q_{x,k}^i|&\geq \dist (X_{x,k}^1,Z)-8\d_Z\\ 
&\geq C^{-1}\diam Z-8\d_Z
\end{align*}
which contradicts \eqref{eq:def of delta X}. If $Z\cap X_{x,k}^j\neq\emptyset $ for some $j\in \{1,2\}$, then Lemma \ref{lem:properties of BM}(ii) asserts the existence of $w\in T(x,z]\cap \partial X_{x,k}^j$. Definition \ref{def:tile decomposition}(ii), Lemma \ref{lem:properties of BM}(vii), and the bounded turning property of $T$ give
\begin{align*}
8C_T\d_Z\ge\ell(\g_{x,k}^i)\geq |x-q_{x,k}^i|\geq |x-z|-|z-q_{x,k}^i|&\geq \frac{\diam T[x,z]}{C_T}-8\d_Z \\ 
&\geq\frac{|x-w|}{C_T}-8\d_Z \\ 
&\geq\frac{\diam X_{x,k}^j}{C_TC^2}-8\d_Z\\ 
&\geq \frac{\diam Z}{C_TC^3}-8\d_Z,
\end{align*}
which contradicts \eqref{eq:def of delta X}. The proof the claim is complete. Last, by \eqref{eq:def of delta X} and Definition \ref{def:tile decomposition}(ii) we have that $\d_{X_{x,k}^{i}}< 2C\d_{X_{x,k}^j}$ for $i,j\in \{1,2\}$ distinct. Thus, claim (iii) follows.

To show (iv), fix $x\in{\bf W}^*$ and $k\geq\level(x)$. We claim that $B(x,r_{x,k})\cap \mathscr{T}_{\d_Z}(Z)=\emptyset$ for all $Z\in{\bf X}^{k}\setminus\{X_{x,k}^1,X_{x,k}^2\}$. Assuming the claim, (iv) follows from Lemma \ref{lem:def of thickening of Tm,eta}(i),(ii). To show the claim, fix $Z\in{\bf X}^{k}\setminus\{X_{x,k}^1,X_{x,k}^2\}$ and $z\in \mathscr{T}_{\d_Z}(Z)$. By Lemma \ref{lem:properties of thickening}(ii) there exists $z'\in Z$ such that $|z-z'|<8\d_Z$. Since $Z\neq X_{x,k}^1, X_{x,k}^2$, Lemma \ref{lem:properties of BM}(ii) implies that $T[z',x)\cap\partial X_{x,k}^j=\{v\}$ for some $j\in \{1,2\}$. Fix $y\in \overline{B}(x,r_{x,k})$. If $Z\cap X_{x,k}^j\neq\emptyset$, then by the bounded turning property of $T$, \eqref{eq:def of delta X}, Definition \ref{def:tile decomposition}(ii), and Lemma \ref{lem:properties of BM}(vii) we have
\begin{align*}
|z-y|\geq |z'-x|-|z-z'|-|x-y|&\geq\frac{\diam T[z',x]}{C_T}-8\d_Z-r_{x,k}\\ 
&\geq \frac{|x-v|}{C_T}-16C\d_{X_{x,k}^j}-r_{x,k}\\ 
&\geq \frac{\diam X_{x,k}^j}{C_TC^2}-16C\d_{X_{x,k}^j}-r_{x,k}>0.
\end{align*}
Suppose now that $Z\cap X_{x,k}^j=\emptyset$. Without loss of generality, we may assume that $\d_{X_{x,k}^j}\geq \d_Z$ Then by \eqref{eq:def of delta X} and Definition \ref{def:tile decomposition}(iii) we have
\begin{align*}
|z-y|\geq |z'-x|-|z-z'|-|x-y| &\geq \dist (Z,X_{x,k}^j)-8\d_Z-r_{x,k}\\ 
&\geq C^{-1}\diam X_{x,k}^j-16C\d_{X_{x,k}^j}-r_{x,k}>0.
\end{align*}

To show (v), fix $x\in {\bf W}^*$, $k\geq \level(x)$ and $i\in \{1,2\}$. For each $r\leq r_{x,k}$, Lemma \ref{lem:properties of BM}(vi) and \eqref{eq:def of delta X} give points $p_{r},q_{r}\in\partial B(x,r)\cap T$ such that $x\in T(p_r,q_r)\subset B(x,r)$. Denote by $U_{x,r}^1,U_{x,r}^2$ the connected components of $B(x,r)\setminus T[p_r,q_r]$. By (i), $\g_{x,k}^i\cap T=\{x\}$. Thus, if $\xi_{x,r}^i$ denotes the connected component of $B(x,r)\cap \g_{x,k}^i$ that contains $x$, then we may assume that $\xi_{x,r}^i\subset U_{x,r}^i\cup\{x\}$. Last, by design, $\xi_{x,r}^1\setminus\{x\}$ and $\xi_{x,r}^2\setminus\{x\}$ are located in distinct connected components of $B(x,r)\setminus T[p_r,q_r]$ which shows the first part of claim (v). By (iii), \eqref{eq:def of delta X}, and the definition of $r_{x,k}$ now we have that $\g_{x,k}^i\subset B(x,r_{x,k})$. Therefore, by (i) we may assume again that $\g_{x,k}^i\subset U_{x,r_k}^i\cup\{x\}$ which gives the second part of claim (v). Finally, since $U_{x,r_k}^1\cap U_{x,r_k}^2=\emptyset$ we have $\g_{x,k}^1\cap\g_{x,k}^2=\{x\}$.

To show (vi), fix $x\in {\bf W}^*$ and $k\geq \level(x)$. By (i) and (v), $\g_{x,k}^1\cup \g_{x,k}^2$ is an arc in $\overline{\Omega_k}$ that intersects $T$ at exactly $x$, has both endpoints on $\G_k$ and the rest of the arc (minus the endpoints and $x$) is in $\Omega_k \setminus T$. Therefore, $\Omega_k\setminus (\g_{x,k}^1\cup \g_{x,k}^2)$ contains exactly two components $O_1,O_2$ which are Jordan domains. Clearly $B_x^1\cup B_x^2 \subset O_1\cup O_2$. Assume for a contradiction that $B_x^1\cup B_x^2\subset O_1$. Let $r=\frac12\min_{Z\in{\bf X}^k}\d_Z$ and $q_1,q_2\in \partial B(x,r)\cap T$ such that $x\in T(q_1,q_2)\subset B(x,r)$. By Lemma \ref{lem:properties of thickening}(ii) and Lemma \ref{lem:def of thickening of Tm,eta}(ii), $B(x,r)\subset \Omega_k$. Thus, $q_{x,k}^1,q_{x,k}^2\notin B(x,r)$. For each $i\in\{1,2\}$, let $t_i\in\partial B(x,r)\cap\g_{x,k}^i$ be the first point of $\g_{x,k}^i$ that meets $\partial B(x,r)$ as we travel from $x$ to $q_{x,k}^i$. By (v), the arcs $\g_{x,k}^1(x,t_1),\g_{x,k}^2(x,t_2)$ are in different connected components of $B(x,r)\setminus T[q_1,q_2]$. Denote by $V_1,V_2$ the connected components of $B(x,r)\setminus (\g_{x,k}^1[x,t_1)\cup\g_{x,k}^2[x,t_2))$ and relabeling we may assume that $q_i \in \partial V_i$ and $V_i\subset O_i$. Fix $u_1\in \g_{x,k}^1(x,t_1)$ and $u_2\in \g_{x,k}^2(x,t_2)$. By Lemma \ref{lem:Jordan domain prerequisite} there exists arc $\sigma$ with endpoints $u_1,u_2$ and other than its endpoints, it is contained in $V_2$. But then, $T[q_2,x)\cap (\sigma_1\setminus\{u_1,u_2\})\neq\emptyset$ which contradicts $O_2\cap T=\emptyset$. Thus, relabeling, we may assume that $B_x^1\subset O_1$ and $B_x^2\subset O_2$, and (i) gives that $O_i\cap T=B_x^i$ for $i=1,2$.

To show (vii), fix $x,y\in {\bf W}^*$ distinct, $k\geq \level(x)$, $l\geq\level(y)$, and $i,j\in \{1,2\}$. Without loss of generality, we may assume that $l\geq k$. By Lemma \ref{lem:properties of BM}, $x,y\in{\bf V}^{l+1}$. Since $x\neq y$, Definition \ref{def:tile decomposition}(i) and Lemma \ref{lem:properties of BM}(ii) assert that there exist $i_1,i_2\in\{1,2\}$ such that $X_{x,l+1}^{i_1}\cap X_{y,l+1}^{i_2}=\emptyset$. Without loss of generality, suppose that $X_{x,l+1}^1\cap X_{y,l+1}^1=\emptyset$ and that for some $a\in \{1,2\}$, 
$$\d_{X_{x,l+1}^a}=\max\{\d_ {X_{x,l+1}^1},\d_{X_{x,l+1}^2},\d_ {X_{y,l+1}^1},\d_{X_{y,l+1}^2}\}.$$
Fix also $z_1\in \g_{x,k}^i[x,q_{x,k}^i]$ and $z_2\in \g_{y,l}^j[y,q_{y,l}^j]$. 
By Definition \ref{def:tile decomposition}(ii),(iii) we have $|x-y|\geq \diam X_{x,l+1}^a/{C^2}$ and by (iii) and \eqref{eq:def of delta X}, $|z_2-y|\leq 16CC_T\d_{X_{x,l+1}^a}$. Last, recall that the arc $\g_{x,k}^i$ is John arc and thus 
$$|x-z_1|\leq \ell(\g_{x,k}^i[x,z_1])\leq \dist(z_1,T)\leq |z_1-z_2|+|z_2-y|.$$
It follows
\begin{align*}
|z_1-z_2|\geq |x-y|-|x-z_1|-|y-z_2|&\geq \frac{\diam X_{x,l+1}^{a}}{C^2}-|z_1-z_2|-32C_TC\d_{X_{x,l+1}^a}.
\end{align*}
We conclude that \eqref{eq:def of delta X} gives that $|z_1-z_2|\geq 2C\d_{X_x,l+1}^a$.

To show (viii), let $i'\in\{1,2\}\setminus\{i\}$ and let $p,q\in \partial B(x,r_{x,k+1})\cap T$ such that $x\in T(p,q)\subset B(x,r_{x,k+1})$. Denote by $U_{x}^1,U_{x}^2$ the connected components of $B(x,r_{x,k+1})\setminus T[p,q]$. By (iii),(v) we may assume that $\g_{x,k}^i\subset U_{x}^i\cup\{x\}$ and that $\g_{x,k}^{i'}\subset U_{x}^{i'}\cup\{x\}$. Let $z_1\in\g_{x,k}^i$ and $z_2\in\omega_{x,k}^{i'}$ where $\omega_{x,k}^{i'}$ is the subarc of $\g_{x,k}^{i'}$ in $\overline{\Omega_{k}}\setminus\Omega_{k+1}$ with one endpoint in $\G_{k}$. Since $B(x,r_{x,k+1})$ is a convex set, the line segment $[z_1,z_2]$ must intersect $T[p,q]$ at some point $u$. In conjunction with the fact that $z_2\notin \Omega_{k+1}$ it follows that there exists $p'\in [z_2,u]\cap \G_{k+1}\cap B(x,r_{x,k+1})$. By (iv), Lemma \ref{lem:def of thickening of Tm,eta}(i),(ii), and the choice of $r_{x,k+1}$ we have that $p'\in \partial \mathscr{T}_{\d_{X_{x,k+1}^{j_1}}}(X_{x,k+1}^{j_1})$ and that $u\in X_{x,k+1}^{j_2}$ for $j_1,j_2\in\{1,2\}$. By Lemma \ref{lem:def of thickening of Tm,eta}(i),(ii), $p'\notin \interior(\mathscr{T}_{\d_{X_{x,k+1}^{j_2}}}(X_{x,k+1}^{j_2}))$ and by Lemma \ref{lem:properties of thickening}(ii),
\[ |z_1-z_2|\geq |p'-u|> \d_{X_{x,k+1}^{j_2}}\geq (2C)^{-1}\max\{\d_{X_{x,k+1}^1},\d_{X_{x,k+1}^2}\}.\qedhere \]
\end{proof}

\subsection{Estimates within the thickenings of $T$} For all $n\in\N$, $x\in{\bf V}^n$ and $i\in\{1,2\}$ denote by $q_{x,n}^i$ the unique point of $\g_{x,n}^i\cap \G_n$ by Lemma \ref{lem:arzela curves}(i). Define also 
$$D_n=\{q_{x,n}^i: x\in {\bf V}^n \text{ and } i\in \{1,2\}\}.$$ 
If $x=q_{v,n}^i \in D_n$, we say that $v$ is the \textit{associated} $n$-vertex of $x$. Note that $\card D_n=2\card {\bf V}^n$ and by Definition \ref{def:tile decomposition}(i), $4\leq\card D_n< \infty$. 

\begin{rem}\label{rem: points in gamma n disjoint}
By Lemma \ref{lem:def of thickening of Tm,eta}(iii) for $n\neq m$ we have that $\G_n\cap\G_m=\emptyset$ which implies that $D_n\cap D_m=\emptyset$.    
\end{rem}   

The next lemma estimates the diameter of the subarcs of $\G_n$ between adjacent points in $D_n$. Recall the definition of tile boundary from \eqref{eq:tileboundary}.

\begin{lem}\label{lem:diameter of the sides of W.S1}
Let $n\in \N$, let $x,y\in D_n$ be adjacent, and let $v_x,v_y$ denote the associated $n$-vertices. 
\begin{enumerate}[(i)]
\item If $v_x=v_y$, then there exists a unique leaf tile $X\in {\bf X}^n$ such that $v_{x}\in\partial X$.
\item If $v_x\neq v_y$, then there exists a unique tile $X\in {\bf X}^n$ such that $v_{x},v_y\in\partial X$.   
\end{enumerate}
\end{lem}

\begin{proof}
By Lemma \ref{lem:arzela curves}(i),(v),(vii) the set
$$T[v_x,v_y]\cup \g_{v_y,n}^i\cup\G_n[x,y]\cup \g_{v_x,n}^j$$
is a Jordan curve for $i,j\in\{1,2\}$. Let $O$ be the Jordan domain enclosed by it. By Lemma \ref{lem:arzela curves}(i), $O\cap \G_n=\emptyset$ and $\partial O \cap T \neq \emptyset$. Therefore, by Lemma \ref{lem:def of thickening of Tm,eta}, $O\subset \Omega_n$.

Suppose first that $v_x=v_y=v$. By Lemma \ref{lem:arzela curves}(vi), $O\cap T$ is one of the two components of $T\setminus\{v\}$. Moreover, by Lemma \ref{lem:properties of BM}(ii),(iii), $v$ belongs to exactly two distinct $n$-tiles; one being in $O\cap T$ and the other being in $T\setminus O$. Denote by $X$ the one of these two tiles that is in $O\cap T$. We claim that $X$ is a leaf tile. Suppose, otherwise, that there exists a point $w\in \partial X\setminus\{v\}$. Since $O\subset \Omega_n$, $\g_{w,n}^t\cap \partial O\neq\emptyset$ for some $t\in\{1,2\}$. By Lemma \ref{lem:arzela curves}(i),(vi),(vii) and the fact that $x,y$ are adjacent we reach a contradiction. Hence, $X$ is a leaf tile.

Suppose now that $v_x\neq v_y$. Let $X,Y$ be the unique $n$-tiles that contain $v_x$. Assume for a contradiction that $v_y\in  T\setminus(X\cup Y)$. By Lemma \ref{lem:properties of BM}(iii) there exists $u\in T(v_x,v_y]\cap (\partial X\cup\partial Y)$. Without loss of generality, we may assume that $u\in\partial X$. Since $T$ is a tree, $u$ is unique. Let $r=\frac12\min_{Z\in{\bf X}^n}\d_Z$ and let $p,q\in T[v_x,v_y]\cap \partial B(u,r)$ such that $u\in T(p,q)\subset B(u,r)$. Lemma \ref{lem:properties of thickening}(ii) and Lemma \ref{lem:def of thickening of Tm,eta} give $\overline{B}(u,r)\cap \G_{n}=\emptyset$. Moreover, if $z_1\in\overline{B}(u,r)$ and $z_2\in \g_{v_x,n}^i$, then by Lemma \ref{lem:properties of BM}(vii), Lemma \ref{lem:arzela curves}(iii), and \eqref{eq:def of delta X} we have 
$$|z_1-z_2|\geq |v_x-u|-|z_1-u|-|z_2-v_x|\geq \frac{\diam X}{C^2}-\frac{\d_X}{2}-16CC_ T\d_X>0.$$
If $z_2\in \g_{v_y,n}^j$, then let $Z$ be the $n$-tile such that $v_y\in \partial Z$ and $u'\in T[v_x,v_y)\cap \partial Z$. Lemma \ref{lem:properties of BM}(vii), Lemma \ref{lem:arzela curves}(iii), \eqref{eq:def of delta X}, and the bounded turning property of $T$ give
\begin{align*}
|z_1-z_2|\geq |v_y-u|-|z_1-u|-|z_2-v_y|&\geq \frac{\diam T[u,v_y]}{C_T}-r-16CC_T\d_Z\\ &\geq \frac{|u'-v_y|}{C_T}-\frac{\d_Z}{2}-16CC_T\d_Z\\ &\geq \frac{\diam Z}{C^2C_T}-17CC_T\d_Z\\ &>0.
\end{align*}
Let now $p\in \partial B(u,r)\cap T[u,v_x]$ and $q\in \partial B(u,r)\cap T[u,v_y]$ such that $u\in T(p,q)\subset B(u,r)$. For each $l\in\{1,2\}$, let $\sigma_l$ be the component of $B(u,r)\cap \g_{x,n}^l$ that contains $x$. By Lemma \ref{lem:arzela curves}(v), $\sigma_1,\sigma_2$ lie in different connected components of $B(u,r)\setminus T[p,q]$. Similar arguments with the proof of Lemma \ref{lem:arzela curves}(vi) assert now that for some $j_0\in \{1,2\}$, $\sigma_{j_0}\setminus\{x\}\subset O$ and thus by Lemma \ref{lem:arzela curves}(i), $\g_{u,n}^{j_0}\cap \partial O\neq\emptyset$. By Lemma \ref{lem:arzela curves}(i),(vii) and the fact that $x,y$ are adjacent, we reach a contradiction.  
\end{proof}

\begin{rem}
If $x,y,v_x,v_y$ are as in Lemma \ref{lem:diameter of the sides of W.S1} and $v_x\neq v_y$, then the tile $X$ given by the lemma, can not be a leaf tile.
\end{rem}

\begin{lem}\label{lem:diameter of the sides of W.S}
Let $n\in \N$, let $x,y\in D_n$ be adjacent, let $v_x,v_y$ denote the associated $n$-vertices, and let $X$ be the $n$-tile given by Lemma \ref{lem:diameter of the sides of W.S1}.
\begin{enumerate}[(i)] 
\item If $O$ is the domain enclosed by the Jordan curve 
$$T[v_x,v_y]\cup \g_{v_y,n}^i\cup\G_n[x,y]\cup \g_{v_x,n}^j,$$ 
then $O\subset \Omega_n$ and $O\cap T\subset X\setminus\partial X$. In particular, if $X$ is a leaf tile, then $O\cap T=X\setminus\partial X$.
\item For each $w\in\partial X$ and for each $n$-tile $Z$ with $Z\cap X\neq\emptyset$, 
$$ B(w,R_{X,n})\subset \mathscr{T}_{\d_Z}(B(w,R_{X,n}))\subset B(w,2R_{X,n})$$
where
\[ R_{X,n}=2^{13}C^4C^2_{T}\min\{\d_Z : Z\in {\bf X}^n, Z\cap X\neq\emptyset\}.\]
\item We have 
\[ \G_n[x,y]\subset \mathscr{T}_{\d_X}(X\setminus \bigcup_{w\in\partial X}B(w,R_{X,n}))\cup \bigcup_{w\in\partial X} B(w,2R_{X,n}).\] 
Moreover, if $W\in {\bf X}^n$ with $W\cap X\neq \emptyset$ and $\d_W=\min_{Z\in {\bf X}^n, Z\cap X\neq\emptyset}\d_Z$, then 
$$(2C)^{-1}\d_W\leq \diam \G_n[x,y]\leq 2C\diam W.$$ 
\end{enumerate} 
\end{lem}

\begin{proof}
For (i), suppose that there exists $y\in O\cap (Z'\cup\partial X\setminus\{v_x,v_y\})$ for some $n$-tile $Z'\neq X$. Let $w\in \partial X$ be the first time that we meet $X$ as we travel from $y$ to $T[v_x,v_y]$. By Lemma \ref{lem:properties of BM}(vi) and Lemma \ref{lem:arzela curves}(vi) it follows that $w\in O$. Similar arguments with Case 1 reach a contradiction. We conclude that $O\cap  T\subset X\setminus\partial X$. Note that by Definition \ref{def:tile decomposition}(i) $\inf_{Z\in{\bf X}^n, Z\cap X\neq\emptyset}\d_X$ is minimized.  

Claim (ii) can be easily verified using Definition \ref{def:tile decomposition}(ii), Lemma \ref{lem:properties of thickening}(ii), and Lemma \ref{lem:arzela curves}(iii).

To show (iii), let $t\in\G_n[x,y]$ such that $t\in\mathscr{T}_{\d_Z}(Z)$ for some $n$-tile $Z\neq X$. Suppose first that $Z\cap X=\emptyset$. Without loss of generality suppose that $\d_Z\leq \d_X$. By Lemma \ref{lem:properties of thickening}(ii) there exists $t'\in Z$ such that $|t-t'|<8\d_Z$. Let $\alpha\subset \mathscr{T}_{\d_Z}(Z)$ be a curve connecting $t$ to $t'$. By Lemma \ref{lem:properties of BM}(vi) and Lemma \ref{lem:arzela curves}(vi),(vii) there exists an arc $\zeta\subset \G_n$ connecting $q_{w,n}^1$ to $q_{w,n}^2$ for some $w\in\partial X$ so that $Z\setminus\{w\}$ lies inside the domain enclosed by the curve $\g_{w,n}^1\cup\zeta\cup\g_{w,n}^2$. Hence, the fact that $O\subset\Omega_n$, and (ii) show that $\alpha\cap B(w,R_{X,n})\neq\emptyset$. But if $z_1\in\mathscr{T}_{\d_Z}(Z)$ and $z_2\in B(w,R_{X,n})$, then by Definition \ref{def:tile decomposition}(iii), the choice of $R_{X,n}$, and \eqref{eq:def of delta X} we have 
$$|z_1-z_2|\ge \dist (X,Z)-R_{X,n}-\d_Z\geq C^{-1}\diam X-R_{X,n}-\d_X>0$$
which is a contradiction.

Suppose now $Z\cap X\neq\emptyset$. We may assume that $t\in \mathscr{T}_{\d_Z}(Z\setminus\bigcup_{w\in\partial X}B(w,R_{X,n}))$ since by (iii) and Lemma \ref{lem:properties of thickening}(iii),(iv) we have
$$\mathscr{T}_{\d_Z}(Z)\subset \mathscr{T}_{\d_Z}(Z\setminus\bigcup_{w\in\partial X}B(w,R_{X,n}))\cup \bigcup_{w\in\partial X}B(w,2R_{X,n}).$$
By Lemma \ref{lem:properties of thickening}(ii) there exists $t'\in Z\setminus\bigcup_{w\in\partial X}B(w,R_{X,n})$ such that $|t-t'|<8\d_Z$. By Lemma \ref{lem:def of thickening of Tm,eta}(ii),(v) there exists a $(2^6C_TC)$-bounded turning curve $\omega\subset\overline{\Omega_n}$ connecting $t$ to $t'$. Similar arguments with the case of $X\cap Y=\emptyset$ show that $\omega\cap B(w,R_{X,n})\neq\emptyset$. Since $t\in\omega\cap (\overline{\Omega_n}\setminus B(w,2R_{X,n}))$, Lemma \ref{lem:properties of thickening}(ii) and Definition \ref{def:tile decomposition}(ii) give
$$8C\d_X\geq 4\d_Z>\frac12|t-t'|\geq \frac{\diam \omega}{2^7C_TC}\geq \frac{\dist(B(w,2R_{X,n}),B(w,R_{X,n})}{2^7C_TC}\geq \frac{R_{X,n}}{2^7C_TC}$$
which is contradiction from the definition of $R_{X,n}$ and \eqref{eq:def of delta X}. Finally, by (ii) and Lemma \ref{lem:properties of thickening}(iii),(iv) we have
\begin{align*}
\G_n[x,y]&\subset \mathscr{T}_{\d_X}(X)\cup\bigcup_{w\in\partial X}B(w,2R_{X,n})\\
&=\mathscr{T}_{\d_X}\left(X\setminus \bigcup_{w\in\partial X}B(w,R_{X,n})\right)\cup \bigcup_{w\in\partial X} B(w,2R_{X,n}).
\end{align*}

It remains to estimate $\diam\G_n[x,y]$. The upper bound follows from Definition \ref{def:tile decomposition}(ii), Lemma \ref{lem:properties of thickening}(ii), and \eqref{eq:def of delta X}. For the lower bound, suppose that $W\in {\bf X}^n$ such that $\d_W=\min_{Z\in {\bf X}^n, Z\cap X\neq\emptyset}\d_Z$. Suppose also that $X$ is a leaf tile and that $v_x=v_y=v$. By (i) we have $\diam O\geq \diam X$. But $\diam O=\diam \partial O$ and thus Lemma \ref{lem:arzela curves}(iii) gives
$$\diam \partial O\leq \diam \G_n[x,y]+\diam \g_{v,n}^1+\diam\g_{v,n}^2\leq \diam \G_n[x,y]+32CC_T\d_X.$$
By \eqref{eq:def of delta X} we conclude that 
$$\diam\G_n[x,y]\geq \diam X-32CC_T\d_X\geq \d_X.$$
Last, if $X$ is a branch tile, then $v_x\neq v_y$ and by Lemma \ref{lem:properties of BM}(vii), Lemma \ref{lem:arzela curves}(iii), and \eqref{eq:def of delta X} we have
\[|x-y|\geq |v_x-v_{y}|-|v_{x}-x|-|v_{y}-y|\geq\frac{\diam X}{C^2}-32CC_{T}\d_X\geq \d_X. \qedhere\]
\end{proof}

The following lemma is a key ingredient to prove Proposition \ref{prop:polygonal curves}.
\begin{lem}\label{lem:prerequisite for polygonal appr}
Let $n\in\N$, $x\in{\bf V}^n$, and let $X_{x,n+1}^{1}, X_{x,n+1}^{2}$ be the unique distinct $(n+1)$-tiles that contain $x$.   
\begin{enumerate}[(i)]
\item If $i\in\{1,2\}$ and $w_1,w_2\in D_{n+1}$ are the points on $D_{n+1}$ adjacent to $q_{x,{n+1}}^i$, then 
$$\g_{x,n}^i\cap \G_{n+1}\subset \G_{n+1}(w_1,q_{x,{n+1}}^i]\cup\G_{n+1}[q_{x,{n+1}}^i,w_2).$$
In particular, if $y_1,y_2\in D_{n+1}\setminus\{q_{x,{n+1}}^i\}$ are adjacent, then
$$\dist(\g_{x,n}^i,\G_{n+1}[y_1,y_2])\geq (2C)^{-1}\max\{\d_{X_{x,n+1}^{1}},\d_{X_{x,n+1}^{2}}\}.$$ 
\item Let $y\in {\bf V}^n$ and let $z,w\in D_{n+1}$ be points adjacent to $q_{y,{n+1}}^j,q_{x,{n+1}}^i$, respectively, where $i,j\in\{1,2\}$ are distinct if $x=y$. If $v_z,v_w$ are the associated $(n+1)$-vertices of $z,w$ respectively, then there exists $(n+1)$-tiles $X,Y$ such that $v_z,x\in \partial X$, $v_w,y\in \partial Y$ and 
$$\dist(\G_{n+1}[z,q_{y,{n+1}}^j],\G_{n+1}[w,q_{x,{n+1}}^i])\geq (2C)^{-1}\max\{\d_{X},\d_Y\}.$$
\end{enumerate}
\end{lem}
\begin{proof}

To show (i), fix $n\in\N$ and $x\in{\bf V}^n$. Let $w_{1},w_{2}$ be the adjacent points of $q_{x,n+1}^i$ for some $i\in\{1,2\}$. Fix $w\in \G_{n+1}[y_1,y_2]$ for some $y_1,y_2\in D_{n+1}\setminus\{q_{x,n+1}^i\}$ adjacent and fix $z\in \g_{x,n}^i$ such that $|w-z|=\dist(\g_{x,n}^i,\G_{n+1}[y_1,y_2])$. Note that 
\[ w\in \G_{n+1}\setminus(\G_{n+1}(w_1,q_{x,n+1}^i]\cup\G_{n+1}[q_{x,n+1}^i,w_2)).\] 
Let also $v_{y_1},v_{y_2}$ be the associated $(n+1)$-vertices of $y_1,y_2$, respectively, and let $Y$ be an $(n+1)$-tile such that $v_{y_1},v_{y_2}\in\partial Y$, given from Lemma \ref{lem:diameter of the sides of W.S1}. By Lemma \ref{lem:properties of BM}(ii),(iii), $x\in {\bf V}^{n+1}$, and for $k\geq n$ there exist unique distinct $k$-tiles $X_{x,k}^{1},X_{x,k}^{2}$ that contain $x$. Let 
\[ R_{n+1}=2^{13}C^4C^2_{T}\min\{\d_{X_{x,n+1}^1},\d_{X_{x,n+1}^2}\}.\] 

We claim that $|w-z|\geq(2C)^{-1}\max\{\d_{X_{x,n+1}^1},\d_{X_{x,n+1}^2}\}$. Assuming the claim, since $\G_{n+1}\subset \Omega_{n}$ from Lemma \ref{lem:def of thickening of Tm,eta} and $\g_{x,n}^i\subset \overline{\Omega_{n}}$ from Lemma \ref{lem:arzela curves}(i) it follows that 
$$\g_{x,n}^i\cap \G_{n+1}\subset \G_{n+1}(w_1,q_{x,{n+1}}^i]\cup\G_{n+1}[q_{x,{n+1}}^i,w_2).$$
To that end, we consider the following cases.

\emph{Case 1:} Suppose that $w\notin B(x,2R_{n+1})$. Note that $w\in\Omega_n$ from Lemma \ref{lem:def of thickening of Tm,eta} and $z\in \overline{\Omega_n}$ from Lemma \ref{lem:arzela curves}(i). By Lemma \ref{lem:def of thickening of Tm,eta}(ii),(v) there exists a $(2^6C_TC)$-bounded turning curve $\alpha\subset \overline{\Omega_n}$ that connects $w$ to $z$. Note that Lemma \ref{lem:arzela curves}(iii) asserts that $z\in \g_{x,n}^i\subset B(x,2R_{n+1})$. Definition \ref{def:tile decomposition}(ii),(iv), the choice of $R_{n+1}$, and \eqref{eq:def of delta X} give
\begin{align*}
|w-z|&\geq (2^6CC_T)^{-1}\diam \alpha\\
&\geq (2^6CC_T)^{-1}\dist(\partial B(x,2R_{n+1}),\partial B(x,R_{n+1}))\\
&= (2^6CC_T)^{-1}R_{n+1}\\
&>\min\{\d_{X_{x,n+1}^1},\d_{X_{x,n+1}^2}\}.
\end{align*}

\emph{Case 2:} Suppose that $w\in B(x,2R_{n+1})$ and that $x\in\{v_{y_1},v_{y_2}\}$. Without loss of generality, we may assume that $x=v_{y_1}$ and hence $y_1=q_{x,n+1}^j$ for some $j\in\{1,2\}\setminus\{i\}$. By Lemma \ref{lem:properties of BM}(ii),(iii),(iv), we have that $Y\in \{X_{x,n+1}^1,X_{x,n+1}^2\}$ and that $Y$ is not a leaf tile. Hence, Lemma \ref{lem:diameter of the sides of W.S1} gives that $x\neq v_{y_2}$. Without loss of generality, suppose that $Y=X_{x,n+1}^1$.

Suppose now that $y_2\in\g_{v_{y_2},n+1}^{i'}$ for some $i'\in\{1,2\}$ and consider the following Jordan curve (see the proof of Lemma \ref{lem:diameter of the sides of W.S1})
$$\g_{v_{y_2},n+1}^{i'}\cup T[v_{y_2},x]\cup\g_{x,n+1}^j\cup\G_{n+1}[q_{x,n+1}^j,y_2].$$
Denote by $U$ the domain enclosed by it. By Lemma \ref{lem:diameter of the sides of W.S}(i), $U\subset \Omega_{n+1}$ and by Lemma \ref{lem:arzela curves}(i),(v),(vi),(vii) and the fact that $y_2$ is adjacent to $q_{x,n+1}^j$ we have that $q_{x,n+1}^i\notin \overline{U}$. By Lemma \ref{lem:def of thickening of Tm,eta}(ii),(v) there exists a $(2^6C_TC)$-bounded turning curve $\alpha\subset \overline{\Omega_{n+1}}$ connecting $w$ to $q_{x,n+1}^i$. If there exists $t\in\alpha$ such that $t\notin B(x,2R_{n+1})$, then Lemma \ref{lem:arzela curves}(iii), the choice of $R_{n+1}$, and \eqref{eq:def of delta X} give
\begin{align*}
|w-z|&\geq |w-q_{x,n+1}^i|-|q_{x,n+1}^i-z|\\
&\geq (2^6C_TC)^{-1}\diam\alpha-\ell(\g_{x,n}^i)\\ 
&\geq (2^6C_TC)^{-1}R_{n+1}-16C_TC\min\{\d_{X_{x,n}^1},\d_{X_{x,n}^2}\}\\ 
&\geq 2^5C_TC\min\{\d_{X_{x,n+1}^1},\d_{X_{x,n+1}^2}\}-16C_TC\min\{\d_{X_{x,n}^1},\d_{X_{x,n}^2}\}\\ 
&\geq (2C)^{-1}\max\{\d_{X_{x,n+1}^1},\d_{X_{x,n+1}^2}\}.
\end{align*}
Hence, for the rest of this case we may assume that $\alpha\subset B(x,2R_{n+1})$. Observe also that $\g_{v_{y_2},n+1}^{i'}\cap B(x,2R_{n+1})=\emptyset$. To see this, let $z_1\in \g_{v_{y_2},n+1}^{i'}$ and $z_2\in B(x,2R_{n+1})$. Lemma \ref{lem:properties of BM}(vii), Lemma \ref{lem:arzela curves}(iii), the choice of $R_{n+1}$, and \eqref{eq:def of delta X} give
\begin{align*}
|z_1-z_2|&\geq|v_{y_2}-x|- |z_1-v_{y_2}|-|z_2-x|\\
&\geq C^{-2}\diam X_{x,n+1}^1-16CC_T\d_{X_{x,n+1}^1}-2R_{n+1}\\
&\ge (2C)^{-1}\max\{\d_{X_{x,n+1}^1},\d_{X_{x,n+1}^2}\}.
\end{align*}

Let $t$ be the first time that $\alpha $ meets $\partial U\setminus\G_{n+1}[q_{x,n+1}^j,y_2]$ as we travel from $w$ to $q_{x,n+1}^i$. By the above observations, $t\in T[x,v_{y_2}]\cup\g_{x,n+1}^j$. Let 
\[p\in \partial B(x,2R_{n+1})\cap  T[v_{y_2},x]\quad \text{and}\quad q\in \partial B(x,2R_{n+1})\cap X_{x,n+1}^2\] 
such that $x\in T(p,q)\subset B(x,2R_{n+1})$. Denote by $W_1,W_2$ the connected components of $B(x,2R_{n+1})\setminus T[p,q]$. By Lemma \ref{lem:arzela curves}(v), without loss of generality we may assume that $q_{x,n+1}^i\in W_i$ and $q_{x,n+1}^j\in W_j$. We may assume also that $w\in W_i$ since otherwise $[w,z]\cap T\neq\emptyset$ and the estimate of $|w-z|$ follows similarly as in the proof of Lemma \ref{lem:arzela curves}(viii). If $t\in \g_{x,n+1}^j\setminus\{x\}$, then $\alpha[w,t]\subset W_j$ since $\alpha[w,t]\cap T[p,q]=\emptyset$, and $\g_{x,n+1}^j\subset W_j\cup \{x\}$, we reach a contradiction. If $t\in T(p,x]$, then by Lemma \ref{lem:Jordan domain prerequisite} there exists an arc $\zeta$ with one endpoint $p_1\in T[p,t)$ and another endpoint $p_2\in T[x,t)$ so that $\zeta\setminus\{p_1,p_2\}\subset U\cap W_i$. Since $q_{x,n}^i\notin \overline{U}$, let $\omega$ be an arc that connects $q_{x,n}^i$ to $t$ so that $\omega\setminus\{t\}\subset \R^2\setminus\overline{U}$. It follows that $\omega\cap (\zeta\setminus\{p_1,p_2\})\neq\emptyset$ which is contradiction since $\zeta\setminus\{p_1,p_2\}\subset U$, and $\omega\setminus\{t\}\subset \R^2\setminus\overline{U}$.

Suppose now that $t\in T(t_1,t_2)$ where $t_1,t_2\in T[p,v_{y_2})\cap\partial B(x,2R_{n+1})$ and $t\in T(t_1,t_2)\subset B(x,2R_{n+1})$. Denote by $V_1,V_2$ the connected components of $B(x,2R_{n+1})\setminus T[t_1,t_2]$. By Lemma \ref{lem:arzela curves}(i), we have that $\g_{x,n}^i,\g_{x,n}^j\subset V_{i_1}$ for some $i_1\in \{1,2\}$. Similar arguments now with the above paragraph show that $w\in V_{i_2}$ for $i_2\in \{1,2\}\setminus\{i_1\}$. Thus, $[w,z]\cap T\neq\emptyset$, and the estimate of $|w-z|$ follows similarly as in the proof of Lemma \ref{lem:arzela curves}(viii). See Figure \ref{fig:different connected components of a ball}.
\begin{figure}
    \centering
    \includegraphics[width=0.7\linewidth]{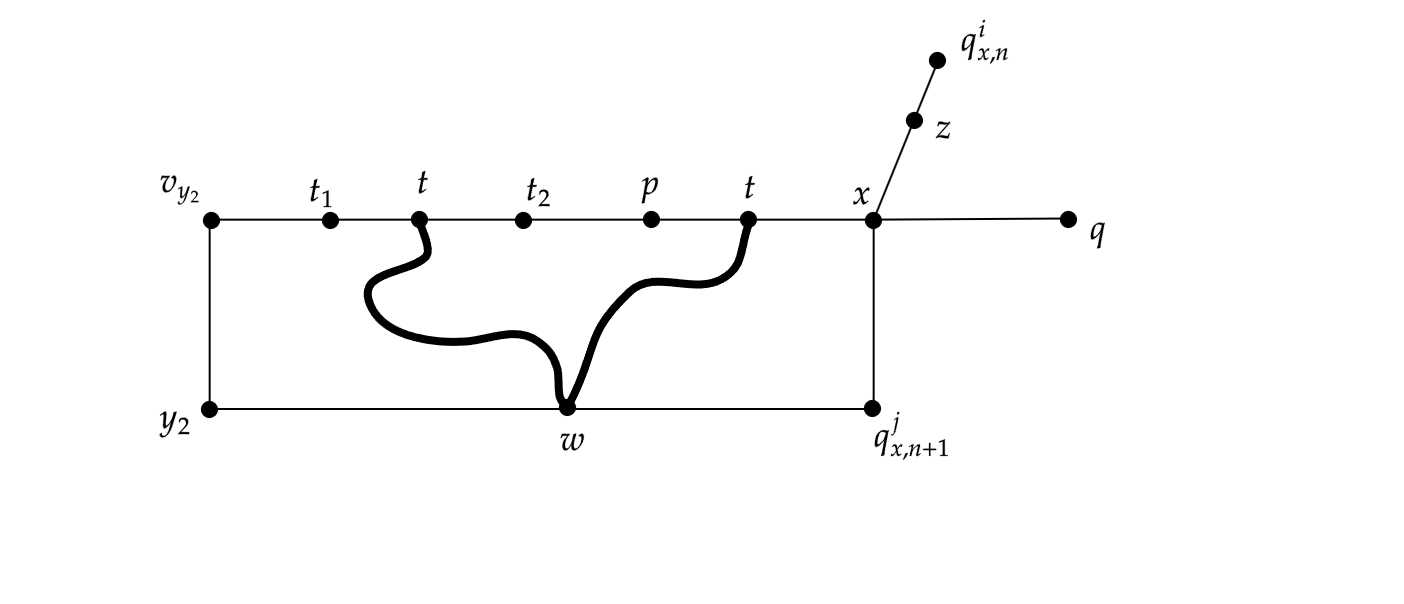}
    \caption{Here the line segment $[w,z]$ intersects $T$.}
    \label{fig:different connected components of a ball}
\end{figure}

\emph{Case 3:} Suppose now that $w\in B(x,2R_{n+1})$ and that $x\notin\{v_{y_1},v_{y_2}\}$. Since $q_{x,n+1}^i\in \G_{n+1}$, Lemma \ref{lem:def of thickening of Tm,eta}(ii),(v) asserts the existence of a $(2^6C_TC)$-bounded turning curve $\alpha\subset \overline{\Omega_{n+1}}$ that connects $w$ to $q_{x,n+1}^i$. Similarly as in Case 2 we may assume that $\alpha\subset B(x,2R_{n+1})$.

For each $k\in\{1,2\}$ suppose that $y_k\in\g_{v_{y_k},n+1}^{i_k}$ where $i_1,i_2\in\{1,2\}$ are distinct if $v_{y_1}=v_{y_2}$. Denote by $O$ the Jordan domain enclosed by the curve
$$T[v_{y_1},v_{y_2}]\cup\g_{v_{y_1},{n+1}}^{i_1}\cup\G_{n+1}[y_1,y_2]\cup \g_{v_{y_2},n+1}^{i_2}.$$
By Lemma \ref{lem:diameter of the sides of W.S}(i) we have $O\subset \Omega_{n+1}$ and by Lemma \ref{lem:arzela curves}(i),(vii), and the fact that $y_1$ is adjacent to $y_2$, we have $q_{x,n+1}^i\notin \overline{O}$.  

We claim that $\overline{B}(x,2R_{n+1})\cap \g_{v_{y_k},n+1}^{i_k}=\emptyset$ for each $k\in \{1,2\}$. Suppose first that $Y\cap (X_{x,n+1}^1\cup X_{x,n+1}^2)=\emptyset$. Without loss of generality, we may assume that $\d_Y\leq \min\{\d_{X_{x,n+1}^1},\d_{X_{x,n+1}^2}\}$. By Definition \ref{def:tile decomposition}(iii), \eqref{eq:def of delta X}, and the choice of $R_{n+1}$,
\begin{align*}
\dist(\g_{v_{y_k,n+1}}^{i_k}, \overline{B}(x,2R_{n+1}))&\geq \dist(Y,X_{x,n+1}^1)-2R_{n+1}-16C_TC\d_Y\\ 
&\ge C^{-1}\diam X_{x,n+1}^1-2R_{n+1}-16C_TC\d_{X_{x,n+1}^1}>0.
\end{align*}
Suppose now that $Y\cap X_{x,n+1}^l\neq\emptyset$ for some $l\in \{1,2\}$. Traveling from $v_{y_k}$ to $x$ along $T[v_{y_k},x]$, let $u\in \partial X_{x,n+1}^l\cap T[v_{y_k},x)$. Definition \ref{def:tile decomposition}(ii), Lemma \ref{lem:properties of BM}(vii), \eqref{eq:def of delta X}, and the choice of $R_{n+1}$ give
\begin{align*}
\dist(\g_{v_{y_k,n+1}}^{i_k}, \overline{B}(x,2R_{n+1}))&\geq (C_T)^{-1}\diam T[x,v_{y_k}]-2R_{n+1}-16C_TC\d_Y\\ 
&\geq (C_T)^{-1}|u-x|-2R_{n+1}-16C_TC\d_Y\\ 
&\geq(C_TC^2)^{-1}\diam X_{x,n+1}^l-2R_{n+1}-32C_TC^2\d_{X_{x,n+1}^l}>0.
\end{align*}
The proof of the claim is complete. 

It follows now with similar arguments as in Case 2 that $w,z$ are in different connected components of $B(x,2R_{n+1})\setminus T[t_1,t_2]$ for some $t_1,t_2\in \partial B(x,2R_{n+1})$ such that $T(t_1,t_2)\subset B(x,2R_{n+1})$. The estimate of $|w-z|$ follows similarly now as in the proof of Lemma \ref{lem:arzela curves}(viii).

To show (ii), fix $n\in\N$, and $x,y\in{\bf V}^n$. Let $w,z\in D_{n+1}$ be adjacent points of $q_{x,{n+1}}^i,q_{y,{n+1}}^j$, respectively, where $i,j\in\{1,2\}$ are distinct if $x=y$. Moreover, let $v_z,v_w$ be the associated $(n+1)$-vertices of $z,w$ respectively. By Lemma \ref{lem:diameter of the sides of W.S1} there exist $(n+1)$-tiles $X,Y$ such that $v_z,x\in \partial X$ and $v_w,y\in \partial Y$. 

Suppose first that $x\neq y$. By Lemma \ref{lem:properties of BM}(ii),(iv) it follows that $X\cap Y=\emptyset$. Let $u_1\in \G_{n+1}[w,q_{x,n+1}^i]$ and $u_2\in \G_{n+1}[z,q_{y,n+1}^j]$. By Lemma \ref{lem:properties of thickening}(iii), Lemma \ref{lem:diameter of the sides of W.S}(iii), 
\[ u_1\in \mathscr{T}_{\d_X}(X)\cup \bigcup_{v\in \partial X}B(v,2R_{X,n+1}), \quad u_2\in \mathscr{T}_{\d_Y}(Y)\cup\bigcup_{v\in \partial Y}B(v,2R_{Y,n+1})\] 
where $R_{X,n+1},R_{Y,n+1}$ are as in the Lemma \ref{lem:diameter of the sides of W.S}(ii). If $u_1\in\mathscr{T}_{\d_X}(X)$ and $u_2\in \mathscr{T}_{\d_{Y}}(Y)$, then by Lemma \ref{lem:def of thickening of Tm,eta}(iv), $|u_1-u_2|\geq (2C)^{-1}\max\{\d_X,\d_Y\}$. Suppose now that $u_1\in B(v,2R_{X,n+1})$, and $u_2\in B(v',2R_{Y,n+1})$ for some $v\in \partial X$ and $v'\in \partial Y$. Assume as we may that $\d_X\geq \d_Y$. By Definition \ref{def:tile decomposition}(iii), \eqref{eq:def of delta X} and the choice of $R_{X,n+1},R_{Y,n+1}$ we have 
\begin{align*}
|u_1-u_2|\geq \dist(X,Y)-2R_{X,n+1}-2R_{Y,n+1}&\geq\frac{\diam X}{C}-2^{14}C^4C_T^2\d_X\\ &>\frac{\max\{\d_X,\d_Y\}}{2C}.
\end{align*}
The rest of the cases follow similarly.

Suppose now that $x=y$. Then $i\neq j$, and $X\cap Y\neq\emptyset$. We claim that 
\[\G_{n+1}[w,q_{x,n+1}^i]\cap\G_{n+1}[z,q_{y,n+1}^j]=\emptyset.\] 
Assuming the claim, it follows from Lemma \ref{lem:properties of thickening}(iii), Lemma \ref{lem:arzela curves}(iv), Lemma \ref{lem:diameter of the sides of W.S}(iii) that $$\G_{n+1}[z,q_{x,n+1}^j]\subset \bigcup_{Z\in {\bf X}^{n+1}, Z\cap X\neq\emptyset}\mathscr{T}_{\d_Z}(Z)$$
and 
$$\G_{n+1}[w,q_{x,n+1}^i]\subset \bigcup_{W\in {\bf X}^{n+1}, W\cap Y\neq\emptyset}\mathscr{T}_{\d_W}(W).$$
In particular, by construction of $\G_{n+1}$, arcs $\G_{n+1}[w,q_{x,n+1}^i]$ and $\G_{n+1}[z,q_{x,n+1}^j]$ are in $\mathcal{G}_{\lambda}$ where $\lambda=\min_{Z\in {\bf X}^{n+1},Z\cap(X\cup Y)\neq\emptyset}\d_Z$. Hence, by Definition \ref{def:tile decomposition}(ii), \eqref{eq:def of delta X}, and the above facts we have
$$\dist(\G_{n+1}[z,q_{y,{n+1}}^j],\G_{n+1}[w,q_{x,{n+1}}^i])\geq \lambda\geq \frac{\max\{\d_{X},\d_Y\}}{2C}.$$
We show the claim. By Lemma \ref{lem:arzela curves}(vi),(viii), and the fact that $\G_{n+1}$ is a Jordan curve, it is enough to show that $w\neq z$. Suppose then that $w=z$. By Lemma \ref{lem:properties of BM}(vi) and Lemma \ref{lem:arzela curves}(vi),(vii), the set
$$\g_{x,{n+1}}^i\cup\G_{n+1}[q_{x,{n+1}}^i,w]\cup\G_{n+1}[w,q_{x,{n+1}}^j]\cup\g_{x,{n+1}}^j$$
is a Jordan curve enclosing a Jordan domain $X\subset O\subset \Omega_{n+1}$. Therefore, if $\g_{v_w,n+1}^{k_0}$ is the arc that does not contain $w$, then Lemma \ref{lem:arzela curves}(i),(vii) asserts that $\g_{v_w,n+1}^{k_0}\cap \partial O\neq\emptyset$. We reach a contradiction from Lemma \ref{lem:arzela curves}(i),(vii) and the fact that the points $q_{x,{n+1}}^i,w$ and $q_{x,{n+1}}^j,w$ are adjacent. 
\end{proof}

\subsection{Polygonal approximations} In the following lemma we approximate the John curves found in Lemma \ref{lem:arzela curves} with polygonal ones.
\begin{lem}\label{lem:polygonal curves}
     There exists a family of arcs
\begin{align*}
\{\zeta_{x,k}^i: x\in {\bf W}^*, k\geq \level(x), i\in\{1,2\}\}
\end{align*} 
so that the following are true for all $x\in {\bf W}^*$, $k\geq \level(x)$, $i\in\{1,2\}$, and for $X_{x,k}^1,X_{x,k}^2$ that are the unique distinct $k$-tiles that contain $x$.
\begin{enumerate}[(i)]
\item One endpoint of $\zeta_{x,k}^i$ is $x$, the other is on $\G_k$, which we call $p_{x,k}^i$, the rest of the curve is in $\Omega_k\setminus T$, and if $l\geq k$, then $\zeta_{x,k}^i\cap \G_l=\{p_{x,l}^i\}$.
\item If $l\ge k$, then $\zeta_{x,l}^i=\zeta_{x,k}^i[x,p_{x,l}^i]$. 
\item Let $0<r\leq \min\{\d_{X_{x,k}^1},\d_{X_{x,k}^2}\}$ and $p_{r},q_{r}\in\partial B(x,r)\cap T$ such that $x\in T(p_{r},q_{r})\subset B(x,r)$. Denote by $U_{x,r}^1,U_{x,r}^2$ the connected components of $B(x,r)\setminus T[p_{r},q_{r}]$. The connected component of $B(x,r)\cap \zeta_{x,k}^i$ that contains $x$ lies in $U_{x,r}^i\cup\{x\}$. 
\item We have $\zeta_{x,k}^i[p_{x,k}^i,p_{x,k+1}^i]\subset\mathcal{G}_{\lambda_{k}}$ where $\lambda_k=2^{-l_k}$ and $l_k\in\N$ is chosen such that $2^{-l_k}<(16C)^{-1}{\min\{\d_{X_{x,k+1}^1},\d_{X_{x,k+1}^2}\}}\leq 2^{-l_k+1}$.
\item We have $\diam\zeta_{x,k}^i[p_{x,k}^i,p_{x,k+1}^i]\simeq_{C_T,C}\min\{\d_{X_{x,k+1}^1},\d_{X_{x,k+1}^2}\}$. 
\item Suppose that $p_{y,k}^j$ is adjacent of $p_{x,k}^i$ in $\G_k$ for some $y\in {\bf W}^k$ and $j\in \{1,2\}$. There exists $X\in {\bf X}^k$ such that $x,y\in \partial X$, $T(x,y)\subset X\setminus\partial X$, and $\G_k[p_{x,k}^i,p_{y,k}^j]\subset \mathcal{G}_{\d_Z}$ where $Z\in{\bf X}^n$ such that $\d_Z=\min_{W\in{\bf X}^n,W\cap X\neq\emptyset}\d_W$. Last, we have that $\diam \G_k[p_{x,k}^i,p_{y,k}^j]\simeq_{C_T,C}\d_Z$.
\item We have $\zeta_{x,k}^1\cap\zeta_{x,k}^2=\{x\}$.
\item Let $y\in {\bf W}^*\setminus\{x\}$, $l\geq \level(y)$, $j\in\{1,2\}$, and $t=\max\{l+1,k+1\}$. Let also $Y_{y,t}^1,Y_{y,t}^2$ be the unique distinct $t$-tiles that contain $y$. Then $$\dist(\zeta_{x,k}^i,\zeta_{y,m}^j)\gtrsim_{C_T,C}\max\{ \d_{X_{x,t}^1},\d_{X_{x,t}^2}, \d_{Y_{y,t}^1},\d_{Y_{y,t}^2}\}.$$
\end{enumerate}
\end{lem}

\begin{proof}
Let $x\in {\bf W}^*$, $n=\level(x)$, and $i\in \{1,2\}$. Also, for each $m\geq n$, let $X_{x,m}^1,X_{x,m}^2$ be the $m$-tiles such that $\{x\}=X_{x,m}^1\cap X_{x,m}^2$, given from Lemma \ref{lem:properties of BM}(ii),(iii). Recall the arc $\g_{x,n}^i$ from Lemma \ref{lem:arzela curves}. For each $k\geq n$ let $q_{x,k}^i$ be the endpoint of the arc $\g_{x,k}^i$ on $\G_k$. By Lemma \ref{lem:arzela curves}(ii) we have $\g_{x,n}^i=\bigcup_{k=n}^{\infty}\g_{x,k}^i[q_{x,k}^i,q_{x,k+1}^i]$. We replace each arc $\g_{x,k}^i[q_{x,k}^i,q_{x,k+1}^i]$ by a polygonal one satisfying the desired properties. We do this by induction on $k$. 
 
Suppose first that $k=n$. Let $w$ be the first time that the arc $\g_{x,n}^i$ meets $\G_{{n+1}}$ as we travel from $q_{x,n}^i$ to $q_{x,n+1}^i$. By Lemma \ref{lem:def of thickening of Tm,eta}(ii),(iii) we have $O_n:=\Omega_{{n}}\setminus\overline{\Omega_{{n+1}}}\neq\emptyset$. Hence, Lemma \ref{lem:arzela curves}(i) gives that $\g_{x,n}^i(q_{x,n}^i,w)\subset O_n$. See Figure \ref{fig:polygonal arcs step 1}.
\begin{figure}[H]
\centering
\includegraphics[width=0.5\linewidth]{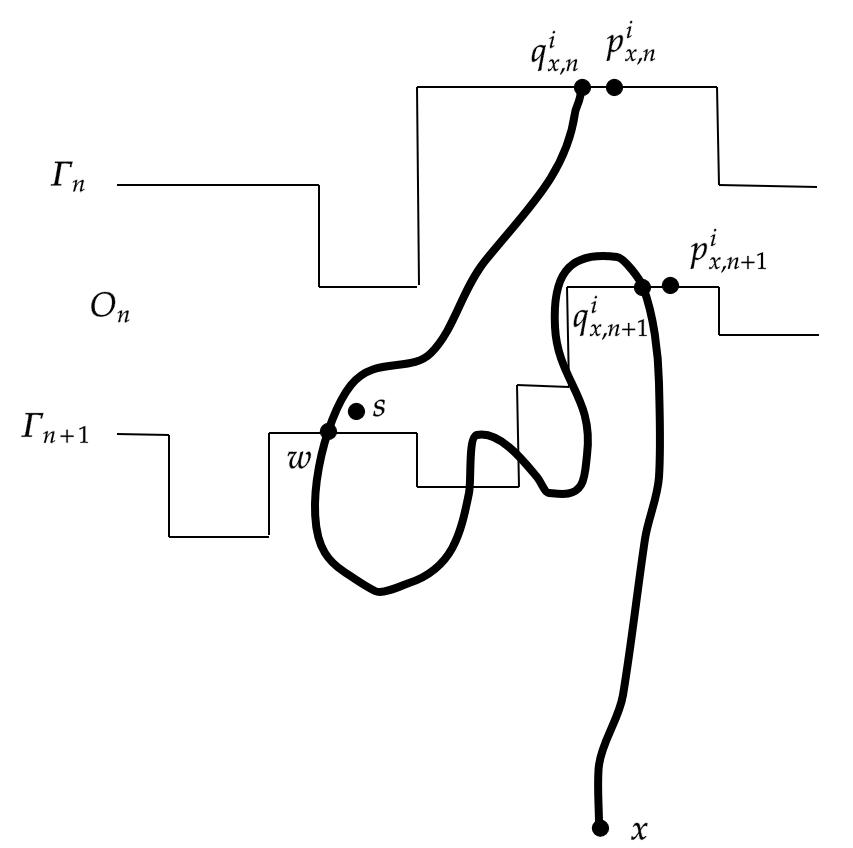}
\caption{}
\label{fig:polygonal arcs step 1}
\end{figure}

First, we replace the arc $\g_{x,n}^i[q_{x,n}^i,w]$. Let $\lambda_{n}$ as in the statement and consider the square grid $\mathcal{G}_{\lambda_{n}}$ and the associated family of (closed) squares $\Sigma_{\lambda_{n}}$. Let 
$$\mathcal{A}=\{S\in \Sigma_{\lambda_{n}}\cap\overline{O_n}:S\cap \g_{x,n}^i[q_{x,n}^i,w]\neq\emptyset\}.$$
Recall the scales $r_{x,l}$, $l\geq n$ from Lemma \ref{lem:arzela curves}. For simplicity, write $r_l$. We claim that 
\begin{align}\label{eq:local distance of Gn,Gn+1}
\dist(\G_n\cap B(x,r_n),\G_{n+1}\cap B(x,r_{n+1}))>4\lambda_{n}.
\end{align}
Assuming the claim, since by Lemma \ref{lem:arzela curves}(iii), $\g_{x,n}^i\subset B(x,r_{x,n+1})\cup B(x,r_{x,n})$, it follows that $\mathcal{A}\neq\emptyset$. To show the claim, let $z_1\in \G_{n+1}\cap B(x,r_{n+1})$ and $z_2\in \G_n\cap B(x,r_n)$. By Lemma \ref{lem:arzela curves}(iv), $z_1\in \mathscr{T}_{\d_{X_{x,n+1}^j}}(X_{x,n+1}^j)$ and $z_2\in \mathscr{T}_{\d_{X_{x,n}^{j'}}}(X_{x,n}^{j'})$ for some $j,j'\in \{1,2\}$. Lemma \ref{lem:properties of thickening}(ii) asserts the existence of $z\in X_j^{n+1}$ such that $|z_2-z|<8\d_{X_{x,n+1}^j}$. Lemma \ref{lem:properties of thickening}(ii) and \eqref{eq:def of delta X} now give
$$|z_1-z_2|\geq |z_2-z|-|z_1-z|\geq \d_{X_{x,n}^{j'}}-8\d_{X_{x,n+1}^j}>4\lambda_{n}.$$
The proof of the claim is complete.

By Lemma \ref{lem:arzela curves}(iv), $\G_n\cap B(x,r_{n})\subset \mathcal{G}_{\min\{\d_{X_{x,n}^1},\d_{X_{x,n}^2}\}}$ and by \eqref{eq:def of delta X}, $\lambda_n$ divides $\min\{\d_{X_{x,n}^1},\d_{X_{x,n}^2}\}$. Thus, there exists $S_1\in\mathcal{A}$ such that $q_{x,n}^i\in \partial S_1$. Let $\ell_1$ be a side of $\partial S_1$ such that $\ell_1$ meets $\G_n$ at a unique point, which we call $p_{x,n}^i$, and $\ell_1\setminus\{p_{x,n}^i\}\subset O_n$ from \eqref{eq:local distance of Gn,Gn+1}. Let $t\in\ell_1$ be the endpoint of $\ell_1$ on $O_n$. The point $p_{x,n}^i$ is the replacement of the point $q_{x,n}^i$. The fact that $\G_{n+1}\cap B(x,r_{n+1})\subset \mathcal{G}_{\min\{\d_{X_{x,n+1}^1},\d_{X_{x,n+1}^2}\}}$ and that $\lambda_n$ divides $\min\{\d_{X_{x,n+1}^1},\d_{X_{x,n+1}^2}\}$ implies that there exists a square $S_2\in\mathcal{A}$ such that $w\in \partial S_2$. By \eqref{eq:local distance of Gn,Gn+1} there exists ${s}\in \partial S_2\cap O_n$ that is a corner of $S_2$. Since $\g_{x,n}^i(q_{x,n}^i,w)\subset O_n$ is connected we can find an arc $\omega\subset \bigcup_{S\in\mathcal{A}}\partial S\cap O_n$ that connects $t$ to $s$. See Figure \ref{fig:polygonal arcs step 2}.

\begin{figure}[H]
\centering
\includegraphics[width=0.7\linewidth]{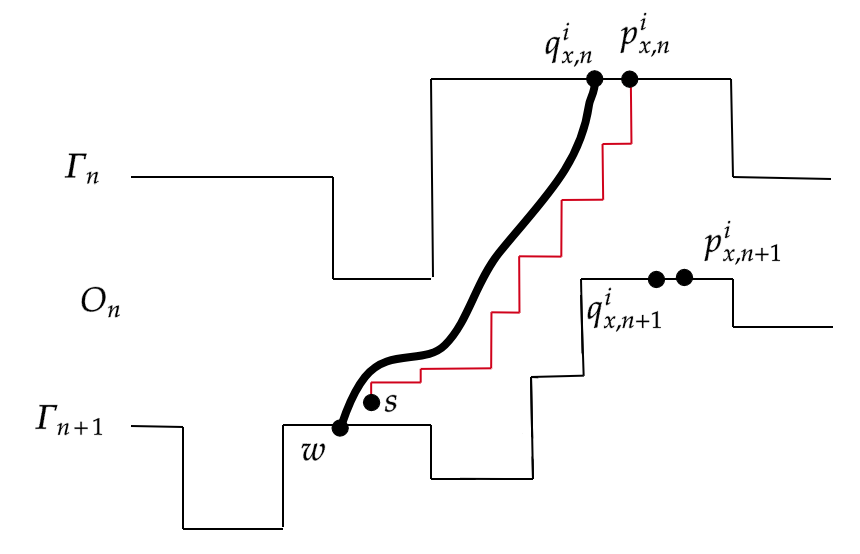}
\caption{}
\label{fig:polygonal arcs step 2}
\end{figure}

Set now 
$\omega_{x,n}^i=\ell_1\cup \omega$. Note that $\omega_{x,n}^i$ is an arc that connects $p_{x,n}^i$ to $s$. Since $\omega_{x,n}^i\subset \overline{O_n}$, it follows from Lemma \ref{lem:def of thickening of Tm,eta} that $\omega_{x,n}^i\cap T=\emptyset$. Note that if $z\in \g_{x,n}^i[q_{x,n}^i,w]$ and $z'\in \omega_{x,n}^i$, then $|z-z'|\leq 2\lambda_n$. Thus, by Lemma \ref{lem:arzela curves}(viii)  we have $\omega_{x,n}^i\cap \{\g_{y,m}^1,\g_{y,m}^2\}=\emptyset$ for $y\in{\bf V}^m\setminus\{x\}$ and that $\omega_{x,n}^i\cap\g_{x,n}^j=\emptyset$ for $j\in\{1,2\}\setminus\{i\}$. Last, Lemma \ref{lem:properties of thickening}(ii), Lemma \ref{lem:arzela curves}(iii), and \eqref{eq:def of delta X} give 
\begin{align*}
\diam \omega_{x,n}^i&\leq 4\lambda_n+ \diam \g_{x,n}^i[q_{x,n}^i,w]\\
&\leq 4\lambda_n+\ell(\g_{x,n}^i)\\ 
&\leq 4\lambda_n+64 C_ TC^3\min\{\d_{X_{x,n+1}^1},\d_{X_{x,n+1}^2}\}\\ 
&\lesssim_{C_T,C}\min\{\d_{X_{x,n+1}^1},\d_{X_{x,n+1}^2}\}.
\end{align*}

We replace now the arc $\g_{x,n}^i(w,q_{x,n+1}^i]$. Note that $q_{x,n+1}^i\in D_{n+1}$ and hence, denote by $z_{1},z_{2}$ its adjacent points in $\G_{n+1}$. By Lemma \ref{lem:prerequisite for polygonal appr}(i) there exists $m\in\{1,2\}$ such that $w\in\G_{n+1}[q_{x,n+1}^i,z_{m}]$. Denote by $\alpha$ the arc that lie in $\G_n[z_m,q_{x,n+1}^i]$ that connects $w$ to $q_{x,n+1}^i$. Consider the following set
$$\mathcal{A}'=\{S\in\Sigma_{\lambda_n}\cap \overline{O_n}:S\cap \alpha\neq\emptyset\}.$$
Note that by \eqref{eq:local distance of Gn,Gn+1}, $S_2\in\mathcal{A}'$. Moreover, by \eqref{eq:def of delta X} and Lemma \ref{lem:arzela curves}(iii),(iv) we have $\g_{x,n}^i\subset B(x,r_{n+1})$, $\G_{n+1}\cap B(x,r_{n+1})\subset \mathcal{G}_{\min\{\d_{X_{x,n+1}^1},\d_{X_{x,n+1}^2}\}}$, and that $\lambda_n$ divides $\min\{\d_{X_{x,n+1}^1},\d_{X_{x,n+1}^2}\}$. Hence, there exists a square $S_3\in\mathcal{A}'$ such that $q_{x,n+1}^i\in \partial S_3$. Let $\ell_2$ be a side of $\partial S_3$ that meets $\G_{n+1}$ at a unique point, which we call $p_{x,n+1}^i\in\G_{n+1}\cap \partial S_3$, and the other endpoint is in $O_n$ from \eqref{eq:local distance of Gn,Gn+1}, say $t'$, and $\ell_{2}\setminus\{p_{x,n+1}^i\}\subset O_n$. Since $\alpha$ is connected we can find an arc $\omega'\subset \bigcup_{S\in \mathcal{A}'}\partial S\cap O_n$ that connects $s$ to $t'$. See Figure \ref{fig:polygonal arcs step 3}.

\begin{figure}[H]
\centering
\includegraphics[width=0.7\linewidth]{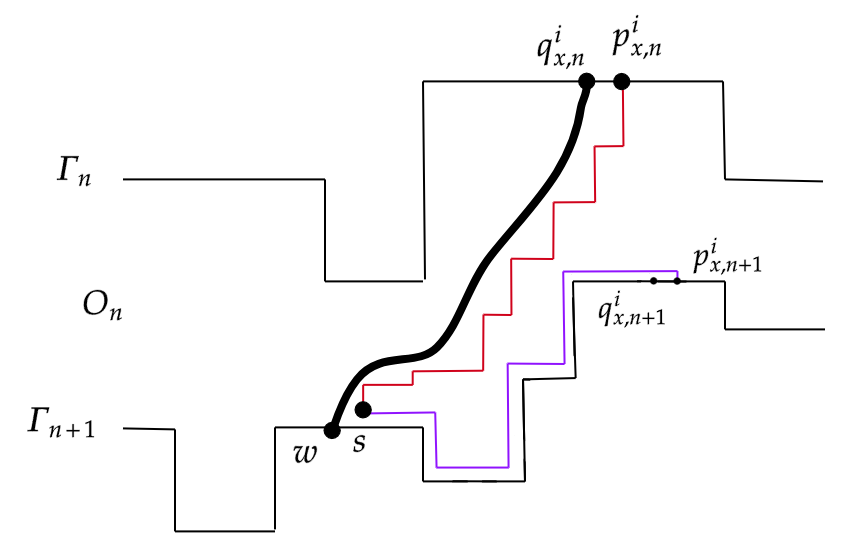}
\caption{}
\label{fig:polygonal arcs step 3}
\end{figure}

Set now $\omega_{x,n+1}^i=\omega'\cup\ell_{2}$. Note that $\omega_{x,n+1}^i$ is an arc that connects $s$ to $p_{x,n+1}^i$. Since $\omega_{x,n+1}\subset \overline{O_n}$, it follows now from Lemma \ref{lem:def of thickening of Tm,eta} that $\omega_{x,n+1}^i\cap  T=\emptyset$. Note also that if $z\in \alpha$ and $z'\in \omega_{x,n+1}^i$, then $|z-z'|\leq 2\lambda_n$. Thus, by Lemma \ref{lem:prerequisite for polygonal appr}(i) we have $\omega_{x,n+1}^i\cap \{\g_{y,m}^1,\g_{y,m}^2\}=\emptyset$ for $y\in{\bf V}^m\setminus\{x\}$ and that $\omega_{x,n+1}^i\cap \g_{x,n}^j=\emptyset$ for $j\in\{1,2\}\setminus\{i\}$. Last, Lemma \ref{lem:diameter of the sides of W.S}(iii) and \eqref{eq:def of delta X} give
\begin{align*}
\diam \omega_{x,n+1}^i&\leq 4\lambda_n+\diam \alpha\\
&\leq 4\lambda_n+\diam\G_{n+1}[q_{x,n+1},z_m]\\ 
&\lesssim_{C,C_T}\min\{\d_{X_{x,n+1}^1},\d_{X_{x,n+1}^2}\}.
\end{align*}

Let now $\zeta_n^i\subset\omega_{x,n}^i\cup\omega_{x,n+1}^i$ be an arc that connects $p_{x,n}^i$ to $p_{x,n+1}^i$. By construction, $\zeta_n^i\cap T=\emptyset$, $\zeta_n^i\subset\mathcal{G}_{\lambda_n}\cap \overline{\Omega_n}$, $\zeta_n^i\cap \G_n=\{p_{x,n}^i\}$, $\zeta_n^i\cap\G_{n+1}=\{p_{x,n+1}^i\}$, $\zeta_n^i\cap \{\g_{y,m}^j\}_{j\in\{1,2\}}=\emptyset$ for $y\in{\bf V}^m\setminus\{x\}$, $\zeta_{n}^i\cap\g_{x,n}^j=\emptyset$ for $j\in\{1,2\}\setminus\{i\}$, and $\diam \zeta_n^i\lesssim\min\{\d_{X_{x,n+1}^1},\d_{X_{x,n+1}^2}\}.$ For the lower bound, 
similar arguments with \eqref{eq:local distance of Gn,Gn+1} show that 
\[ |p_{x,n}^i-p_{x,n+1}^i|\gtrsim_{C_T,C}\min\{\d_{X_{x,n+1}^1},\d_{X_{x,n+1}^2}\}.\]

Suppose now that we have replaced the arc $\g_{x,k}^i[q_{x,k}^i,q_{x,k+1}^i]$, for $k\geq n$, with $\zeta_k^i[p_{x,k}^i,p_{x,k+1}^i]$. To replace the arc $\g_{x,k}^i[q_{x,k+1}^i,q_{x,k+2}^i]$ the process is similar with the basis step with the only difference that we replace the point $q_{x,k+1}^i$ with $p_{x,k+1}^i$.

For each $k\geq n$, set 
$$\zeta_{x,k}^i=\bigcup_{\ell=k}^{\infty}\zeta_{\ell}^i\cup\{x\}=\bigcup_{\ell=k}^{\infty}\zeta_{x,\ell}^i[p_{x,\ell}^i,p_{x,\ell+1}^i]\cup\{x\}.$$
Note that $\zeta_{x,k}^i$ is a closed arc since by Lemma \ref{lem:arzela curves}(iii)
$$\dist(x,\zeta_{x,k}^i)\leq |x-p_{x,k}^i|\leq |x-q_{x,k}^i|+|q_{x,k}^i-p_{x,k}^i|\lesssim_{C_T,C}\min\{\d_{X_{x,k+1}^1},\d_{X_{x,k+1}^2}\}$$ 
and thus for $k\to \infty$, $\min\{\d_{X_{x,k+1}^1},\d_{X_{x,k+1}^2}\}\to 0$ from Lemma \ref{lem:properties of BM}(viii) and \eqref{eq:def of delta X}.

Claims (i),(ii),(iv),(v) follow from the construction.

Claims (iii),(vii) follow from the construction and Lemma \ref{lem:arzela curves}(v),(vii),(viii).

We show claim (vi). The existence of $k$-tile $X$ follows similar arguments with Lemma \ref{lem:diameter of the sides of W.S}(i) where instead of Lemma \ref{lem:arzela curves}(v), we use (iii). The fact that $T(x,y)\subset X\setminus\partial X$ follows from Lemma \ref{lem:properties of BM}(iv). The fact that $\G_n[p_{x,n}^i,p_{y,n}^j]\subset \mathcal{G}_{\d_Z}$ where $Z\in{\bf X}^n$ such that $\d_Z=\min_{W\in{\bf X}^n,W\cap X\neq\emptyset}\d_W$ and the estimate of $\diam \G_n[p_{x,n}^i,p_{y,n}^j]$  follows from Lemma \ref{lem:arzela curves}(iv), Lemma \ref{lem:diameter of the sides of W.S}(iii), and the fact that for each $n\in\N$, $z\in {\bf V}^n$, $|q_{x,n}^i-p_{x,n}^i|\leq 2\lambda_n$.

Finally, claim (viii) follows Lemma \ref{lem:arzela curves}(vii), Lemma \ref{lem:prerequisite for polygonal appr}(i),(ii), and the fact that for each $n\in\N$, $z\in {\bf V}^n$, $|q_{x,n}^i-p_{x,n}^i|\leq 2\lambda_n$.
\end{proof}

We are finally ready to prove Proposition \ref{prop:polygonal curves}.

\begin{proof}[{Proof of Proposition \ref{prop:polygonal curves}}]
Fix a tree $T\subset \R^2$ that is $C_T$-bounded turning and has a $C$-tile decomposition $({\bf X}^n)_{n\in\N}$. Let $(k_X)_{X\in{\bf X}^n}$, $(\d_X)_{X\in{\bf X}^n}$ as in \eqref{eq:def of delta X} and let $({\bf V}^n)_{n\in\N_0}$ as in Definition \ref{def:tile decomposition}. Note that Definition \ref{def:tile decomposition}(iv) asserts that $k_X\simeq_{C}k_Y$ for any $X\in {\bf X}^n$ and $Y\in {\bf X}^{n+1}$ with $X\cap Y\neq\emptyset$. Let also $(\Omega_n)_{n\in\N}$ as in Lemma \ref{lem:def of thickening of Tm,eta} and set $\G_n=\partial \Omega_n$. Note that by Lemma \ref{lem:def of thickening of Tm,eta}, $\Omega_1\supset\Omega_2\supset\cdots$. Let $K\in \N$ be the constant given from Definition \ref{def:tile decomposition}(i) and let $N\in\N$ such that $N>4K$. Note that $N$ depends only on $C$. 

We define the set $\mathscr{W}$ in an inductive fashion. To that end, let 
$$K_n=\{p_{x,n}^i:x\in {\bf V}^n, i\in \{1,2\}\},$$
where $p_{x,n}^i$ is as in Lemma \ref{lem:polygonal curves}(i).
Fix a leaf $1$-tile $X$ in $T$ and let $\{v\}=\partial X$. Without loss of generality let $p_1,\dots,p_{m_1}$ be a counterclockwise enumeration of the points in $K_1$ along $\G_1$ with $p_1=p_{v,1}^1$ and $p_{2}=p_{v,1}^2$. By Definition \ref{def:tile decomposition}(i), Lemma \ref{lem:polygonal curves} and the definition of $N$ we have $m_1\in \{4,\dots, N\}$. The points $p_1,\dots,p_{m_1}$ induce a sequence on the double points of ${\bf V}^1$ and the arcs $(\zeta_{x,1}^i)_{x\in {\bf V}^1,i\in \{1,2\}}$, namely
\begin{align*}
x_1,x_2,\dots,x_{m_1}\qquad \text{and}\qquad \zeta_1,\zeta_2\dots,\zeta_{m_1}.
\end{align*}
where $x_1=x_2=v$, and $\zeta_1=\zeta_{v,1}^1$, $\zeta_2=\zeta_{v,1}^2$.

Set $\mathscr{W}_0=\{\e\}$ and $\mathscr{W}_1=\{1,\dots,m_1\}\subset \{1,\dots,N\}$. Suppose that we have defined $\mathscr{W}_n=\{w_1^n,\dots,w_{m_n}^n\}\subset \{1,\dots,N\}^{n}$ for some $n\in\N$. For simplicity, we write $w_i=w_i^n$. Let 
$$\mathscr{W}_{n+1}=\{w_11,\dots,w_1N_{w_1},w_21,\dots,w_{m_n}1,\dots, w_{m_n}N_{w_{m_n}}\}$$
where for each index $i\in \{1,\dots,m_n\}$ and each word $u\in \{1,\dots,N_{w_i}\}$, $p_{w_iu}$ is the unique endpoint of $\zeta_{w_i}|\overline{\Omega_{n+1}}$ on $\G_{n+1}$ by Lemma \ref{lem:polygonal curves}(i). Set $x_{w_i1}=x_{w_i}$ and $\zeta_{w_i1}=\zeta_{w_i}|\overline{\Omega_{n+1}}$. 
Finally, set $\mathscr{W}=\bigcup_{n\in\N_0}\mathscr{W}_n$ and note that for each $n\in\N$, $\card(\mathscr{W}_n)<\infty$.

Properties (ii.a) and (ii.c) follow from the design of $\mathscr{W}$. Property (ii.b) follows from Lemma \ref{lem:properties of BM}(ii),(iii). Property (ii.d) follows from Lemma \ref{lem:polygonal curves}(viii). Property (ii.e) follows the fact that for each $n\in\N$, $T=\bigcup_{X\in {\bf X}^n}X$. 

Property (iii) follows mainly from . Properties (iii.a) and (iii.b) follow from Lemma \ref{lem:polygonal curves}(i). Property (iii.c) follows from Lemma \ref{lem:polygonal curves}(ii),(iv),(v). Property (iii.e) follows from Lemma \ref{lem:polygonal curves}(viii). Last, property (iii.d) follows from the design of $\mathscr{W}$.

Property (v) follows from Lemma \ref{lem:polygonal curves}(iii),(vii)

It remains to show properties (i) and (iv). Property (i.b) follows from the design of $\mathscr{W}$. Property (iv.a) follows from  Lemma \ref{lem:properties of thickening}(ii) and similar arguments with Lemma \ref{lem:def of thickening of Tm,eta}(iii),(iv). For the rest of the proof we show properties (i.a) and (iv.b) simultaneously.

Fix $n\in\N$, $i\in \{1,\dots,m_n\}$ and $w_i\in\mathscr{W}_n$. We show that $N_{w_i}\in \{2,\dots,N\}$. Note that by design of $\mathscr{W}$, the points $p_{w_i},p_{w_{i+1}}$ are adjacent in $\G_n$. By Lemma \ref{lem:polygonal curves}(i),(vii),(viii) the set
$$\zeta_{w_i}\cup T[x_{w_i},x_{w_{i+1}}]\cup \zeta_{w_{i+1}}\cup \G_n[p_{w_{i+1}},p_{w_i}].$$
is a Jordan curve and let $O$ denote the Jordan domain enclosed by it. Using similar arguments as in the proof of Lemma \ref{lem:diameter of the sides of W.S}(i),(ii), it follows that there exists a unique $X\in{\bf X}^n$ such that $x_{w_i},x_{w_{i+1}}\in \partial X$, $O\subset \Omega_n$, and that $O\cap T\subset X\setminus\partial X$. Let $\g_1,\g_2$ be the two arcs in $\G_{n+1}$ that connect $p_{w_i1}$ to $p_{w_{i+1}1}$, where $\g_1$ is the arc that contains the points $p_{w_i1},\dots,p_{w_iN_{w_i}},p_{w_{i+1}1}$. 

We claim that $\g_1(p_{w_i1},p_{w_{i+1}1})\subset O$. Assume otherwise. By Lemma \ref{lem:polygonal curves}(i) either $\g_2(p_{w_i1},p_{w_{i+1}1})\subset O$ or $O\cap (\g_1\cup\g_2)=\emptyset$. If the latter is true, then $(\overline{O}\setminus\{p_{w_i1},p_{w_{i+1}1}\})\cap \G_{n+1}=\emptyset$ and since $\overline{O}\setminus\{p_{w_i1},p_{w_{i+1}1}\}$ is connected, and $x_{w_i},x_{w_{i+1}}\in\overline{O}$ we have that $\overline{O}\setminus\{p_{w_i1},p_{w_{i+1}1}\}\subset \Omega_{n+1}$ which is false since $p_{w_i},p_{w_{i+1}}\in\overline{O}\cap \G_n$ and, therefore, by Lemma \ref{lem:def of thickening of Tm,eta} they cannot be interior points of $\Omega_{n+1}$. Suppose now that $\g_2(p_{w_i1},p_{w_{i+1}1})\subset O$. Note that by construction of $\g_2$, there exists $p_{v1}\in \g_2$ where $v\in\mathscr{W}_n\setminus\{w_i,w_{i+1}\}$. By Lemma \ref{lem:properties of BM}(iv) and Lemma \ref{lem:polygonal curves}(i),(vii),(viii) it follows that $\zeta_{v1}\subset O$. In particular, since $O\cap T\subset X\setminus\partial X$, we have that $x_{v1}\in X$. But by (ii), $x_{v1}=x_v$ which implies that $x_v\in \partial X\cup Z$ for some $Z\in{\bf X}^n\setminus\{X\}$. We reach a contradiction. The proof of the claim is complete

Now by Lemma \ref{lem:properties of BM}(ii),(iv) and Lemma \ref{lem:polygonal curves}(i),(iii),(viii), $N_{w_i}\geq 2$. Since, 
\[ x_{w_i2},\dots,x_{w_iN_{w_i}}\in X\cap {\bf V}^{n+1},\] 
Definition \ref{def:tile decomposition}(i) gives that $N_{w_i}\leq N$ and property (i.b) follows. Last, property (iv.a) follows from the above work and Lemma \ref{lem:polygonal curves}(vi).
\end{proof}

\section{Proof of Theorem \ref{thm:main thm} and Theorem \ref{thm:Schoenflies-trees}}\label{sec: proof of main theorem}
In this section we prove our main results Theorem \ref{thm:main thm} and Theorem \ref{thm:Schoenflies-trees}.

\subsection{Compatible Whitney decompositions}\label{sec:compatibility of qc trees}
We start by establish a compatibility property of the Whitney decompositions constructed in Section \ref{sec: Decomposition of R^2 minus T}. More precisely, we show that a quasisymmetric or a bi-Lipschitz branch-orientation-preserving homeomorphism between two planar quasiconformal trees induces the same indexing set $\mathscr W$ for their respective Whitney decompositions. This allows us to match corresponding Whitney polygons in the proof of Theorem \ref{thm:Schoenflies-trees}.

\begin{prop}\label{prop:combinatorial connection T, T^W^T}
Let $T$ be a planar quasiconformal tree and let $f:T\to T'$ be a quasisymmetric branch-orientation-preserving homeomorphism. Let also, $({\bf X}^n)_{n\in\N}$ be a $C$-tile decomposition of $T$ and $\{\mathbb{X}^n\}_{n\in\N}$ be the $C'$-tile decomposition given from Proposition \ref{prop:tile-decomposition}. If $\mathscr{W}$, $\mathscr{W}'$ are the sets given from Proposition \ref{prop:polygonal curves} for $(T,({\textbf X}^n)_{n\in\N})$ and $(T',({\mathbb X}^n)_{n\in\N})$, respectively, then $\mathscr{W}=\mathscr{W}'$.
\end{prop}

\begin{proof}
Applying Proposition \ref{prop:polygonal curves}, to $(T,({\bf X}^n))$ and $(T',({\mathbb X}^n)_{n\in\N})$ there exist integers $N,N'\geq 1$ and sets  $\mathscr{W}\subset\{1,\dots,N\}^*$, $\mathscr{W}'\subset \{1,\dots,N'\}$ satisfying the properties of Proposition \ref{prop:polygonal curves}. Since $g$ is a homeomorphism it follows that $N=N'$ and that for each $n\in\N$, $\card(\mathscr{W}_n)=\card(\mathscr{W}_n')$. It remains to show that $\mathscr{W}=\mathscr{W}'$. To this end, we claim that if $n\in\N$, if $w\in\mathscr{W}_n'$, and if $y_w\in g({\bf V}^n)$, then $y_w=g(x_w)$, i.e., $w\in \mathscr{W}_n$. Assuming the claim, it follows that $\mathscr{W}'_n\subset\mathscr{W}_n$. Similar arguments applied to $g^{-1}$ show that $\mathscr{W}_n\subset \mathscr{W}_n'$. For simplicity, we show the claim for $n=1$ and identical arguments work for any $n\in\N$. 

Let $\Gamma_1$ be the Jordan curve associated with $T$ from \textsection\ref{sec:sqthick}, and let $(\zeta_i)_{i\in\mathscr{W}_1}$ denote the arcs with endpoints $x_i$ and $p_i\in\Gamma_1$ as in Proposition \ref{prop:polygonal curves}. As in the proof of Proposition \ref{prop:polygonal curves}, we may assume that $x_1=x_2\in X$, where $X\in{\bf X}^1$ is a fixed leaf tile of $T$, and that the points $\{p_i\}_{i\in\mathscr{W}_1}$ are enumerated in counterclockwise order along $\Gamma_1$. Let also $(g(x_{i_j}))_{j\in\mathscr{W}_1}\subset g(\mathbf{V}^n)$, $i_j\in\mathscr{W}_1$, denote the double points of $T'$, let $\G_1'$ be the Jordan curve associated with $T'$ from \textsection\ref{sec:sqthick}, and let $(\xi_j)_{j\in\mathscr{W}_1}$ denote the corresponding arcs with endpoints $g(x_{i_j})$ and $q_j\in \G_1'$ as in Proposition \ref{prop:polygonal curves}. Suppose also that the points $\{q_j\}_{j\in\mathscr{W}_1}$ are likewise enumerated in counterclockwise order along $\Gamma_1’$. After relabeling, we may assume that 
\[ g(x_{i_1})=g(x_1)=g(x_2)=g(x_{i_2}).\] 
Finally, after a cyclic permutation of the indexing, we may further assume that $i_3\neq 3$; in particular, $i_3\notin\{1,2,3\}$.

Denote by $\g[p_2,p_{i_3}]\subset \G_1$ the counterclockwise arc in $\G_1$ that connects $p_2$ to $p_{i_3}$. Consider the following sets
\begin{align*}
    \g[p_2,p_{i_3}]\cup \zeta_{i_3}\cup T[x_{i_3},x_2]\cup \zeta_2,\qquad \G_1'[q_2,q_{i_3}]\cup\xi_{i_3}\cup \mathcal{T}^{\W}[g(x_{i_3}),g(x_2)]\cup \xi_2.
\end{align*}
By Proposition \ref{prop:polygonal curves}(iv),(vii) both of these sets are Jordan curves. Denote by $O,O'$ the domains enclosed by them, respectively.

By Proposition \ref{prop:polygonal curves}(ii), $x_2\notin \{x_3,x_{i_3}\}$ and since $g$ is a homeomorphism $g(x_2)\notin\{g(x_3), g(x_{i_3})\}$. By construction of $\g$, $p_3\in \g(p_2,p_{i_3})$ and $p_1\notin \g$. By Proposition \ref{prop:polygonal curves}(i), $O\cap \G_1=\emptyset$ and $O$ is connected with $\partial O\cap T\neq\emptyset$. Thus, by Proposition \ref{prop:polygonal curves}(iii), $O\subset \Omega_1$ and similarly, $O\subset \Omega_1'$. Finally, by Proposition \ref{prop:polygonal curves}(iv),(vii) and the fact that $O\subset \Omega_1$, we have that $\zeta_3\setminus\{x_3,p_3\}\subset O$ and that $\xi_3\setminus\{g(x_3)\}\subset \R^2\setminus\overline{O'}$. Hence, either $x_3\in T(x_2,x_{i_3})$ or $x_3\in O$.

Suppose first, that $x_3\in T(x_2,x_{i_3})$. Since $g$ is homeomorphism we have $g(x_3)\in T'(g(x_2),g(x_{i_3}))$. However, since $q_2,q_{i_3}$ are adjacent points in $\G_1'$, Proposition \ref{prop:polygonal curves}(v) asserts the existence of $Y\in g({\bf X}^1)$ such that $g(x_2),g(x_{i_3})\in \partial Y$ and it asserts that $T'(g(x_2),g(x_{i_3}))\subset Y\setminus\partial Y$ which is a contradiction. 

Suppose now that $x_3\in O$. Traveling along $T[x_2,x_3]$ from $x_3$ to $x_2$, let $b\in T[x_2,x_3]$ be the first time that we meet $T[x_2,x_3]$. We claim that $b\in T(x_2,x_{i_3})$. Suppose otherwise that $b\in\{x_2,x_{i_3}\}$. Since $p_2,p_3$ are adjacent, by Proposition \ref{prop:polygonal curves}(v) there exists $Z\in{\bf X}^1$ such that $x_3,x_2\in \partial Z$ and $T(x_2,x_3)\subset Z\setminus\partial Z$. Thus, $b\neq x_{i_3}$. If $b=x_2$, then we claim that $\zeta_1\setminus\{x_2,p_1\}\subset O$. Assuming the claim, it follows from Proposition \ref{prop:polygonal curves}(iv),(vii) that $p_1\in \g(p_2,p_{i_3})$ which is contradiction. To show the claim, suppose otherwise. By Proposition \ref{prop:polygonal curves}(iv),(vii) we may assume that $\zeta_1\setminus\{x_2\}\subset \R^2\setminus\overline{O}$. Consider the following Jordan curve
$$\zeta_{i_3}\cup T[x_{i_3},x_2]\cup T[x_2,x_3]\cup \zeta_3\cup (\g[p_2,p_{i_3}]\setminus\G_1[p_2,p_3))$$
and denote by $V$ the domain enclosed. Note that $V\subset O$. Let $B=B(x_2,r)$ where $r=\frac12\min_{W\in{\bf X}^1}\d_W$. It is easy to verify that $x_3,x_{i_3},p_2,p_3\notin B$. Let $t_1\in \partial B\cap T(x_2,x_{i_3})$ and $t_2\in \partial B\cap T(x_2,x_{3})$ such that $x_2\in T(t_1,t_2)\subset B$. Denote by $W_1,W_2$ the connected components of $B\setminus T(t_1,t_2)$. By Lemma \ref{lem:Jordan domain prerequisite} there exists an arc $\alpha$ with one endpoint $t_1'\in T(x_2,t_1)$ and the other $t_2'\in T(x_2,t_2)$ such that $\alpha\setminus\{t_1',t_2'\}\subset V\cap W_i\subset O\cap W_i$ for some $i\in \{1,2\}$. But then, by Proposition \ref{prop:polygonal curves}(vii) we have that $(\alpha\setminus\{t_1',t_2'\})\cap (\zeta_1\cup \zeta_2)\neq\emptyset$ which is contradiction since $\zeta_2\subset \partial O$ and $\zeta_1\setminus\{x_2\}\subset \R^2\setminus\overline{O}$. See also Figure \ref{fig:combinatorial-intersection 1}. This proves the claim, and we get that $b\in T(x_2,x_3)$.

\begin{figure}
\centering
\includegraphics[width=0.6\linewidth]{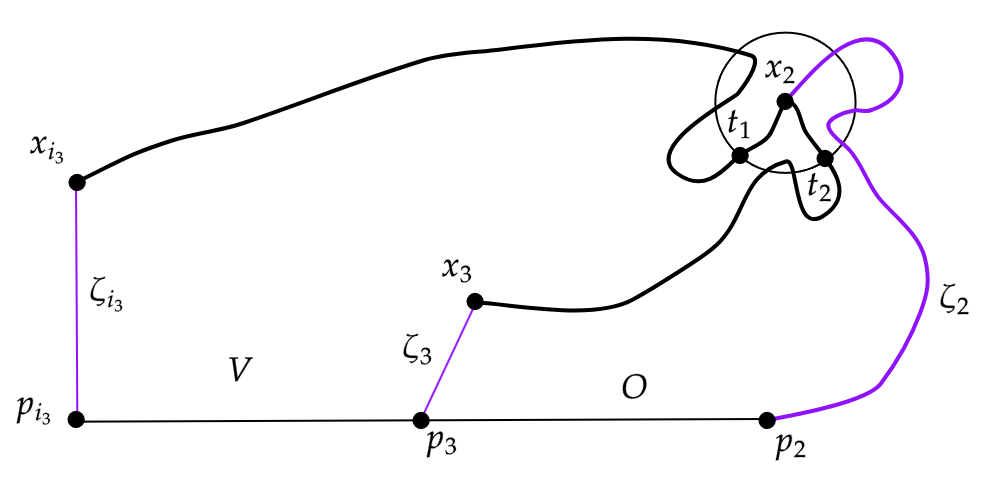}
\caption{Here $\zeta_1,\zeta_2$ are in the same connected component of $B(x_2,r)\setminus T[t_1,t_2]$.}\label{fig:combinatorial-intersection 1}
\end{figure}

By Lemma \ref{lem: metric trees properties of BM}(ii) it follows that $b$ is a branch point and, therefore, $g(b)\in T'(g(x_{i_3}),g(x_2))$ is also
a branch point. By Lemma \ref{lem:Jordan domain prerequisite}, there exist arcs
$\sigma_1,\sigma_2,\omega_1,\omega_2$ such that
\begin{enumerate}
\item $\sigma_1$ and $\sigma_2$ have common endpoints $s_1\in T(b,x_2)$ and $s_2\in T(b,x_{i_3})$, whereas $\omega_1$ and $\omega_2$ have common endpoints 
\[ s_1'\in T'(g(b),g(x_2)) \quad\text{and}\quad s_2'\in T'(g(b),g(x_{i_3})),\]
\item the interiors (with the induced topology) of $\sigma_1$ and $\omega_1$ 
\[ \sigma_1\setminus\{s_1,s_2\}\subset O, \qquad \omega_1\setminus\{s_1',s_2'\}\subset O'.\]
\item the interiors (with the induced topology) of $\sigma_2$ and $\omega_2$
\[ \sigma_2\setminus\{s_1,s_2\}\subset\mathbb{R}^2\setminus\overline{O}, \qquad \omega_2\setminus\{s_1',s_2'\}\subset\mathbb{R}^2\setminus\overline{O'}.\]
\end{enumerate}

Denote by $D$, resp.\ $D'$, the domain enclosed by $\sigma_1\cup\sigma_2$, resp.\ $\omega_1\cup\omega_2$. Let $B_1,\ldots,B_k$ be the branches of $T$ at $b$. Then $g(B_1),\ldots,g(B_k)$ are the branches of $T'$ at $g(b)$. For each $i\in\{1,\ldots,k\}$, we have $\partial D\cap B_i\neq\emptyset$ and $\partial D'\cap g(B_i)\neq\emptyset$. Moreover, by Lemma \ref{lem: maximal arcs}(i), there exist $y_i\in\partial D\cap B_i$ and $z_i\in\partial D'\cap g(B_i)$ such that $T[b,y_i)\subset D$ and $T'[g(b),z_i)\subset D'$. Relabeling the branches, we may assume that $x_2\in B_1$, $x_{i_3}\in B_2$, and $x_3\in B_3$. Without loss of generality, we may choose $s_1=y_1$, $s_2=y_2$, $s_1'=z_1$, and $s_2'=z_2$. By Proposition \ref{prop:polygonal curves}(iii), the fact that $T$ is a tree, and that $x_3\in B_3$, it follows that $B_3\subset O$. Hence, since $y_3\in B_3$, we have that $y_3\in O$. Similarly, since $g(x_3)\notin\overline{O'}$, we have $z_3\notin\overline{O'}$.

Consider now the Jordan curves. 
\begin{align*}
\alpha_1 &=T[p_2,p_{i_3}]\cup\zeta_{i_3}\cup T[x_{i_3},y_2] \cup\sigma_1\cup T[y_1,x_2]\cup\zeta_2,\\
\alpha_2 &=T'[q_2,q_{i_3}]\cup\xi_{i_3}\cup T'[g(x_{i_3}),z_2] \cup\omega_1\cup T'[z_1,g(x_2)]\cup\xi_2.
\end{align*}
Denote by $U,U'$ the domains enclosed by $\alpha_1,\alpha_2$, respectively. Since $y_3\in O\cap\partial U$ and $\sigma_1\setminus\{y_1,y_2\}\subset O\cap\partial U$, the counterclockwise order of the points $\{y_i\}_{i=1}^{\val(b)}$ along $\partial D$, starting from $y_2$, is
\[ y_2,\dots, y_3,\dots y_1,\dots .\]
On the other hand, since $z_3\notin\omega_1$ and $\omega_1\setminus\{z_1,z_2\}\subset O'\cap\partial U'$, their corresponding points occur counterclockwise along $\partial D'$, starting from $z_2$, in the order
\[z_2,\dots, z_1,\dots, z_3,\dots\]
This contradicts the assumption that $g$ is branch-orientation-preserving.
\end{proof}

\begin{figure}[H]
    \centering
    \includegraphics[width=\linewidth]{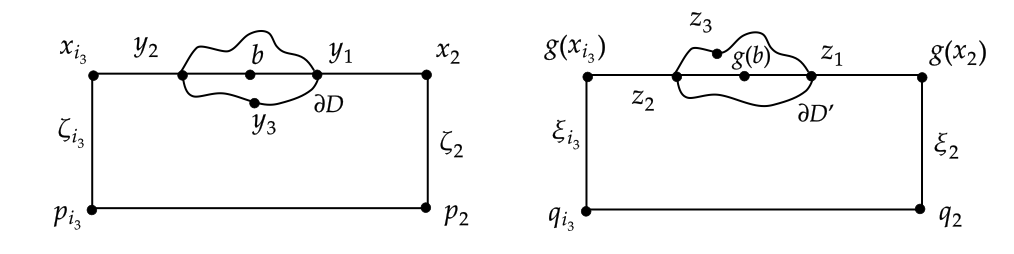}
    \caption{}
    \label{fig:Combinatorial intersection 2}
\end{figure}

\subsection{Proof of main results}
We first show the following quantitative version of Theorem \ref{thm:Schoenflies-trees}.

\begin{thm}\label{thm: intro 2 quant.}
Let $T\subset\mathbb R^2$ be planar quasiconformal tree and let $g:T\to \R^2$ be a branch-orientation-preserving embedding.
\begin{enumerate}[(i)]
\item If $g$ is $\eta$-quasisymmetric, then $g$ extends to a $K$-quasiconformal homeomorphism $ F:\mathbb R^2\to\mathbb R^2$, with $K$ depending only on the bounded turning constant of $T$ and on $\eta$.
\item If $g$ is $L$-bi-Lipschitz, then $g$ extends to an $L'$-bi-Lipschitz homeomorphism $F:\mathbb R^2\to\mathbb R^2,$ with $L'$ depending only on the bounded turning constant of $T$ and on $L$.
\end{enumerate}
Conversely, if $F:\R^2\to\R^2$ is an orientation preserving homeomorphism and if $T\subset \R^2$ is a tree, then $F|T$ is branch-orientation-preserving.
\end{thm}

\begin{proof}
The converse part follows directly by definition. 

Let $T,$ and $g$ be as in (i) and let $T'=g(T)$. By Proposition \ref{prop:tile-decomposition} there exist a $C_1$-tile decomposition $({\bf X}^n)_{n\in\N}$ of $T$ and a $C_2$-tile decomposition $({\mathbb X}^n)_{n\in\N}$ of $T'$ where $C_1$ depends on the bounded turning constant of $T$ and $C_2$ depends on the same constant and the distortion $\eta$. Moreover, by Propositions \ref{prop:polygonal curves}, \ref{prop:combinatorial connection T, T^W^T} both $T$ and $T'$ induce the same set $\mathscr{W}\subset \{1,\dots,N\}^*$ where $N\in\N$ depends only on $C_1$. For each $n\in\N$ denote by $\G_n$, resp. $\G_n'$, the associated Jordan curve of $T$, resp. of $T'$, given from Proposition \ref{prop:polygonal curves}(iv). Last, by Proposition \ref{prop:polygonal curves}(iii),(iv), let $Z_w\in\mathbf{X}^n$ and $l_w$ be associated with the tree $T$, and $W_w\in{\mathbb X}^n$ and $l_w'$ be associated with the tree $T'$, respectively.

Fix now $n\in\N$ and let $\mathscr{W}_n=\{w_1^n,\dots,w_{m_n}^n\}$. For simplicity, we write $w_i^n = w_i$. By Proposition \ref{prop:polygonal curves}(iii),(iv),(v), there exist Jordan domains $D_{w_i}$ and $\D_{w_i}$ such that 
\begin{align*}
\partial D_{w_i}&= e_{w_i}^U\cup e_{w_i}^L \cup \bigcup_{j=1}^{N_{w_i}} e_{w_i}^{B_j} \cup e_{w_i}^R, \qquad \partial \D_{w_i} = f_{w_i}^U\cup f_{w_i}^L\cup \bigcup_{j=1}^{N_{w_i}} f_{w_i}^{B_j} \cup f_{w_i}^R.
\end{align*}
so that
\begin{enumerate}
\item the upper edges satisfy 
\begin{align*}
e_{w_i}^U=\G_n[p_{w_i},p_{w_{i+1}}] \subset\mathcal G_{\d_{Z_{w_i}}} \quad&\text{and}\quad \diam e_{w_i}^U\simeq_{C_1} \d_{Z_{w_i}}\\
f_{w_i}^U=\G_n'[q_{w_i},q_{w_{i+1}}] \subset G_{\d_{W_{w_i}}} \quad&\text{and}\quad \diam f_{w_i}^U\simeq_{C_2} \d_{W_{w_i}},
\end{align*}
\item the left edges satisfy
\begin{align*}
e_{w_i}^L=\zeta_{w_{i+1}}\setminus \zeta_{w_{i+1}1}(p_{w_{i+1}1},x_{w_{i+1}1}] \subset \mathcal{G}_{l_{w_{i+1}}} \quad&\text{and}\quad \diam e_{w_i}^L\simeq_{C_1} l_{w_{i+1}}\\
f_{w_i}^L=\xi_{w_{i+1}}\setminus\xi_{w_{i+1}1}(q_{w_{i+1}1},g(x_{w_{i+1}1})] \subset \mathcal{G}_{l_{w_{i+1}}'} \quad&\text{and}\quad \diam f_{w_i}^L\simeq_{C_2} l_{w_{i+1}}',
\end{align*}
\item
the bottom edges satisfy
\begin{align*}
e_{w_i}^{B_j} = \G_{n+1}[p_{w_ij},p_{w_i(j+1)}] \subset \mathcal{G}_{\d_{Z_{w_ij}}} \quad&\text{and}\quad \diam e_{w_i}^{B_j}\simeq_{C_1}\d_{\d_{Z_{w_ij}}}\\
f_{w_i}^{B_j} = \G_{n+1}'[q_{w_ij},q_{w_i(j+1)}] \subset\mathcal G_{\d_{W_{w_ij}}} \quad&\text{and}\quad \diam f_{w_i}^{B_j}\simeq_{C_2} \d_{W_{w_ij}},
\end{align*}
with $\lambda_{w_i}^{B_j}=\d_{Z_{w_ij}}$ and $\ell_{w_i}^{B_j}=\d_{W_{w_ij}}$
\item
the right edges satisfy
\begin{align*}
e_{w_i}^R = \zeta_{w_i} \setminus \zeta_{w_i1}(p_{w_i},x_{w_i1}] \subset \mathcal{G}_{l_{w_i}} \quad&\text{and}\quad \diam e_{w_i}^R\simeq_{C_1}l_{w_i}\\
f_{w_i}^R= \xi_{w_i} \setminus \xi_{w_i1}(q_{w_i},g(x_{w_i1})] \subset\mathcal G_{l_{w_i}'}\quad&\text{and}\quad \diam f_{w_i}^R\simeq_{C_2}l_{w_i}'.
\end{align*}
\end{enumerate} 

\begin{figure}[ht]
\centering
\includegraphics[width=0.6\linewidth]{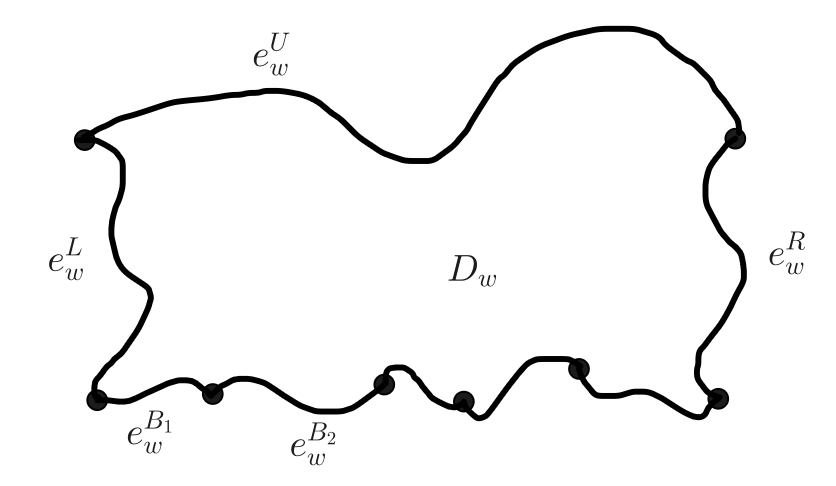}
\caption{A Whitney tile $D_w$.}
\label{fig:whitneytile}
\end{figure}

Recalling that $N_{w_i}\leq N$ and that
\begin{align*}
\d_{Z_{w_i}} \simeq_{C_1} l_{w_{i+1}} \simeq_{C_1} \d_{Z_{w_ij}} \simeq_{C_1} l_{w_i}\\
\d_{W_{w_i}} \simeq_{C_2} l_{w_{i+1}}' \simeq_{C_2} \d_{W_{w_ij}} \simeq_{C_2} l_{w_i}'
\end{align*}
it follows that there exists $K\in\N$ depending on $C_1,C_2$ such that all edges above consist of at most $K$ line segments.

Consider now, for each $A\in\{U,L,R,B_1,\ldots,B_{N_w}\}$, the normalized arclength parametrizations $G_w^A:e_w^A\to[0,1]$ and $H_w^A:f_w^A\to[0,1]$, with corresponding endpoints mapped to $0$ and $1$, and set $h_w^A=(H_w^A)^{-1}\circ G_w^A.$ By the finite-model construction, the source and target edges, after normalization by similarities, belong to finite families of nondegenerate polygonal arcs determined by $C_1$ and $C_2$. In particular, their lengths are comparable to their diameters. Together with the diameter estimates above, this gives
\[\frac{\ell(f_w^A)}{\ell(e_w^A)}\simeq_{C_1,C_2}\frac{\diam f_w^A}{\diam e_w^A}\simeq_{C_1,C_2}\frac{\delta_{W_w}}{\delta_{Z_w}}.\]
Since $h_w^A$ is defined by normalized arclength, each $h_w^A$ is an $(L,\delta_{W_w}/\delta_{Z_w})$-quasi-similarity,
where $L$ depends only on $C_1,C_2$.

Moreover, the edges of $\partial D_w$ lie on dyadic grids with mutually comparable mesh sizes and have uniformly controlled diameters. Hence, after normalization by similarities, the boundaries $\partial D_w$ belong to a finite family of nondegenerate Jordan polygons. The same holds for $\partial\Delta_w$. Consequently, these boundaries are uniformly quasiconvex, and, respectively, $(M,\delta_{Z_w})$- and $(M,\delta_{W_w})$-quasi-similar to $\mathbb{S}^1$,
where $M$ depends only on $C_1,C_2$.

For each $w\in\mathcal W$, define $\alpha_w:\partial D_w\to\partial\Delta_w$ by $\alpha_w|_{e_w^A}=h_w^A$. Since both boundaries are uniformly quasiconvex and the corresponding arclength ratios are comparable to $\delta_{W_w}/\delta_{Z_w}$, the map $\alpha_w$ is an $(L,\delta_{W_w}/\delta_{Z_w})$-quasi-similarity, where $L$ depends only on $C_1,C_2$. Similarly, define $h':\Gamma_1\to\Gamma'_1$ by $h'|_{e_v^U}=h_v^U$ for each $v\in\mathcal W_1$. Since the scales of all first-level tiles are mutually comparable and $\card(\mathcal W_1)\leq N$, the same finite-model argument shows that $\Gamma_1,\Gamma'_1$ are uniformly quasiconvex and that $h'$ is an $(L,\delta_{W_v}/\delta_{Z_v})$-quasi similarity for any fixed $v\in\mathcal W_1$. 

Applying Lemma \ref{lem:Tukia quasi-similarity}, for each $w\in\mathcal W$ we obtain an extension $\alpha'_w:\overline{D_w}\to\overline{\Delta_w}$ which is an $(L',\delta_{W_w}/\delta_{Z_w})$-quasi-similarity.
Similarly, we obtain an extension $H':\mathbb R^2\setminus\Omega_1
\to\mathbb R^2\setminus\Omega'_1$ which is an $(M',\delta_{W_v}/\delta_{Z_v})$-quasi-similarity for any fixed $v\in\mathcal W_1$, where $L',M'$ depend only on $C_1,C_2$.

Let now $F:\R^2\to\R^2$ with $F|_{\R^2\setminus\Omega_1}=H'$, $F|_{\overline{D_w}}=\alpha'_{w}$ for each $w\in \mathscr{W}$, and $F|_{T}=g$. We claim that $F|_{\R^2\setminus T}$ is $M$-quasiconformal, where $M\geq 1$ depends only on $C_1,C_2$. By construction, $F|_{\R^2\setminus T}: \R^2\setminus T \to \R^2\setminus T'$ is a homeomorphism. Let $x\in\R^2\setminus T$. If $x$ belongs to the interior of a Whitney domain $D_w$, then, for all sufficiently small $r>0$, we have $B(x,r)\subset D_w$. Since $F|_{D_w}=\alpha'_w$ is an $(s_w,L')$-quasi-similarity for some $s_w>0$, it follows directly from the definition of quasiconformal distortion $H_F$ that $ H_F(x)\leq (L')^2$. The same argument applies if $x\in \R^2\setminus\overline{\Omega}_1$, with $M'$ in place of $L'$. Suppose now that $x$ belongs to the boundary of a Whitney domain. By construction, there are at most 4 such domains meeting at $x$, and their scales are mutually comparable. The corresponding scaling factors of the quasi-similarities $F|_{D_w}$ are therefore mutually comparable, with constants depending only on $C_1,C_2$. Applying the quasi-similarity inequalities on each of the domains containing $x$ gives $H_F(x)\leq M$, where $M$ depends only on $C_1,C_2$. Consequently, $F|_{\R^2\setminus T}$ is $M$-quasiconformal.

It remains to show that $F$ is a homeomorphism. Assuming the claim, it follows from Lemma \ref{lem: Nakki vaisalla} and Lemma \ref{lem:qc removable}, that $F$ is a $\Lambda$-quasiconformal map with $F(T)=g(T)=T'$, where $\Lambda$ depends only on $C_1,C_2$.

Suppose that $x_n\to x\in T$. Passing to a subsequence we may assume that there exists an increasing sequence $j_n$ such that $x_n\in \overline{D_{w_{j_n}}}$ for some $w_{j_n}\in \mathscr{W}_{j_n}$. Let $z_n=p_{w_{j_n}}$ and let $v_n=x_{w_{j_n}}$. By Proposition \ref{prop:polygonal curves}, $\diam \overline{D_{w_{j_n}}},\diam \overline{\Delta_{w_{j_n}}}\to 0$. By Proposition \ref{prop:polygonal curves},
$$|v_n-x|\leq |v_n-z_n|+|z_n-x_n|+|x_n-x|\lesssim \d_{X_{j_n}}+\diam \overline{D_{w_{j_n}}}+|x_n-x|\to 0,$$
where $X_{j_n}\in\mathbf{X}^{j_n}$ contain $v_n$ on their boundary. Hence, $v_n\to x$. By continuity of $g$ we have that $g(v_n)\to g(x)$. Note also that $F(x_n)\in\overline{\D_{w_{j_n}}}$, $F(z_n)=q_{w_{j_n}}$, and $F(v_n)=g(x_{w_{j_n}})$. Therefore, by Proposition \ref{prop:polygonal curves},
\begin{align*}
|F(x_n)-F(x)| &\leq |F(x_n)-F(z_n)|+|F(z_n)-F(v_n)|+|F(v_n)-F(x)| \\ & \lesssim \diam \overline{\Delta_{w_{j_n}}}+\d_{Y_{j_n}}+|g(v_n)-g(x)|\to 0,
\end{align*}
where $Y_{j_n}\in{\mathbb X}^{j_n}$ contain $g(v_n)$ on their boundary. It follows that $F(x_n)\to F(x)$ which means that $F$ is continuous. Identical arguments show that $F^{-1}$ is also continuous. Thus, $F$ is a homeomoprhism.

For (ii), suppose that $g:T\to T'$ is $L$-bi-Lipschitz. Since the corresponding tiles have $L$-comparable diameters, it follows from the definition of the scales in \eqref{eq:def of delta X} that $\delta_{W_w}\simeq_{C_1,C_2,L}\delta_{Z_w}$ for every $w\in\mathcal W$. The same holds for the scaling factor of the exterior extension $H'$. Recall that the boundary maps $\alpha_w$ and $h'$ are uniformly quasi-similarities and that their source boundaries are uniformly bi-Lipschitz equivalent to $\mathbb{S}^1$, after normalization. Hence, by Lemma \ref{lem:Tukia quasi-similarity}, their extensions $\alpha'_w$ and $H'$ can be chosen uniformly bi-Lipschitz, with constants depending only on $C_1,C_2,L$. Thus, there exists $L'\geq1$, depending only on the bounded turning constants of $T,T'$ and on $L$, such that
\begin{equation}\label{eq:bilip-pieces}
(L')^{-1}\ell(\gamma) \leq \ell(F\circ\gamma)\leq L'\ell(\gamma)
\end{equation}
whenever $\gamma$ is a rectifiable arc contained either in $T$ or in one of the Whitney domains.

We claim that $F$ is $L'$-BLD. Let $\gamma:[a,b]\to\mathbb R^2$ be a rectifiable path, which we may assume to be parametrized by arclength. Since $T$ and the boundaries of the Whitney domains form a closed set, the portions of $\gamma$ contained in the interiors of the Whitney polygons form a countable collection of pairwise disjoint open subarcs. Applying \eqref{eq:bilip-pieces} to these subarcs and summing gives the required length estimate on the portion of $\gamma$ outside $T$. On the remaining portion, the same estimate follows from $F|_T=g$, which is $L$-bi-Lipschitz. By the additivity of length, $\ell(F\circ\gamma)\leq L'\ell(\gamma).$ The same argument applies to $F^{-1}$, since its restrictions to $T'$ and to the corresponding Whitney polygons are uniformly bi-Lipschitz. Hence,
\[ \ell(\gamma) =\ell(F^{-1}\circ(F\circ\gamma)) \leq L'\ell(F\circ\gamma),\]
and therefore
\[ (L')^{-1}\ell(\gamma) \leq\ell(F\circ\gamma) \leq L'\ell(\gamma).\]
Thus $F$ is an $L'$-BLD homeomorphism of $\mathbb R^2$. Hence, $F$ is bi-Lipschitz, with constant depending only on the bounded turning constant of $T$ and on $L$.
\end{proof}

We finish the section by showing a quantitative version of Theorem \ref{thm:main thm}. 

\begin{thm}\label{thm:quantitative version}
Given $C\geq 1$, $\epsilon>0$, and $m\in\{2,3,\dots\}$, there exists $c=c(m)\geq 1$ and $K=K(C,m,\e)\geq 1$ with the following property. 
\begin{enumerate}
\item If $T$ is a planar tree with $\val(T)\leq m$, then there exists a homeomorphism $F:\R^2\to\R^2$ such that $F(T)$ is a semi-geodesic tree with constant $c$ and with $\dim_H{F(T)}<1+\epsilon$.
\item If $T\subset \R^2$ is a $C$-bounded turning tree, with $C$-uniform branch separation, and with $\val(T)\leq m$, then there exists a $K$-quasiconformal map $F:\R^2\to\R^2$ such that $F(T)$ is a semi-geodesic tree with constant $c$ and with $\dim_H{F(T)}<1+\epsilon$.
\end{enumerate}
\end{thm}

\begin{proof}[Proof of Theorem \ref{thm:quantitative version}]
Fix constants $C$, $\epsilon$, and $m$ as in the statement. 

Suppose first that $T\subset \R^2$ is a tree. By Proposition \ref{prop: enumeration of T=enumeration of Tmeta}(ii) there exists an $m$-disctionary $\W$ and a branch-orientation-preserving homeomorphism $g:T\to \mathcal{T}^{\W}$ such that $\dim_H\mathcal{T}^{\W}<1+\epsilon$ and $\mathcal{T}^\W$ is semi-geodesic with constant $c=c(m)$. By \cite[Theorem 2]{ADMaclane}, the map $f$ extends to a homeomorphism $F:\R^2 \to \R^2$.

Suppose now that $T$ is $C$-bounded turning and has $C$-uniform branch separation. By Proposition \ref{prop: enumeration of T=enumeration of Tmeta}(iii) there exists an $m$-disctionary $\W$ and a $\theta$-quasisymmetric branch-orientation-preserving map $g:T\to \mathcal{T}^{\W}$ such that $\dim_H\mathcal{T}^{\W}<1+\epsilon$, $\mathcal{T}^\W$ is semi-geodesic with constant $c=c(m)$, and $\theta$ depends only on $C$ ,$\epsilon$, $m$. Theorem \ref{thm: intro 2 quant.} asserts now the existence of a $K$-quasiconformal map $F:\R^2\to \R^2$ such that $F(T)=g(T)=\mathcal{T}^{\W}$ and $K$ depends only on $m$, $C$, and $\epsilon$. 
\end{proof}

\section{A quasiconformal uniformization of the Vicsek tree}\label{sec:Vicsekunif}
Using complex notation, define 5 similarities on $\mathbb{C}$:
\[ \phi_1(z)=\frac13(z-2-2i), \qquad \phi_2(z)=\frac13(z-2+2i), \qquad \phi_3(z)=\frac13(z+2+2i), \]
\[\phi_4(z)=\frac13(z+2-2i), \qquad \phi_5(z)=\frac13 z.\]
The Vicsek tree $\V$ is the unique nonempty compact subset of
$\mathbb C$ satisfying
\[\V=\bigcup_{i=1}^5 \phi_i(\V).\]

\begin{figure}[h]
    \centering
    \includegraphics[width=0.9\textwidth]{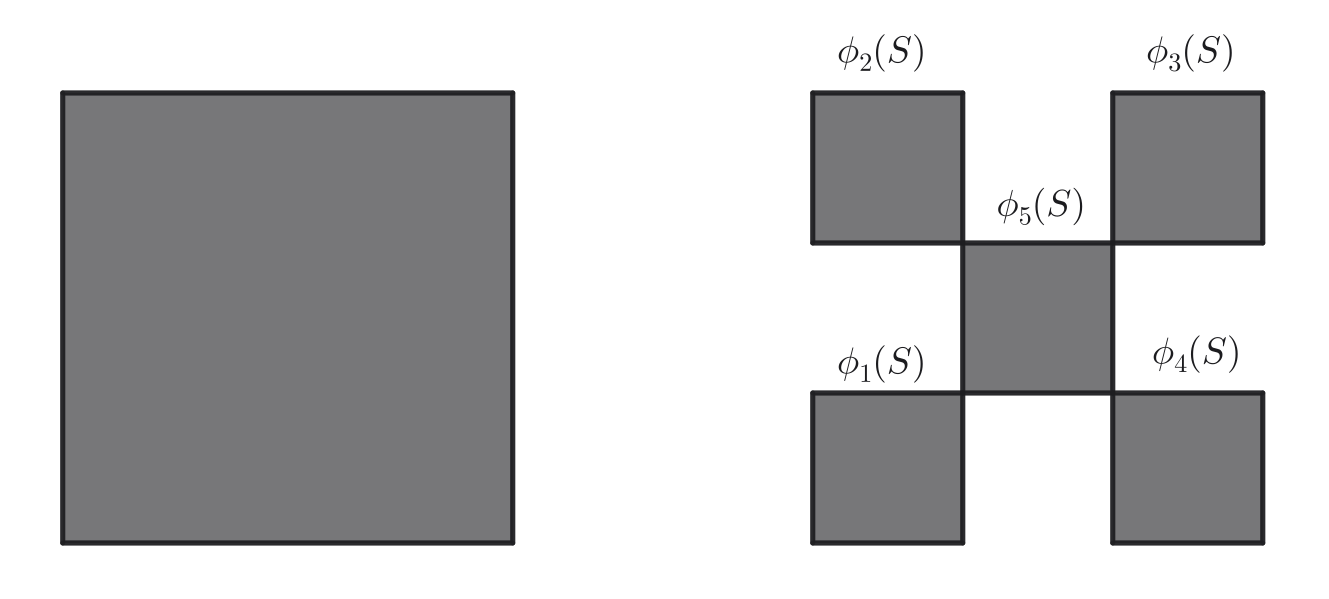}
    \caption{The square $S=[-1,1]^2$ and its images under $\phi_1,\dots,\phi_5$.}
    \label{fig:vicsek-model-tiles1}
\end{figure}

Given a word $w=i_1\cdots i_k\in A_5^*$, we write
\[\phi_w=\phi_{i_1}\circ\cdots\circ\phi_{i_k}\quad\text{and}\quad\V_w=\phi_w(\V).\]

\begin{lem}\label{lem:semigeodesic Vicsek}
The Vicsek tree $\V$ is semi-geodesic.
\end{lem}

\begin{proof}
It is known that the geodesic metric on $\V$ is bi-Lipschitz equivalent to the restriction of the Euclidean metric to $\V$; see \cite[Section 2.2]{BaudoinChen1}. Hence $\V$ is quasiconvex. It remains to verify the second condition. Let
\[L=[-1-i,1+i]\cup[-1+i,1-i]\]
be the union of the two diagonals of $\V$, and for $n\geq 0$ set
\[\Gamma_n:=\bigcup_{|w|=n}\phi_w(L).\]
Then each $\Gamma_n$ is a finite tree consisting of line segments, $\Gamma_n\subset \V$, and $\bigcup_n\Gamma_n$ is dense in $\mathbb V$.

Let $p\in \V\setminus\mathcal L(\V)$. Choose points $x_1,x_2$ in two distinct components of $\V\setminus\{p\}$ and in $\bigcup_n\Gamma_n$. For sufficiently large $n$ we have $x_1,x_2\in\Gamma_n$. By uniqueness of arcs, $ \V[x_1,x_2]=\Gamma_n[x_1,x_2],$ so $\V[x_1,x_2]$ is a polygonal arc containing $p$. If $p$ is not a corner of this arc, then a sufficiently short subarc containing $p$ is a line segment. If $p$ is a corner, then $p=\phi_w(0)$ for some word $w$. Since $\phi_w([-1-i,1+i])\subset\mathbb V$, this line segment contains $p$ in its interior. Thus, in either case, $p$ lies in the interior of a line segment contained in $\mathbb V$, proving the second condition in the definition of a semi-geodesic tree.
\end{proof}

For the rest of this section, we use the following definition. 

\begin{definition}
Given a tree $T$ and a branch point $b\in T$ with $\val_T(b)=4$ let $B_1(b),B_2(b),B_3(b),B_4(b)$ be a counterclockwise enumeration of its branches. We say that $B_1(b),B_{3}(b)$ (resp. $B_2(b),B_{4}(b)$) are \textit{opposite in $T$ at $b$}.
\end{definition}

The next proposition gives a characterization for all trees that are quasiconformally homeomorphic to $\mathbb{V}$. 

\begin{prop}\label{prop: uniformization of T^4}
A planar tree $T\subset \R^2$ is $K$-quasiconformally equivalent to $\mathbb{V}$ if and only if there exist $C\geq 1$ and a sequence $({\bf X}^n)_{n\in\mathbb N_0}$ of finite covers of $T$ with the following properties.
\begin{enumerate}
\item For each $n\in\N_0$, ${\bf X}^n$ is a collection of nondegenerate subtrees of $T$ with ${\bf X}^0 =\{T\}$.
\item If $n\in\mathbb N_0$ and $X\in\mathbf X^n$, then there exists $b_X\in X\setminus\partial X$ such that $X\setminus\{b_X\}$ has exactly four components. Their closures are precisely the four members of $\mathbf X^{n+1}$ contained in $X$. 
\item If $n\in\N$ and $X\in\mathbf X^n$, then $\card(\partial X)\in\{1,2\}$. If $\partial X=\{p_X,q_X\}$, then $p_X$ and $q_X$ belong to opposite branches of $T$ at $b_X$.
\item If $n\in\N$ and $X\in {\bf X}^n$, then there exists $X'\in {\bf X}^{n-1}$ such that $X\subset X'$.
\item If $n\in\N$, $X,Y \in {\bf X}^n$ are distinct, and $X\cap Y \neq \emptyset$, then they intersect on a point of $\partial X\cap \partial Y$ and $\diam{X} \leq C\diam{Y}$.
\item If $n\in\N$, $X,Y \in {\bf X}^n$, and $X\cap Y = \emptyset$, then $\dist(X,Y) \geq C^{-1}\diam{X}$.
\item If $n\in\N_0$, $X\in {\bf X}^n$, $Y\in {\bf X}^{n+1}$, and $X\cap Y \neq \emptyset$, then 
\[ C^{-1}\diam{Y} \leq \diam{X} \leq C\diam{Y}.\]
\item If $n,m \in \N_0$, $X\in {\bf X}^n$, $Y\in {\bf X}^m$, with $m\geq n+C$ and $X\cap Y\neq \emptyset$, then $\diam{Y} \leq \frac12\diam{X}$.
\end{enumerate}
The constants $K$ and $C$ are quantitatively related.
\end{prop}

We note that properties (5)--(8) of Proposition \ref{prop: uniformization of T^4} are exactly those of a \emph{quasivisual approximation} of $T$ in the sense of Bonk and Meyer; see \cite[Definition 2.1]{BM22}. We split the proof of Proposition \ref{prop: uniformization of T^4} in two parts. 

\begin{lem}\label{lem: Vicsek is QC with T^4}
There exists a sequence $({\bf X}^n)_{n\in\mathbb N_0}$ of finite covers of $\mathbb{V}$ that satisfies (1)--(8) of Proposition \ref{prop: uniformization of T^4}.
\end{lem}

\begin{proof}
Recall $A_k=\{1,\dots,k\}$ from \textsection\ref{sec:words}. For each branch point $p$ of $\mathbb V$, set
$$\gen(p):= \min\{|w|:w\in A^*_5,\ p=\phi_w(0)\}.$$

We define three types of tiles on $\V$. First, any image of the set
\[ D_{\dagger}:=\{0\}\cup\bigcup_{k\geq 0}\mathbb V_{5^{(k)}1}\]
under a similarity is called a $\dagger$-tile. Note that $\V$ is the union of 4 $\dagger$-tiles which intersect exactly at 0 and are isometric to each other. Second, $D_{\dagger}$ is a union of 4 subtrees that intersect at $\psi_1(0)$; the three of them are $\dagger$-tiles isometric to each other and the fourth is denoted by $D_{\ddagger}$. Any image of $D_{\ddagger}$ under a similarity is called a $\ddagger$-tile. Third, $D_{\ddagger}$ is a union of 4 subtrees that intersect exactly at $\psi_{13}(0)$; one is a $\ddagger$-tile, two are $\dagger$-tiles that are opposite and isometric to each other, and the third is denoted by $D_{\times}$. Any image of $D_{\times}$ under a similarity is called a $\times$-tile. Finally, $D_{\times}$ is a union of four subtrees that intersect exactly at $\psi_{51}(0)$; two are $\ddagger$tiles which are opposite and isometric and two are $\dagger$-tiles that are opposite and isometric. The boundaries of these three models, relative to $\mathbb V$, are $\partial D_\dagger=\{0\}$, $\partial D_\ddagger=\{0,\phi_1(0)\}$, and $\partial D_\times=\{0,\phi_{13}(0)\}$. 

\begin{figure}[h]
    \centering
    \includegraphics[width=0.31\textwidth]{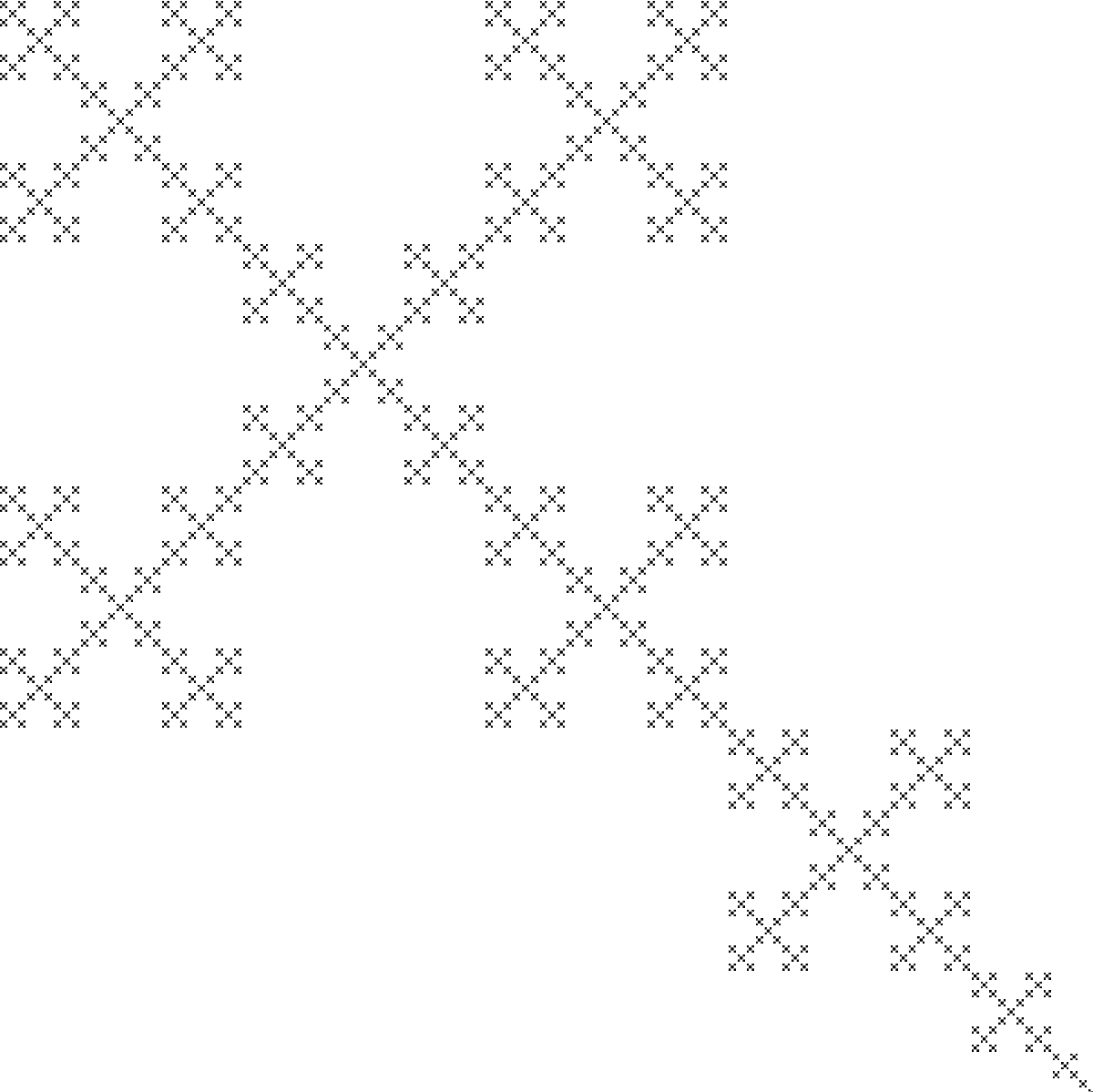}
    \hfill
    \includegraphics[width=0.31\textwidth]{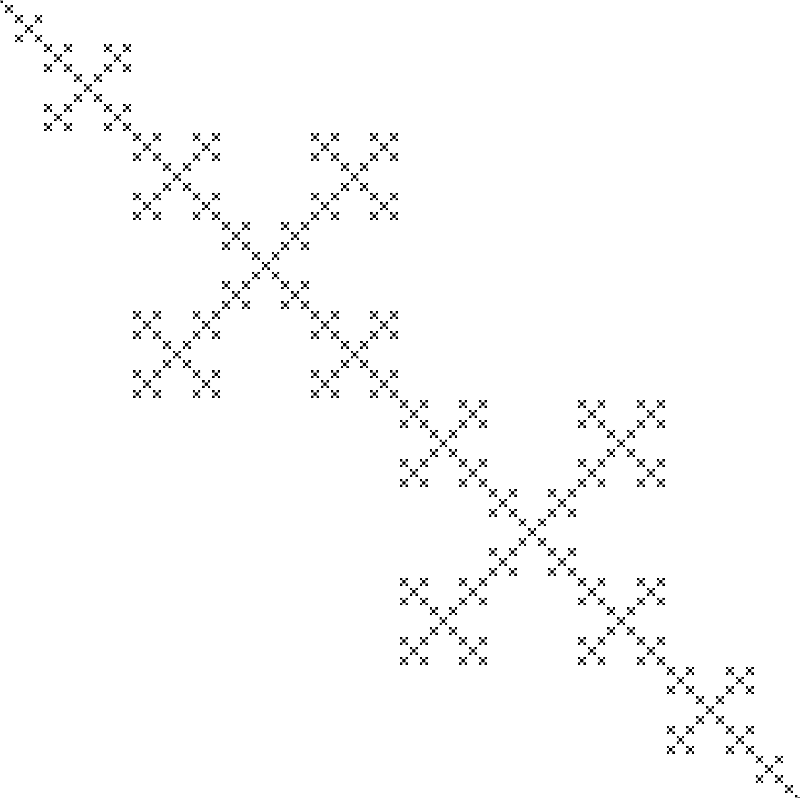}
    \hfill
    \includegraphics[width=0.31\textwidth]{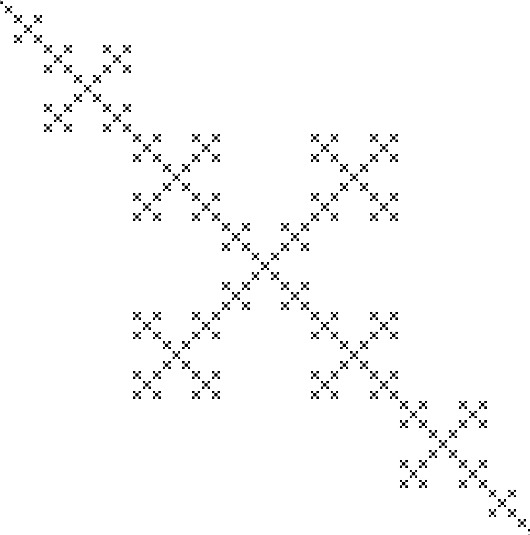}
    \caption{The three model tiles: a $\dagger$-tile, a $\ddagger$-tile,
    and a $\times$-tile.}
    \label{fig:vicsek-model-tiles}
\end{figure}

Let ${\bf X}^0 = \{\V\}$ and let ${\bf X}^1 = \{X_1,\dots,X_4\}$ where each $X_i$ is a $\dagger$-tile isometric to $D_{\dagger}$ and all $X_i$ intersect at $0$. Assume now that we have defined ${\bf X}^n = \{X_w : w\in A_4^n\}$ for some $n\in\N$ so that the following properties hold.
\begin{enumerate}
\item[(a)] Each $X_w$ is either a $\dagger$-tile, or a $\ddagger$-tile, or a $\times$-tile.
\item[(b)] If $X_w$ is a $\dagger$-tile, then $\partial X_w$ contains exactly one branch point.
\item[(c)] If $X_w$ is a $\ddagger$-tile, then $\partial X_w$ contains exactly two points $\{p,q\}$, both branch points. One of them satisfies $p\in \partial X_{w^{\uparrow}}$ and $\gen(p)\leq \gen(q)$.
\item[(d)] If $X_w$ is a $\times$-tile, then $\partial X_w$ contains exactly two points $\{p,q\}$, both branch points. They satisfy $\gen(p)< \gen(q)$.
\end{enumerate}
Fix $w\in A_4^n$. We consider the following instances. 
\begin{enumerate}
\item[(i)] If $X_w$ is a $\dagger$-tile with $\partial X_w=\{p\}$, then $X_w$ consists of 3 isometric $\dagger$-tiles, two of which are opposite, and one $\ddagger$-tile,all intersecting at $b_w\in X\setminus\{p\}$ with $\gen(b_w)=\gen (p)+1$.
\item[(ii)] If $X_w$ is a $\ddagger$-tile with $\partial X_w=\{p,q\}$ and $p\in \partial X_{w^\uparrow}$, then $\gen(p)\le \gen(q)$ and $X_w$ consists of a $\times$-tile with $p$ in its boundary, $2$ opposite isometric $\dagger$-tiles and $1$ $\ddagger$-tile, all intersecting at a branch point $b_w\in X\setminus\{p,q\}$ with $\gen(b_{w})=\gen(q)+1$.
\item[(iii)] If $X_w$ is a $\times$-tile with $\partial X_w=\{p,q\}$, then, $\gen(p)<\gen(q)$, and $X_w$ consists of 2 isometric opposite $\dagger$-tiles and $2$ isometric opposite $\ddagger$-tiles, all intersecting at a branch point $b_w$ with $\gen (b_w)=\gen(q)$.
\end{enumerate}
The new collection ${\bf X}^{n+1}$ consisting of all new tiles produced as above, satisfies (a)--(d). A simple induction now shows that properties (1)--(4) and the intersection assertion in (5) of Proposition \ref{prop: uniformization of T^4} hold.

We next verify the comparability of intersecting tiles. For a similarity copy $X$ of a model tile, let $s(X)$ denote its similarity depth, so that $\diam X=3^{-s(X)}d_{\tau(X)}$, where $\tau(X)\in\{\dagger,\ddagger,\times\}$ is its type and $d_\tau$ is the diameter of the corresponding model. The four initial $\dagger$-tiles have depth $0$. The defining similarities give the following parent--child increments:
\[
\begin{array}{c|c|c|c}\label{eq:table for Vicsek}
\text{parent}&\text{child type}&\text{multiplicity}
&s(X')-s(X)\\ \hline
\dagger&\dagger&3&1\\
\dagger&\ddagger&1&0\\
\ddagger&\times&1&0\\
\ddagger&\dagger&2&2\\
\ddagger&\ddagger&1&1\\
\times&\ddagger&2&1\\
\times&\dagger&2&2
\end{array}
\tag{1.1}
\]

Let $Z,W\in\mathbf X^N$ distinct such that $Z\cap W\neq\emptyset$, and let $Y$ be their longest common ancestor. Write
\[Y=Z_0\supset Z_1\supset\cdots\supset Z_k=Z\quad\text{and}\quad Y=W_0\supset W_1\supset\cdots\supset W_k=W. \]
Their intersection is $\{p\}$, where $p=b_Y$, and $Z_i\cap W_i=\{p\}$ for $1\leq i\leq k$. Thus each subsequent tile in either chain is the unique child containing $p$. If $k=1$, \eqref{eq:table for Vicsek} gives $|s(Z)-s(W)|\leq2$; when $Y=\mathbb V$, the two children have equal depth. Suppose henceforth that $k\geq2$. We consider the following cases.

\emph{Case 1: $Y=\mathbb V$.} The tiles $Z_1,W_1$ are isometric $\dagger$-tiles. Their children containing $p=0$ are $\ddagger$-tiles of equal depth, in which $p$ has smaller generation than the other boundary point. The two chains thereafter have identical similarity-depth increments. Hence $s(Z)=s(W)$.

\emph{Case 2: $Y$ is a $\dagger$-tile.}
Write $\gen(\partial Y)=n$. Then $\gen(p)=n+1$. A $\dagger$-child of $Y$ has depth $s(Y)+1$, and its child containing $p$ is a $\ddagger$-tile of the same depth. The $\ddagger$-child of $Y$ has depth $s(Y)$ and boundary generations $n,n+1$, with $p$ of larger generation. Its child containing $p$ is therefore its $\ddagger$-child, of depth $s(Y)+1$. Consequently, regardless of which two children $Z_1,W_1$ occur, $Z_2,W_2$ are $\ddagger$-tiles of depth $s(Y)+1$, with boundary generations $n+1,n+2$ and $p$ of smaller generation. Their subsequent increments agree, and $s(Z)=s(W)$.

\emph{Case 3: $Y$ is a $\times$-tile.} Write $\partial Y=\{a,q\}$, where $\gen(a)=m<n=\gen(q)$. Then $\gen(p)=n$. Each $\dagger$-child of $Y$ has depth $s(Y)+2$, and its child containing $p$ is a $\ddagger$-tile of the same depth. The two $\ddagger$-children of $Y$ have depth $s(Y)+1$ and boundary generations $m,n$ and $n,n$. In the first, $p$ has larger generation. In the second, $p$ is the newly created boundary point, rather than the inherited one. Thus, in both cases, the child containing $p$ is the $\ddagger$-child, of depth $s(Y)+2$. It follows that $Z_2,W_2$ are $\ddagger$-tiles of equal depth, with boundary generations $n,n+1$ and $p$ of smaller generation. Their subsequent increments agree, so $s(Z)=s(W)$.

\emph{Case 4: $Y$ is a $\ddagger$-tile.} Let $n$ be the larger boundary generation of $Y$, allowing the two generations to be equal. Then $\gen(p)=n+1$. Each $\dagger$-child has depth $s(Y)+2$ and contributes a $\ddagger$-child containing $p$ at the same depth. This latter tile has boundary generations $n+1,n+2$. The $\ddagger$-child of $Y$ has depth $s(Y)+1$, with $p$ of larger boundary generation. Its child containing $p$ is a $\ddagger$-tile of depth $s(Y)+2$, again with boundary generations $n+1,n+2$. Finally, the $\times$-child has depth $s(Y)$. Its child containing $p$ is a $\ddagger$-tile of depth $s(Y)+1$, with both boundary generations equal to $n+1$. In this tile, $p$ is the boundary point inherited from its parent. Therefore $Z_2,W_2$ are $\ddagger$-tiles whose depths differ by at most one. In each, $p$ either has smaller generation or is the inherited boundary point in the equal-generation case. The two alternating-pattern observations apply, so all subsequent similarity-depth increments agree. Hence $|s(Z)-s(W)|\leq1$.

Combining these cases with $k=1$ gives $|s(Z)-s(W)|\leq2$. Set $d_-:=\min\{d_\dagger,d_\ddagger,d_\times\}$ and $d_+:=\max\{d_\dagger,d_\ddagger,d_\times\}$. Then
\[ \left(9\frac{d_+}{d_-}\right)^{-1} \leq \frac{\diam Z}{\diam W}= 3^{s(W)-s(Z)}
\frac{d_{\tau(Z)}}{d_{\tau(W)}} \leq9\frac{d_+}{d_-}.\]
This proves the diameter comparison in (5).

We show (6). For each tile $X$ with two boundary points, denote by $I_X$ its core, namely the arc joining these points. Since such tiles are similarity copies of $D_\ddagger$ or $D_\times$, whose cores are nondegenerate, $\diam I_X\simeq_{\V} \diam X$.

Let $X,Y\in\mathbf X^n$ be disjoint, and suppose that $\diam X\geq\diam Y$. Choose $x\in X$ and $y\in Y$ with $|x-y|=\operatorname{dist}(X,Y)$. After leaving $X$, the arc $\mathbb V[x,y]$ traverses a tile $Z\in\mathbf X^n$ adjacent to $X$. Since $X$ and $Y$ are disjoint, it enters and leaves $Z$ at distinct boundary points. Thus $Z$ is an arc tile and $I_Z\subset\mathbb V[x,y]$. Using the quasiconvexity of $\mathbb V$ (see Lemma \ref{lem:semigeodesic Vicsek}) and the comparison of intersecting tiles, we obtain
\[\diam X\simeq\diam Z\simeq\diam I_Z\leq\diam \mathbb V[x,y]\lesssim |x-y| =\operatorname{dist}(X,Y).\]

Property (7) follows from \eqref{eq:table for Vicsek}. Finally, we verify (8). Let $C_0$ be the comparability constant in (5), and choose $k\in\mathbb N$ sufficiently large that
\[C_0\frac{d_+}{d_-}3^{-k}\leq\frac12.\]
Set $N=3k$. Suppose first that $Y\in {\bf X}^{n+N}, X\in{\bf X}^n$, and $Y\subset X$. Then by \eqref{eq:table for Vicsek} $s(Y)-s(X)\geq k,$ and hence
\[\diam Y\leq \frac{d_+}{d_-}3^{-k}\diam X.\]
If $Y\not\subset X$, it follows similarly by applying property (5).
This proves (8).
\end{proof}

\begin{proof}[{Proof of Proposition \ref{prop: uniformization of T^4}}]
Let $({\bf X}^n)_{n\in\N_0}$ be the subdivision of $\mathbb{V}$ from Lemma \ref{lem: Vicsek is QC with T^4}.

Assume first that there exists a $K$-quasiconformal map $f:\R^2\to\R^2$ with $f(\mathbb{V})=T$. For each $n\in\mathbb N_0$, set ${\bf Y}^n=\{f(X):X\in {\bf X}^n\}$. We show that $({\bf Y}^n)_{n\in\N_0}$ satisfies properties (1)--(8). Properties (1)--(4) are immediate by the fact that $f$ is a homeomorphism. 

For (5) let $n\in\N$ and $f(X_1),f(X_2)\in {\bf Y}^n$ be distinct with $f(X_1)\cap f(X_2) \neq \emptyset$. Then $X_1\cap X_2$ contains exactly one point which is on $\partial X_1\cap \partial X_2$, so $f(X_1)\cap f(X_2)$ contains exactly one point which is on $\partial f(X_1)\cap \partial f(X_2)$. Moreover, since $\diam{X_1} \leq C\diam{X_2}$, by quasiconformality, $\diam{f(X_1)} \leq C_1\diam{f(X_2)}$ for some $C_1\geq 1$ depending only on $C$ and $K$.

For (6), let $n\in\N$ and $f(X_1),f(X_2)\in {\bf Y}^n$ with $f(X_1)\cap f(X_2) = \emptyset$. Then $\dist(X_1,X_2) \geq C^{-1}\diam{X_1}$ and, by quasiconformality, 
\[ \dist(f(X_1),f(X_2)) \geq C_2^{-1}\diam{f(X_1)}\] 
for some $C_2\geq 1$ depending only on $C$ and $K$.

For (7), let $n\in\N$, $f(X_1)\in {\bf Y}^n$, and $f(X_2)\in {\bf Y}^{n+1}$ with $f(X_1)\cap f(X_2) \neq \emptyset$. Then $C^{-1}\diam{X_2} \leq \diam{X_1} \leq C\diam{X_2}$ and, by quasiconformality, $a\diam{f(X_2)} \leq \diam{f(X_1)} \leq b\diam{f(X_2)}$ for some $a,b>0$ depending only on $C$ and $K$.

For (8), by quasiconformality of $f$, there exists $\rho\in (0,1)$ such that if $X_1,X_2 \in {\bf X}^*$ with $X_1\cap X_2 \neq \emptyset$ and $\diam{X_1} \leq \rho \diam{X_2}$, then $\diam{f(X_1)} \leq \frac12 \diam{f(X_2)}$. Choose $k_0\in\N$ such that $(1/2)^{k_0}c < \rho$. Applying (8) for $({\bf X}^n)_{n\in\N_0}$ for $k_0$ many times, we can find $N\in\N$ depending only on $K$ and $C$ such that $\diam{X_1} \leq \rho \diam{X_2}$ if $X_2 \in {\bf X}^n$ and $X_1 \in {\bf X}^m$ with $m\geq n+N$.

For the opposite direction, suppose that $T$ is a tree in $\R^2$ which admits a decomposition $({\bf Y}^n)_{n\in\N_0}$ satisfying (1)--(8). Set $Y_{\e}=T$ and write ${\bf Y}^1 =\{Y_1,\dots,Y_4\}$ where tiles are enumerated according to the counterclockwise orientation. Similarly, set $X_{\e}=\V$ and write ${\bf X}^1 =\{X_1,\dots,X_4\}$ where tiles are enumerated according to the counterclockwise orientation. Suppose now that for some $n\in\N$ we have labeled $\bigcup_{m=0}^n{\bf Y}^m = \{Y_w : w\in \bigcup_{m=0}^nA_4^m\}$ and $\bigcup_{m=0}^n{\bf X}^m = \{X_w : w\in \bigcup_{m=0}^nA_4^m\}$ so that 
\begin{enumerate}
\item $X_w \subset X_{w^{\uparrow}}$ and $Y_w\subset Y_{w^{\uparrow}}$ if $w\in A_4^m$ and $m\in \{1,\dots,n\}$,
\item $\card{\partial Y_w}=2$ if and only if $\card{\partial X_w} =2$,
\item $Y_w\cap Y_u \neq \emptyset$ for $w,u \in \bigcup_{m=0}^nA_4^m$ if and only if $X_w\cap X_u \neq \emptyset$.
\end{enumerate}

Assume first that for $w\in A^n_4$ we have $\card{\partial X_w}=\card{\partial Y_w}=1$. Then there exists unique $u\in A^n_4\setminus\{w\}$ such that $Y_w\cap Y_u \neq \emptyset$. Write $Y_w = Y_{w1}\cup \cdots \cup Y_{w4}$ enumerated according to the counterclockwise orientation of branches and so that $Y_{w1} \cap Y_u \neq \emptyset$. We similarly define tiles $X_{w1},\dots,X_{w4}$.

Assume now that for $w\in A^n_4$ we have $\card{\partial X_w}=\card{\partial Y_w}=2$. Let $b_{Y_w}$ be the branch point at which $Y_w$ is subdivided by (2). Since the boundary points of an arc tile lie in opposite branches at $b_{Y_w}$ (by (3)), we may label the four children of $Y_w$ and the four children of $X_w$ so that the correspondence preserves containment, intersections, arc and leaf tiles, and opposite pairs. 

Proceeding inductively we define $Y_w$ and $X_w$ for every $w\in A_4^*$. By property (8) for $({\bf Y}^n)_{n\in\N_0}$ and $({\bf X}^n)_{n\in\N_0}$, for each $w\in A^{\N}_4$,
\[ \lim_{n\to\infty}\diam{X_{\{w\}}} = \lim_{n\to\infty}\diam{Y_{\{w\}}} =0.\] 
Therefore, we can define surjective maps $\pi_T : A^{\N}_4 \to T$ and $\pi_{\V} : A^{\N}_4 \to \V$ by
\[ \pi_T(w)=\bigcap_{n\geq1}Y_{w\{n\}},\qquad\pi_{\V}(w)= \bigcap_{n\geq1} X_{w\{n\}}.\]
Since $Y_w\cap Y_u \neq \emptyset$ if and only if $X_w\cap X_u \neq \emptyset$, we have that $\pi_T(w)=\pi_T(u)$ if and only if $\pi_{\V}(w) = \pi_{\V}(u)$. Consequently, there exists a bijection $h:T\to\V$ such that $h(\pi_T(w))=\pi_{\V}(w)$. The function $h$ also satisfies $h(Y_w)=\V_w$ for all $w\in A_4^*$. Since $\{\mathbf Y^n\}_{n\in\N_0}$ is a quasivisual approximation of $T$, and $\{\mathbf X^n\}_{n\in\N_0}$ is a quasivisual approximation of $\V$, by \cite[Proposition 2.7]{BM22} $h$ is a quasisymmetry.

In particular, $h$ is a homeomorphism and hence preserves the valence of branch points. Since every branch point of $\V$ is $4$-valent, the same is true for $T$. Moreover, every branch point of $\V$ occurs as the subdivision point of some tile $X_w$. Since $h(Y_w)=X_w$ and the correspondence preserves the four children, we have $ h(b_{Y_w})=b_{X_w}.$ It follows that every branch point of $T$ occurs as some $b_{Y_w}$.

Finally, the labeling of tiles $Y_w$ was chosen so that the cyclic order of the four branches at each $b_{Y_w}$ agrees with that of the corresponding branch point of $\V$. Hence $h$ is branch-orientation-preserving. By Theorem \ref{thm: intro 2 quant.}(i), $h$ extends to a global quasiconformal homeomorphism of $\mathbb R^2$.
\end{proof}

\section{Julia sets that are quadrivalent trees}\label{sec:Julia}
Given $c\in \mathbb{C}$, we denote for the rest of this section $f_c(z) = z^2+c$. The goal of this section is to prove the following proposition which yields Theorem \ref{thm:Vicsek}. 

\begin{prop}\label{prop: J(f) is QC with T^4}
There exist infinitely many distinct $c\in (-2,-1 )$ so that $J(f_c)$ is quasiconformally equivalent to $\V$.
\end{prop}

As a corollary, any planar tree with a decomposition $({\bf X}^n)_{n\in\N_0}$ that satisfies (1)--(8) in Proposition \ref{prop: uniformization of T^4} is a Julia set of a quadratic UQR map. 

We note that Proposition \ref{prop: J(f) is QC with T^4} contrasts with the fact that the Vicsek tree is not the Julia set of any rational map. The statement and proof of that fact were communicated to us by Alastair Fletcher.

\begin{lem}\label{lem:VicseknotJulia}
The Vicsek tree $\V$ is not the Julia set of any rational map.
\end{lem}

\begin{proof}
Let $Q$ be the initial square in the construction of $\V$, and let $\phi_1,\ldots,\phi_5$ be its defining similarities. For $w\in\{1,\ldots,5\}^k$, set $\phi_w=\phi_{w_1}\circ\cdots\circ\phi_{w_k}$ and call $Q_w:=\phi_w(Q)$ a level-$k$ construction cell. Both diagonals of every $Q_w$ belong to $\V$. If $\Delta$ is the union of the diagonals of $Q$, then
\[K_k:=\bigcup_{w\in\{1,\ldots,5\}^k}\phi_w(\Delta)\]
is a connected finite tree consisting of segments in the two diagonal directions. Moreover, $K_k\subset K_{k+1}$ and $\overline{\bigcup_{k\ge0}K_k}=\V$: the union contains the center of every construction cell, whose diameter tends to zero with its level.

Note that every regular analytic arc $A\subset \V$ is a line segment. Indeed, every compact subarc of $A$ avoiding its endpoints lies in some $K_k$. Choose points of the dense union $\bigcup_kK_k$ in the components beyond its two endpoints and use uniqueness of arcs. Since $K_k$ is a finite union of line segments, this subarc is polygonal. However, a regular analytic arc has a continuous nonzero tangent and therefore cannot have corners. Thus $A$ is a line segment. The same conclusion holds for every similarity image of $\V$.

Suppose, towards a contradiction, that $J(f)=\V$ for a rational map $g$ with $\deg g\ge2$. Let $D$ be a diagonal of a construction cell $Q_w$. Choose $z\in\operatorname{int}D$ with $g'(z)\ne0$ and a disk $B=B(z,r)$ on which $g$ is holomorphic and injective. Set $S=D\cap B$. Since $g(\V)=\V$, the set $g(S)\subset \V$ is a regular analytic arc, and hence a line segment.

Let $E\subset S$ consist of the centers of construction cells $Q_{wu}\subset Q_w$ lying on $S$. The subdivisions of $D$ have mesh tending to zero, so $E$ is dense in $S$. At each $q\in E$, the other diagonal of the corresponding cell supplies a nondegenerate segment perpendicular to $S$, with $q$ in its interior.

Choose complex affine similarities $L_1,L_2$ such that $L_1(I)=S$ for an open interval $I\subset\mathbb R$ and $L_2(g(S))\subset\mathbb R$. Set $F=L_2\circ g\circ L_1$ and $E'=L_1^{-1}(E)$. Then $g(I)\subset\mathbb R$, so all Taylor coefficients of $F$ at points of $I$ are real. For $x\in E'$, a sufficiently short vertical segment through $x$ belongs to $L_1^{-1}(\V\cap B)$. Its image under $F$ is a regular analytic arc in $L_2(\V)$, hence a line segment. By conformality, this segment is perpendicular to $F(I)\subset\mathbb R$. Consequently, for sufficiently small $t\in\mathbb R$,
\[\operatorname{Re}F(x+it)=F(x).\]
On the other hand, the Taylor expansion at $x$ gives
\[\operatorname{Re}F(x+it)=F(x)-\frac{F''(x)}{2}t^2+O(t^4).\]
Thus $F''(x)=0$ for every $x\in E'$. Since $E'$ is dense in $I$, the identity theorem gives $F''\equiv0$. Hence $F$ is affine, and therefore $\deg g=\deg F=1$, a contradiction.
\end{proof}

\begin{lem}\label{lem: choice of c}
For every $N\ge3$, there exists $c_N\in(-2,-1)$ such that
\[ f_{c_N}^{N+1}(0)=f_{c_N}^N(0), \qquad f_{c_N}^{j+1}(0)\ne f_{c_N}^j(0),\quad\text{for $j\in \{0,\dots,N-1\}$.}\]
Moreover, $c_N\ne c_M$ whenever $N\ne M$.

Conversely, if $c\in\mathbb R$ satisfies these conditions for some $N\ge3$, then $c\in(-2,-1)$ and $f^N_c(0)$ is a repelling fixed point of $f_c$.
\end{lem}

\begin{proof}
Fix $N\geq 3$ as in the statement. Let $a_j(c)=f_c^j(0)$ and consider $Q_N(c)= (a_N(c)+a_{N-1}(c))/c$ for $c\neq 0$. A simple induction shows that $a_j(c)=c+O(c^2)$ for every $j\geq 1$. Hence $Q_N$ is a polynomial with $Q_N(0)=2$. Moreover, note that $Q_N(-2)=-2$ and $Q_N(-1)=1$, and thus, by the intermediate value theorem, there exists $c_N\in(-2,-1)$ such that $Q_N(c_N)=0$. Equivalently, $a_N=-a_{N-1}$. Then
\[ a_{N+1}=P_{c_N}(a_N) =P_{c_N}(-a_{N-1}) =P_{c_N}(a_{N-1})=a_N,\]
which implies that $a_N$ is fixed. Suppose for a contradiction that $a_j$ is fixed for some $j\in \{1,\ldots,N-1\}$. Then $a_{N-1}=a_N$, and thus $a_N=0$, which implies that $c_N=0$, a contradiction. Similarly, if $M\neq N$, then $c_M\neq c_N$.

It remains to show the converse statement. If $c<-2$, then the orbit of $0$ escapes to infinity, while for $c=-2$ it has exact preperiod two. If $c>1/4$, there are no real fixed points. For $0<c\leq 1/4$, let $\alpha>0$ be the smaller real fixed point. Since $0\leq x<\alpha$ implies $x<f_c(x)<\alpha$, the critical orbit is strictly increasing and remains below $\alpha$. Hence it cannot reach a fixed point after finitely many iterations. The case $c=0$ is also excluded, since $0$ is fixed.  Now suppose that $-1\leq c<0$. Then $f_c([c,0])=[c,c^2+c]\subset[c,0]$, since $c^2+c\leq0$ and $c^2+c\geq c$. Hence, the critical orbit remains in $[c,0]$. Suppose that the critical orbit reaches a negative fixed point $\alpha<0$ for the first time at iteration $N$. Since $f_c^{-1}(\alpha)=\{\alpha,-\alpha\}$, the point immediately preceding this iteration must be $-\alpha>0$, which is impossible because the entire critical orbit is contained in $[c,0]$. Clearly, the orbit cannot reach a positive fixed point either. Thus $c\in(-2,-1)$.

Finally, put $s=-c\in(1,2)$ and $a_2=s(s-1)$. The interval $[c,a_2]$ is forward invariant under $f_c$. Let $q>0$ be the positive fixed point of $f_c$. Since $s=q(q-1)$ and $q<2$, we have $s<q$, and therefore  $[c,a_2]\subset(-q,q).$ In particular, the entire critical orbit belongs to $(-q,q)$. If $a_N=q$, then $a_{N-1}=-q$, a contradiction. Thus, 
\[f^{N}_c(0):=a=\frac{1-\sqrt{1-4c}}{2}<-\frac12,\]
where the last inequality follows from $c<-1$. Consequently, $|f_c'(\alpha)|=|2\alpha|>1$, and hence $\alpha$ is repelling.
\end{proof}

For the rest of this section, fix $N\geq 3$ and $c\in (-2,-1)$ such that $f_c^{N+1}(0) = f_c^N(0)$ and $f_c^{j+1}(0)\neq f^j_c(0)$ for $j=0,\dots,N-1$. Note that since $c\in \R$, the forward orbit of $0$ consists of real numbers. Let $q>0$ be the positive fixed point of $f_c$. 

\begin{lem}\label{lem:J(f) is a tree}
The Julia set $J(f_c)$ is a tree. 
\end{lem}

\begin{proof}
By Lemma \ref{lem: choice of c}, the critical point $0$ is strictly preperiodic and its orbit reaches a negative repelling fixed point $-p$. Moreover, as shown in the proof of Lemma \ref{lem: choice of c}, the interval $[c,c^2+c]$ is forward invariant and contained in $(-q,q)$. Consequently,
\[ \post(f_c):=\bigcup_{n\geq1}f^n_c(0) =\{f_c(0),\ldots,f^N_c(0)\}\subset(-q,q)\setminus\{0\}.\]
In particular, $f_c$ is postcritically finite, and its only finite critical point is strictly preperiodic. By \cite[Corollary~9.5.3] {Beardon}, there is only one connected and simply connected Fatou component. Moreover, $J(f_c)$ is connected and locally connected, since $f_c$ is a postcritically finite polynomial. Thus, $J(f_c)$ is a tree.
\end{proof}

Since $f_c([-q,q])=[c,q]\subset [-q,q]$, we have $[-q,q]\subset J(f_c)$. For each $n\in \N$ set
\begin{equation}\label{eq:finite trees def}
 T_n=f^{-n}_c([-q,q]), \text{ for $n\geq 0$}, \quad\text{and}\quad T=\bigcup_{n\in\N_0}T_n.
\end{equation}
The sequence $(T_n)_{n\geq 0}$ increasing sequence of ``finite'' trees with each $T_n$ obtained recursively by attaching curvilinear arcs at preimages of the critical point $0$. See Figure \ref{fig:iterations of T_n} for the first 4 non zero iterations of the case of $N=3$ with one real solution $c\approx -1.54368901269208$.

\begin{figure}[H]
    \centering
    \includegraphics[width=1\linewidth]{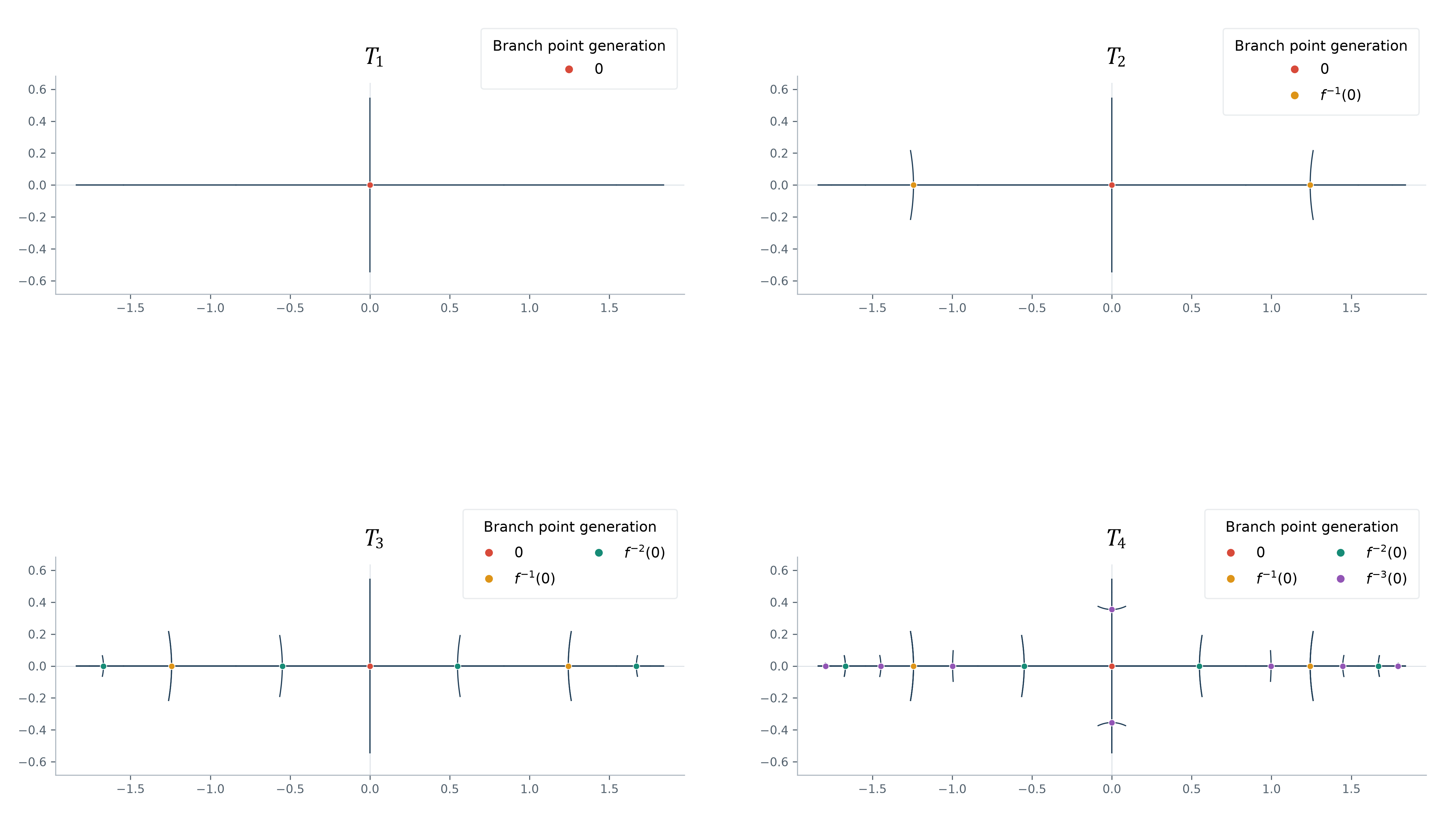}
    \caption{}
    \label{fig:iterations of T_n}
\end{figure}

\begin{lem}\label{lem:valence in tree Julia set}
If $b\in J(f_c)$, then $\val (f_c(b))=\val(b)/d$ where $d$ is the local degree of $f_c$ at $b$. 
\end{lem}

\begin{proof}
Suppose that $d$ is the local degree of $f_c$ at $b$ and that $\val f_c(b)=p$. Then there exist open neighborhoods $U$ of $f_c$ and $V$ of $f_c(b)$ such that $f_c:U\setminus\{b\}\to V\setminus\{f_c(b)\}$ is a $d$-sheeted covering map. 

Let $B_1,\dots,B_p$ be the closures of components of $J(f_c)\setminus \{f_c(b)\}$. For each $i\in \{1,\dots,p\}$ let $\g_i\subset V\cap \overline{B_i}$ be arcs emanating from $f_c(b)$. Then the lifting of these arcs $\widetilde{\g_{i,j}}\subset U$ emanate from $b$ and 
\[ (\widetilde{\g_{i,j}}\setminus\{b\})\cap (\widetilde{\g_{i,j'}}\setminus\{b\})=\emptyset\] 
for distinct $j, j'\in \{1,\dots,d\}$. Since $J(f_c)$ is a tree it follows that for each $i\in \{1,\dots,p\}$ the arcs $\widetilde{\g_{i,j}}$ are located in distinct branches of $b$. Moreover, if for some $i\in \{1,\dots,p\}$, $k,k'\in \{1,\dots,d\}$, the arcs  $\widetilde{\g_{i,k}},\widetilde{\g_{i',k'}}$ are located in the same branch of $b$, then there exists a common subarc $\alpha\subset \widetilde{\g_{i,k}}\cup\widetilde{\g_{i',k'}}$ emanating from $b$ that has two distinct images and hence contradiction. Thus, $\val(b)\geq dp$. Denote now by $B_1,\dots, B_{pd}$ the associated branches of $\{\widetilde{\g_{i,j}}\}_{i,j}$. If $B$ is a branch of $b$ that is not one of $B_1,\dots, B_{pd}$, then let $\omega\subset U\cap B$ be an arc emanating from $b$. Since Julia sets are completely invariant $f_c(\omega)$ lies in one of the branches of $f_c(b)$. Thus, there exists a subarc of $f_c(\omega)$ that has $d+1$ preimages which is contradiction.
\end{proof}

A rational map is called \emph{semihyperbolic} if it has no parabolic periodic points and no recurrent critical points in its Julia set. The polynomial $f_c$ considered here is semihyperbolic. Indeed, its only finite critical point is $0$ with finite orbit, and hence nonrecurrent. Moreover, $f_c$ has no parabolic periodic points due to the fact every cycle of immediate parabolic basins contains a critical point. Since the only finite critical point of $f_c$ is $0$, the existence of a parabolic cycle would imply that the orbit of $0$ converges to this cycle. This contradicts the above orbit, which lands on the fixed point $-p$. Thus $f_c$ is semihyperbolic. 

\begin{lem}\label{lem:valence of J(f)}
The Julia set $J(f_c)$ is a quadrivalent quasiconformal tree with 
$$\mathcal{B}(J(f_c))=\bigcup_{n\in\N_0}f_c^{-n}(0)$$
with the convention that if $n=0$, then $f_c^{0}=\operatorname{Id}$.
\end{lem}

\begin{proof}
Since $J(f_c)$ is a tree from Lemma \ref{lem:J(f) is a tree} and $f_c$ is semi-hyperbolic, by \cite[Theorem 1.1]{CJY} it is a quasiconformal tree. 

Note that $0$ is a branch point of $T$ and that a simple induction shows that the branch points of $T$ are preimages of $0$. Since $c$ is not preimage of 0 and $c\in (-q,q)\subset T_1$, it follows that $c$ is a double point in $T$. Moreover, the local degree of $f_c$ at $0$ is $2$ and hence by Lemma \ref{lem:valence in tree Julia set} we have that $\val_{T}(0)=4$. Let now $C_1,\ldots,C_4$ be the four components of $T\setminus\{0\}$. They lie in four distinct components of $J(f_c)\setminus\{0\}$. Indeed, if $x_i\in C_i$ and $x_j\in C_j$, then $x_i,x_j\in T_n$ for some $n$, and the unique arc $J(f_c)[x_i,x_j]$ coincides with the arc joining them in the subtree $T_n$, which contains $0$. Conversely, every component of $J(f_c)\setminus\{0\}$ is relatively open in $J(f_c)$, since $J(f_c)$ is a tree, and therefore intersects the dense set $T$. Thus there can be no fifth component, and $\val_{J(f_c)}(0)=4$. Consequently, since $0$ is the only critical point of $f_c$, Lemma \ref{lem:valence in tree Julia set} gives that $\bigcup_{n\in\N_0}f_c^{-n}(0)\subset \mathcal{B}(J(f_c))$ and that for all $b\in \bigcup_{n\in\N_0}f_c^{-n}(0)$, $\val_{J(f_c)}(b)=4$.

Conversely, let $b$ be a branch point of $J(f_c)$. Choose three distinct components of $J(f_c)\setminus\{b\}$. Each is relatively open and therefore intersects the dense set $T=\bigcup_n T_n$. Choose one point in each component and then $M$ large enough that all three belong to $T_M$. By uniqueness of arcs, their joining arcs in $T_M$ coincide with those in $J(f_c)$, so $b$ is a branch point of $T_M$. By the finite-tree construction, $b$ is precritical. This proves the asserted description of $\mathcal B(J(f_c))$, and the valence assertion follows from the preceding argument.
\end{proof}

\begin{lem}\label{lem: properties of branches of J(f)}
Let $b\in\mathcal{B}(J(f_c))$. 
\begin{enumerate}[(i)]
\item There exists a minimal $m\geq 1$ such that
\[b\in f_c^{-(m-1)}(0)\cap\mathcal{B}(T_m).\]
\item There are two opposite branches $N_{b,m}^+$ and $N_{b,m}^-$ of $b$ in $T_m$ such that
\[ N_{b,m}^+\cap T_{m-1}=N_{b,m}^-\cap T_{m-1}=\emptyset.\]
\item For every $n> m$ and $s\in\{+,-\}$, let $N_{b,n}^s$ be the unique branch of $b$ in $T_n$ containing $N_{b,m}^s$. Then the set
\[ N_b^s:= \overline{\bigcup_{n\geq m}N_{b,n}^s}\setminus\{b\} \]
is a branch of $b$ in $J(f_c)$ (where the closure is taken in $J( f_c)$). Moreover, $0\notin N_b^s$, and the branches $N_b^+$,$N_b^-$ are opposite of $b$ in $J(f_c)$.
\item After choosing the signs consistently, the map
\[ f_c^{m-1}\big|_{N_b^+\cup N_b^-}: N_b^+\cup N_b^- \to N_0^+\cup N_0^-\]
is a homeomorphism and for each $s\in\{+,-\}$, $N_b^s\cap \post (f_c)=\emptyset$.
\end{enumerate}
\end{lem}

\begin{proof}
By Lemma \ref{lem:valence of J(f)} and the nestedness of $\{T_n\}_{n\in\N_0}$ there exists a minimal $m\geq1$ such that $b\in f_c^{-(m-1)}(0)\cap\mathcal B(T_m).$ This gives (i).

To show (ii), let $I=[-q,c]$. Since $f_c^{-1}(I)$ is the non-real arc through $0$, the two components $N_{0,1}^+$ and $N_{0,1}^-$ of $f_c^{-1}(I)\setminus\{0\}$ are opposite branches of $0$ in $T_1$. Now, it follows by construction that the two pullback arcs of $I$ under $f_c^m$ having $b$ as an endpoint are precisely the lifts under $f_c^{m-1}$ of $N_{0,1}^+$ and $N_{0,1}^-$. Their interiors are contained in $T_m\setminus T_{m-1}$, and hence they determine two branches $N_{b,m}^+$ and $N_{b,m}^-$ of $b$ in $T_m$ satisfying $N_{b,m}^s\cap T_{m-1}=\emptyset.$ Since $b,f_c(b),\ldots,f_c^{m-2}(b)$ are not critical points, $f_c^{m-1}$ is a local orientation-preserving homeomorphism at $b$. It therefore preserves the cyclic order of the four local branches. Since $N_{0,1}^+$ and $N_{0,1}^-$ are opposite at $0$, their lifts $N_{b,m}^+$ and $N_{b,m}^-$ are opposite at $b$.

To show (iii), for every $n> m$, let $N_{b,n}^s$, $s\in\{+,-\},$ denote the unique branch of $b$ in $T_n$ containing $N_{b,m}^s$. Since $T_n\subset T_{n+1}$, $N_{b,n}^s\subset N_{b,n+1}^s$. We claim that
$$N_b^s:=\overline{\bigcup_{n\geq m}N_{b,n}^s}\setminus\{b\}$$
is a branch of $b$ in $J(f_c)$.

Let $B_b^s$ denote the branch of $b$ in $J(f_c)$ containing $N_{b,m}^s$. Clearly, $\bigcup_{n\geq m}N_{b,n}^s\subset B_b^s$. For the reverse inclusion, let $x\in B_b^s$. Since $B_b^s$ is relatively open in $J(f_c)$ and $J(f_c)=\overline{\bigcup_{n\geq 0}T_n}$, there exist $x_k\in B_b^s\cap\bigcup_{n\geq 0}T_n$ such that $x_k\to x$. We may assume that $x_k\notin N_{b,m}^s$. For each $k$, fix $p_k\in N_{b,m}^s$ and choose $n_k\geq m$ sufficiently large so that $x_k,p_k\in T_{n_k}$. Since $x_k$ and $p_k$ belong to the same component of $J(f_c)\setminus\{b\}$, the arc $J(f_c)[x_k,p_k]$ does not contain $b$. Moreover, since $T_{n_k}$ is a subtree of $J(f_c)$ and $x_k,p_k\in T_{n_k}$, we have $J(f_c)[x_k,p_k]= T_{n_k}[x_k,p_k]$. Thus $x_k$ and $p_k$ belong to the same branch of $b$ in $T_{n_k}$. Since $p_k\in N_{b,m}^s\subset N_{b,n_k}^s$, it follows that $x_k\in N_{b,n_k}^s$. Therefore, by the definition of $N_b^s$ and the fact that $x_k\to x\neq b$, we have $x\in N_b^s$. Hence $B_b^s=N_b^s$.

We also note that $0\notin N_b^s$. This is immediate when $b=0$, so suppose that $b\neq0$. Then $m\geq2$ and $0\in T_{m-1}$. By (ii), $N_{b,m}^s\cap T_{m-1}=\emptyset$. Fix $p\in N_{b,m}^s$. Thus $p$ and $0$ lie in distinct branches of $b$ in $T_m$, and consequently the arc $T_m[p,0]$ contains $b$. If $0\in N_{b,n}^s$ for some $n\geq m$, then, since $p\in N_{b,m}^s\subset N_{b,n}^s$, the arc $T_n[p,0]$ would not contain $b$. On the other hand, $p,0\in T_m$ and $T_m$ is a subtree of $T_n$, so uniqueness of arcs gives $T_n[p,0]=T_m[p,0]$, a contradiction. Hence $0\notin N_{b,n}^s$ for every $n\geq m$. Since $N_b^s$ is the branch of $b$ determined by these nested branches, it follows that $0\notin N_b^s$. Last, by nestedness of $\{T_n\}_{n\in\N_0}$ the branches $N_{b,n}^+,N_{b,n}^-$ are opposite in $T_n$, for $n>m$ and thus $N_b^+,N_b^-$ are opposite as well.

To show (iv), we show first the following elementary fact. For every $x\in J(f_c)$, $0\in J(f_c)[x,-x].$ Indeed, $f_c(-z)=f_c(z)$, and hence $J(f_c)$ is invariant under the involution $A(z)=-z.$ Therefore
$$A(J(f_c)[x,-x]) = J(f_c)[-x,x] = J(f_c)[x,-x].$$
The restriction of $A$ to the arc $J(f_c)[x,-x]$, which is homeomorphic to a closed interval, interchanges its two endpoints. Hence, it has a fixed point on this arc. The only fixed point of $A(z)=-z$ is $0$ and thus $0\in J(f_c)[-x,x]$. In particular, if $K\subset J(f_c)\setminus\{0\}$ is connected, then $f_c|_K$ is injective. To see this, if $x,y\in K$, $x\neq y$, and $f_c(x)=f_c(y)$, then $y=-x$. Since a connected subset of a tree containing $x$ and $-x$ contains their unique arc, $J(f_c)[x,-x]\subset K$. Hence, $0\in K,$ which is contradiction.

We now claim that, for every branch point $p\neq 0$ and every sufficiently large $n$, $f_c(N_{p,n}^s)=N_{f_c(p),n-1}^s$, with the signs chosen consistently. Indeed, this holds at the first level at which $p$ appears by the definition. If it holds at level $n$, then every new arc added to $N_{p,n}^s$ is a lift of an arc added to $N_{f_c(p),n-1}^s$. Conversely, every such arc has a unique lift in $N_{p,n}^s$, since $0\notin N_{p,n}^s$ and hence $f_c$ is injective there. Thus the claim follows by induction.

Applying the claim successively to $b,f_c(b),\ldots,f_c^{m-2}(b)$ we obtain 
\[ f_c^{m-1}(N_{b,n}^s)=N_{0,n-m+1}^s.\] 
Passing to the increasing unions and using continuity gives $f_c^{m-1}(N_b^s)\subset N_0^s$. Conversely, let $y\in N_0^s$ and choose $y_k\in N_{0,n_k-m+1}^s$ with $y_k\to y$. Choose $x_k\in N_{b,n_k}^s$ with $f_c^{m-1}(x_k)=y_k$. By compactness of $\overline{N_b^s}$, after passing to a subsequence, $x_k\to x\in\overline{N_b^s}$. Since $y\neq0=f_c^{m-1}(b)$, we have $x\neq b$, and hence $x\in N_b^s$. Therefore $f_c^{m-1}(x)=y$, proving $f_c^{m-1}(N_b^s)=N_0^s$. Finally, by (ii) we have that $f_c^{m-1}$ is injective on $N_b^+\cup N_b^-$. Hence
\[f_c^{m-1}|_{N_b^+\cup N_b^-}:N_b^+\cup N_b^-\longrightarrow N_0^+\cup N_0^-\]
is a homeomorphism. 

It remains to verify that for each $s\in \{+,-\}$, $N_b^s\cap \post (f_c)=\emptyset$. By construction, $N_0^s\cap \post (f_c)=\emptyset$. Now, suppose that $x\in N_b^s\cap\post(f_c)$ for some $s\in\{+,-\}$. Since $\post(f_c)$ is forward invariant, $f_c^{m-1}(x)\in\post(f_c).$ On the other hand, by the above
$f_c^{m-1}(x)\in N_0^s.$ Thus $f_c^{m-1}(x)\in N_0^s\cap\post(f_c),$ which is impossible.
\end{proof}

\begin{lem}\label{lem:koebe}
Let $D,U$ be simply connected domains in $\mathbb{C}$ with $\overline{D} \subset U$ let $z_0\in D$. There exists a constant $C\geq1$, depending only on $D$, $U$, and $z_0$, such that every univalent map $g:U\to\mathbb C$ satisfies
\[ C^{-1}|g'(z_0)|\,|z-w| \leq |g(z)-g(w)| \leq C|g'(z_0)|\,|z-w| \]
for all $z,w\in D$. Consequently, for every compact set $E\subset D$,
\[ C^{-1}|g'(z_0)|\diam E \leq \diam g(E) \leq C|g'(z_0)|\diam E.\]
\end{lem}

\begin{proof}
Since $\overline{D}\subset U$, the annular region separating $D$ from $\partial U$ has positive modulus. More precisely, there is $m>0$ such that $\operatorname{mod}(D,U)\geq m.$ McMullen's two-point distortion theorem \cite[Theorem 2.9]{McM94} therefore gives a constant $C=C(m)\geq1$ such that, for every univalent map $g:U\to\mathbb C$ and all $x,y,z\in D$,
\[C^{-1}|g'(x)|\leq\frac{|g(y)-g(z)|}{|y-z|}\leq C|g'(x)|.\]
Taking $x=z_0$ yields the desired estimate.

Now let $E\subset D$ be compact. The upper bound follows by applying the preceding estimate to arbitrary $y,z\in E$ and taking the supremum. Since $E$ is compact, there exist $y,z\in E$ such that $|y-z|=\diam E$. Applying the lower bound to this pair gives the lower bound.
\end{proof}

\begin{lem}[{\cite{CJY}}]\label{lem:koebe2}
There exist $r>0$, $C>0$, and $\lambda\in(0,1)$ such that, for every $z\in J(f_c)$, every $n\in\mathbb N$, and every component $W$ of $f_c^{-n}(B(z,r))$, we have $ \diam W\leq C\lambda^n$.
\end{lem}

\begin{proof}
The only finite critical point is $0$ and it has a finite orbit, and hence nonrecurrent. Moreover, $f_c$ has no parabolic periodic points. Indeed, every cycle of immediate parabolic basins contains a critical point. Since the only finite critical point of $f_c$ is $0$, the existence of a parabolic cycle would imply that the orbit of $0$ converges to this cycle. This contradicts the above
orbit, which lands on the fixed point $-p$. Thus $f_c$ is semihyperbolic in the sense of \cite{CJY}. The exponential shrinking follows now from  a standard characterization of semihyperbolicity; see \cite{CJY} and \cite{Miha}.
\end{proof}

We split the proof of Proposition \ref{prop: J(f) is QC with T^4} into two parts. For $b\in B(J(f_c))$, set
\[\gen (b):=\min\{k\in\mathbb N_0:f_c^k(b)=0\}.\]
Recall that $N\ge3$, $c\in(-2,-1)$, $a_j=f_c^j(0)$, and $S=\operatorname{post}(f_c)=\{a_1,\ldots,a_N\}$, with $a_j=-p$ for $j\ge N$.

\begin{lem}\label{lem:julia-subdivision}
There exists a sequence of finite covers 
\[ (\mathbf X^n)_{n\geq 0}=(\{X_w:w\in A_4^n\})_{n\geq 0}\] 
of $J(f_c)$, such that if $m_w:=\max_{b\in\partial X_w}\gen (b)$, then the following properties hold.
\begin{enumerate}
\item We have that $\mathbf X^0=\{J(f_c)\}$ and $(\mathbf X^n)_{n\geq 0}$ satisfies properties (1)--(4) and the intersection assertion in (5) of Proposition \ref{prop: uniformization of T^4}. 
\item For $n\ge1$ and $w\in A_4^n$, the points of $\partial X_w\subset B(J(f_c))$ are leaves of $X_w$, and $\card (\partial X)\in \{1,2\}$.
\item For each $w\in A_4^*$, $E_w:=X_w\cap T_{m_w+1}$ is an edge of $T_{m_w+1}$ whose endpoints contain $\partial X_w$, and
\[(X_w\setminus\partial X_w)\cap\bigcup_{j=0}^{m_w}f_c^{-j}(0)=\emptyset.\]
\item For each $w\in A_4^*$, there is a unique point $b_w\in\operatorname{int}E_w$ of least generation. The sets $X_{w1},\ldots,X_{w4}$ are the closures of the four components of $X_w\setminus\{b_w\}$ and
\[|w|-1\le m_w\le N(|w|-1),\quad m_w+1\le\gen (b_w)\le m_w+N.\]
\end{enumerate}
\end{lem}

\begin{proof}
For $M\ge1$, recall $T_M=f_c^{-M}([-q,q])$ and set
\[V_M:=\bigcup_{j=0}^{M-1}f_c^{-j}(0)\ \cup\ f_c^{-M}(\{-q,q\}).\]
We call the closure of a component of $T_M\setminus V_M$ an \emph{edge} of $T_M$; for $E=J(f_c)[u,v]$, set $\operatorname{int}E:=E\setminus\{u,v\}$. A simple induction gives
\begin{equation}\label{eq:iterate-derivative}
(f_c^k)'(x)=2^k\prod_{j=0}^{k-1}f_c^j(x).
\end{equation}
It follows by construction of $T_M$ and \eqref{eq:iterate-derivative} that if $E$ is an edge of $T_M$, then $f_c^M$ maps $E$ homeomorphically onto a real interval with endpoints in $S\cup\{-q,q\}$. 

Every interval between consecutive points of $S\cup\{-q,q\}$ contains a point $y\in B(J(f_c))$ with $\gen (y)\le N-1$. To see this, if both endpoints lie in $S$, their images under $f_c^{N-1}$ coincide, so Rolle's theorem and \eqref{eq:iterate-derivative} give $\gen (y)\le N-2$. Since $S\subset[c,a_2]$, the rightmost interval is $(a_2,q)$. We have $f_c^{N-2}(a_2)=-p<0<q=f_c^{N-2}(q)$, so the intermediate value theorem again gives a point of generation at most $N-2$. Finally, $f_c$ maps $(-q,c)$ homeomorphically onto $(a_2,q)$, giving the bound $N-1$ in the leftmost interval.

Every open interval $I\subset[-q,q]$ with endpoints in $S\cup\{-q,q\}$ contains such a consecutive interval, and hence a point of generation at most $N-1$. The point of least generation in $I$ is unique: two points of generation $d\ge1$ would, by Rolle's theorem applied to $f_c^d$, enclose a point of generation less than $d$; for $d=0$, uniqueness is immediate. Pulling back along $f_c^M|_E$, the first branch point in $\operatorname{int}E$ is therefore unique and has generation in $\{M,\ldots,M+N-1\}$.

Set $\mathbf X^0=\{J(f_c)\}$ and let $\mathbf X^1=\{X_1,\ldots,X_4\}$ be the closures of the four branches of $J(f_c)$ at $0$, labeled in cyclic order. Thus $\partial X_i=\{0\}$ and $m_i=0$. Suppose that $\mathbf X^n=\{X_w:w\in A_4^n\}$ is defined and satisfies:
\begin{enumerate}[(i)]
\item $\card (\partial X_w)\in\{1,2\}$, and every boundary point is a leaf of $X_w$,
\item $E_w:=X_w\cap T_{m_w+1}$ is an edge of $T_{m_w+1}$ whose endpoints contain $\partial X_w$,
\item if $\partial X_w=\{p_w\}$, then $m_w=\gen (p_w)$,
\item if $\partial X_w=\{p_w,q_w\}$, then $\gen (q_w)<\gen (p_w)=m_w$,
\item $(X_w\setminus\partial X_w)\cap\bigcup_{j=0}^{m_w}f_c^{-j}(0)=\emptyset$.
\end{enumerate}
Let $b_w$ be the first branch point introduced in $\operatorname{int}E_w$. By the edge observation,
\[\gen (b_w)\in\{m_w+1,\ldots,m_w+N\}.\]
No branch point of smaller generation occurs in $X_w\setminus\partial X_w$. By the inductive costruction of $T_M$, an arc outside $E_w$ can enter $X_w$ only after an attachment in $\operatorname{int}E_w$; attachments at $\partial X_w$ lie outside $X_w$, since each boundary point is a leaf of $X_w$. Thus no new arc occurs in $X_w$ before $b_w$ is introduced.

Let $X_{w1},\ldots,X_{w4}$ be the closures of the four components of $X_w\setminus\{b_w\}$, labeled in cyclic order. If $\partial X_w=\{p_w\}$, choose the labeling so that $\partial X_{w1}=\{p_w,b_w\}$ and $\partial X_{wj}=\{b_w\}$ for $j\in\{2,3,4\}$. If $\partial X_w=\{p_w,q_w\}$, its two boundary points lie in opposite branches at $b_w$. Choose the labelling so that
\[\partial X_{w1}=\{q_w,b_w\},\quad \partial X_{w3}=\{p_w,b_w\},\quad\partial X_{w2}=\partial X_{w4}=\{b_w\}.\]
The intersection $X_w\cap T_{\gen (b_w)+1}$ consists of four edges meeting at $b_w$. Hence every $X_{wi}$ satisfies (i)--(v), with $m_{wi}=\gen (b_w)$. The collection of all $X_{wi}$ defines $\mathbf X^{n+1}$. The induction and Lemma \ref{lem: properties of branches of J(f)}(iii) prove properties (1)--(4) and the intersection assertion in (5). At each subdivision the exponent increases by at least one and at most $N$, so $n-1\le m_w\le N(n-1)$.
\end{proof}
For the rest of this section, set $\mathbf X^*:=\bigcup_{n\geq1}\mathbf X^n$. For each $X\in\mathbf X^*$, write $m_X:=\max_{b\in\partial X}\gen(b)$ and $E_X:=X\cap T_{m_X+1}$. Thus, if $X=X_w$, then $m_X=m_w$ and $E_X=E_w$.

\begin{lem}\label{lem:julia-cores}
For each $X\in\mathbf X^*$, set
\[I_X:=
\begin{cases}
E_X, & \card (\partial X)=1,\\
J(f_c)[p_X,q_X], & \partial X=\{p_X,q_X\}.
\end{cases}\]
\begin{enumerate}
\item The set $I_X$ is an arc in $X$ and
$\diam I_X\simeq_f\diam X$.
\item For each $b\in\partial X$,
\begin{equation*}
\diam X\simeq_f
\begin{cases}
|(f_c^{m_X})'(b)|^{-1}, & \gen (b)=m_X,\\
|(f_c^{m_X})''(b)|^{-1/2}, & \gen (b)<m_X.
\end{cases}
\end{equation*}
\item If $X,Y\in\mathbf X^n$ and $X\cap Y=\emptyset$, then there exists $Z\in\mathbf X^n$ such that $\card (\partial Z)=2$, $Z\cap X\ne\emptyset$, and
$I_Z\subset J(f_c)[x,y]$ for all $x\in X$ and $y\in Y$.
\end{enumerate}
\end{lem}

\begin{proof}
If $\partial X=\{p_X,q_X\}$, then $I_X=J(f_c)[p_X,q_X]=E_X$. If $\partial X=\{b\}$, then $I_X=E_X=J(f_c)[b,a]$, where $a\in f_c^{-(m_X+1)}(\{-q,q\})$ is its other endpoint.

We describe the finitely many normalized models. For distinct $a,b\in J(f_c)$, let $B_a(b)$ denote the component of $J(f_c)\setminus\{a\}$ containing $b$, and, for $t\in S$, set
\[P_t:=\overline{B_0(t)}\cap\overline{B_t(0)}.\]
Let $P_1,\ldots,P_4$ be the closures of the four components of $J(f_c)\setminus\{0\}$. We equip $P_t$ with the core $J(f_c)[0,t]=[0,t]$, and each $P_i$ with its corresponding edge of $T_1$.

Let $X\in\mathbf X^*$ and write $m=m_X$. Recall that $0\in J(f_c)[z,-z]$ for every $z\in J(f_c)$. Thus $f_c$ is injective on any connected subset of $J(f_c)$ on which $0$ is absent or is a leaf. Applying this observation successively proves that $f_c^m|_X$ is injective: at each step a preimage of $0$ in $X$, if present, is a boundary point and hence a leaf of $X$. The image of $X\setminus\partial X$ is relatively open in $J(f_c)$, since $f_c^m$ is a local homeomorphism there. It is also relatively closed in the component determined by the images of the boundary points, by compactness of $X$. Consequently, $f_c^m$ maps $X$ homeomorphically onto the corresponding model and maps $I_X$ onto its prescribed core.

We use the following square-root observation. Let $K\subset U$ be a compact connected set in a simply connected domain, let $\phi$ be a nonconstant polynomial, and suppose that $U\cap\phi^{-1}(S)=K\cap\phi^{-1}(S)$. Suppose that a solution $h$ of $f_c^m\circ h=\phi$ has coherent holomorphic germs along $K$. Set $H_0=\phi$. Inductively, choose
\[H_j^2=H_{j-1}-c,\quad 1\le j\le m,\]
with the germs prescribed by $H_j=f_c^{m-j}\circ h$ along $K$. A zero of $H_{j-1}-c$ satisfies $\phi(\zeta)=f_c^{j-1}(c)=a_j$, so it belongs to $K$. Its order is even, by the prescribed germ of $H_j$. A holomorphic function with only even-order zeros on a simply connected domain has a holomorphic square root: factor its zeros and take a holomorphic logarithm of the nonvanishing factor. Thus the induction extends each $H_j$ to $U$, and $H_m$ extends $h$. The signs are fixed by the germs along the connected set $K$.

For each regular model $P$, choose a bounded simply connected neighborhood $U_P$ with $U_P\cap S=P\cap S$. Such neighborhoods may be chosen explicitly: write $P\cap\mathbb R=[u,v]$, choose $u'<u\le v<v'$ with $[u',v']\cap S=[u,v]\cap S$, and choose $R$ with $P\subset B(0,R)$ and $-R<u'<v'<R$. Then take
\[U_P=B(0,R)\setminus\bigl(({-\infty},u']\cup[v',\infty)\bigr).\]
A normalized inverse branch regular along $P$ extends holomorphically to $U_P$ by the observation with $K=P$ and $\phi=\operatorname{id}$. This extension is univalent, since it is a right inverse of $f_c^m$.

Suppose first that $\card (\partial X)=1$. The inverse branch is regular along its entire model. Lemma \ref{lem:koebe} on a fixed relatively compact neighborhood of the model in $U_P$ gives (1), since there are finitely many models and their prescribed cores are nondegenerate.

Suppose that $\card (\partial X)=2$. Let $b$ be its older boundary point and set $t=f_c^m(b)\in S$. The point $b$ is the only critical point of $f_c^m$ in $X$, and its local degree is two. Thus $g=(f_c^m|_X)^{-1}$ has square-root behavior at $t$. Choose $\epsilon\in\{-1,1\}$ with $-\epsilon t>0$ and set $\pi_t(\zeta)=t+\epsilon\zeta^2$. The model $P_t$ does not meet the real ray issuing from $t$ away from its core. Hence $\zeta=\sqrt{\epsilon(z-t)}$ lifts $P_t$ to a compact tree $Q_t\subset\{\operatorname{Re}\zeta>0\}\cup\{0\}$, with lifted core $J_t=[0,\sqrt{|t|}]$ and $Q_t\cap\mathbb R=J_t$. Define $h_X=g\circ\pi_t$ on $Q_t$. The quadratic change of variables removes the criticality at $b$; in particular, $h_X$ has a holomorphic germ at $0$ with $h_X'(0)\ne0$.

Set $S_t=\pi_t^{-1}(S)$. This is a finite subset of the real and imaginary axes. Choose $a>\sqrt{|t|}$ below every positive real point of $S_t\setminus Q_t$, choose $0<\delta<a$ below the modulus of every nonzero point of $S_t$, and choose $R>\max\{a,\sup_{Q_t}|\zeta|\}+1$. If the relevant finite set is empty, the corresponding restriction is omitted. Then
\[U_t=\bigl((B(0,R)\cap\{\operatorname{Re}\zeta>0\})\setminus[a,R)\bigr)\cup B(0,\delta)\]
is a bounded simply connected neighborhood of $Q_t$, with
\[U_t\cap S_t=Q_t\cap S_t,\qquad U_t\cap(-U_t)=B(0,\delta).\]
The square-root observation, with $K=Q_t$ and $\phi=\pi_t$, extends $h_X$ holomorphically to $U_t$.

This extension is univalent. If $h_X(\zeta)=h_X(\xi)$, then $\pi_t(\zeta)=\pi_t(\xi)$, so $\xi=\zeta$ or $\xi=-\zeta$. In the second case both points belong to $D:=U_t\cap(-U_t)$. If $\zeta\ne0$, differentiation gives
\[(f_c^m)'(h_X(\zeta))h_X'(\zeta)=2\epsilon\zeta\ne0.\]
Local uniqueness of the inverse therefore implies $h_X(z)=h_X(-z)$ near $\zeta$. The identity theorem on $D$ forces equality near $0$, contradicting $h_X'(0)\ne0$. Thus $h_X$ is univalent on $U_t$. Lemma \ref{lem:koebe} on a fixed relatively compact neighborhood of $Q_t$ gives
\[\diam h_X(J_t)\simeq_c\diam h_X(Q_t).\]
Since $h_X(J_t)=I_X$ and $h_X(Q_t)=X$, this proves (1).

The same coordinates give (2). At the newer boundary, ordinary inverse differentiation gives the first estimate; in a square-root model, this boundary corresponds to $\sqrt{|t|}\ne0$, where $\pi_t'$ does not vanish. At the older boundary, differentiation twice at $0$ gives
\[(f_c^m)''(b)h_X'(0)^2=2\epsilon,\]
which gives the second estimate by Lemma \ref{lem:koebe}. All constants are uniform over the finite model family.

Finally, for (3), let $X,Y\in\mathbf X^n$ be disjoint. Follow the arc joining $X$ to $Y$ after its exit from $X$, and let $Z$ be its first traversed $n$-tile. Then $Z\cap X\ne\emptyset$ and $Z\ne Y$, since $X\cap Y=\emptyset$. The arc enters and leaves $Z$ through distinct boundary points, so $\card (\partial Z)=2$. The core $I_Z$ lies in this joining arc and hence in $J(f_c)[x,y]$ for every $x\in X$ and $y\in Y$.
\end{proof}

The next lemma, in conjunction with Lemma \ref{lem:julia-subdivision} and Proposition \ref{prop: uniformization of T^4}, gives the proof of Proposition \ref{prop: J(f) is QC with T^4}.

\begin{lem}\label{lem:julia-metric-properties}
The subdivision $(\mathbf X^n)_{n\in\mathbb N_0}$ of Lemma~\ref{lem:julia-subdivision} satisfies properties (5)--(8) of Proposition \ref{prop: uniformization of T^4}, with constants depending only on $c$.
\end{lem}

\begin{proof}
For $t\in S$, set
\[d(t):=\min\left\{k\ge1:f_c^{-k}(0)\cap(\min\{0,t\},\max\{0,t\})\ne\emptyset\right\}.\]
Since $f_c^N(0)=f_c^N(t)=-p$, Rolle's theorem and \eqref{eq:iterate-derivative} give $1\le d(t)\le N-1$. The point realizing this minimum is unique by the Rolle argument in Lemma~\ref{lem:julia-subdivision}. Moreover, $d(-p)=2$: set $s=-c=p(p+1)>1$. The generation one points $\pm\sqrt{s}$ lie outside $(-p,p)$, whereas
\[0<u:=\sqrt{s-\sqrt{s}}<p,\qquad f_c^2(-u)=0.\]

We claim that $|m_X-m_Y|\le3N-4$ whenever $X,Y\in\mathbf X^n$ are distinct and $X\cap Y\neq\emptyset$. Write $X\cap Y=\{b\}$, let $Z$ be their longest common ancestor, and write
\[Z=X_0\supset X_1\supset\cdots\supset X_m=X,\qquad Z=Y_0\supset Y_1\supset\cdots\supset Y_m=Y,\]
where $X_i,Y_i$ belong to the same subdivision level and $X_i\cap Y_i=\{b\}$ for $1\le i\le m$. Set $k=\gen (b)$. Since $Z$ is subdivided at $b$, we have $m_{X_1}=m_{Y_1}=k$. After the next subdivision retaining $b$, write $m_{X_2}=k+A_0$, with $1\le A_0\le N$. Thereafter, write $m_{X_{r+2}}=k+A_r$. Since $f_c^{k+A_r}(b)=a_{A_r}$, the normalized core is $[0,a_{A_r}]$, and the subdivision rule gives
\[A_{r+1}=A_r+d(a_{A_r}).\]

Let now $r\ge0$ and set $\tau=\min\{r\ge0:A_r\ge N\}$. Since $1\le d(t)\le N-1$, we have $\tau\le N-1$ and $N\le A_\tau\le2N-2$. For $r<\tau$, we have $r+1\le A_r<N$. For $r\ge\tau$, we have $a_{A_r}=-p$ and $A_{r+1}=A_r+2$. Hence,
\[2-N\le A_r-2r\le2N-2.\]
Applying the same argument to $Y_i$ gives $m_{Y_{r+2}}=k+B_r$, with the same bounds for $B_r-2r$. Since $X$ and $Y$ belong to the same subdivision level, their corresponding indices agree, and
\[|m_X-m_Y|=|(A_r-2r)-(B_r-2r)|\le3N-4.\]
The cases $m=1,2$ follow directly from the initial bounds. This proves the claim.

We now prove (5). Let $X,Y\in\mathbf X^n$ be distinct intersecting tiles, with $X\cap Y=\{b\}$. Note that $b$ is either the newer or the older boundary point of both $X,Y$. If it is the newer boundary point, then $m_X=m_Y=\gen (b)$, and Lemma \ref{lem:julia-cores}(2) gives
\[\diam X\simeq_f|(f_c^{m_X})'(b)|^{-1}=|(f_c^{m_Y})'(b)|^{-1}\simeq_c\diam Y.\]
Suppose now that $b$ is the older boundary point. Assume $m_Y=m_X+j$, with $0\le j\le3N-4$, and put $t=f_c^{m_X}(b)\in S$. Since $(f_c^{m_X})'(b)=0$, the chain rule gives
\[(f_c^{m_Y})''(b)=(f_c^j)'(t)(f_c^{m_X})''(b).\]
Since $0\notin S$, the finitely many values $|(f_c^j)'(t)|$, with $t\in S$ and $0\le j\le3N-4$, are positive. By Lemma \ref{lem:julia-cores}(2),
\[\diam X\simeq_f|(f_c^{m_X})''(b)|^{-1/2}\simeq_c|(f_c^{m_Y})''(b)|^{-1/2}\simeq_c\diam Y.\]

For (6), let $X,Y\in\mathbf X^n$ be disjoint, and choose $x\in X$, $y\in Y$ with $|x-y|=\operatorname{dist}(X,Y)$. Assume $\diam X\ge\diam Y$. By Lemma~\ref{lem:julia-cores}, there exists $Z\in\mathbf X^n$ with $Z\cap X\ne\emptyset$ and $I_Z\subset J(f_c)[x,y]$. Thus, Lemma \ref{lem:julia-cores}(1), (5), and bounded turning give
\[\diam X\simeq_{c}\diam Z\simeq_{c}\diam I_Z\le\diam J(f_c)[x,y]\lesssim_{c}|x-y|=\operatorname{dist}(X,Y).\]

For (7), suppose that $X\in\mathbf X^n$, $Y\in\mathbf X^{n+1}$, and $Y\subset X$. Normalize $X$ to one of the finitely many models in the proof of Lemma~\ref{lem:julia-cores}. The subdivision rule is preserved: its splitting point maps to the first branch point of the prescribed core, and $Y$ maps to one of the four subtrees obtained by subdividing that model there. Thus, in the ordinary or square-root coordinates, there is a fixed model $P$ and one of finitely many nondegenerate compact subtrees $P'\subset P$ such that $X=h(P)$ and $Y=h(P')$, where $h$ is univalent on a fixed neighborhood of $P$. Lemma \ref{lem:koebe} gives
\[\diam Y=\diam h(P')\simeq_c\diam h(P)=\diam X.\]
The case $X=J(f_c)$ is absorbed by enlarging the constant. If $X\cap Y\ne\emptyset$ but $Y\not\subset X$, let $X'\in\mathbf X^n$ contain $Y$. Then $X'\cap X\ne\emptyset$, and the conclusion follows from (5) and the nested case.

Finally, we verify (8). Let $X\in\mathbf X^n$ and $Y\in\mathbf X^{n+k}$ such that $Y\subset X$. By Lemma \ref{lem:koebe2} and the finite model estimates above, there exist $C_1\ge1$ and $\rho\in(0,1)$ such that
\begin{equation}\label{eq:relative-shrinking}
\diam Y\le C_1\rho^k\diam X.
\end{equation}
To see this, first suppose that $f_c^\ell$ maps a compact subtree $Q$ homeomorphically onto a normalized model $R$. Choose $r>0$, $C>0$, and $\lambda\in(0,1)$ as in Lemma \ref{lem:koebe2}. Each model admits a finite cover by compact subtrees $K_i\subset B(z_i,r)$, with $z_i\in R$. Indeed, if $C_b$ is a bounded-turning constant, choose a finite $\delta$-net $\{z_i\}$ in $R$, with $C_b\delta<r$, and set
\[K_i=\overline{\bigcup_{x\in R\cap\overline B(z_i,\delta)}J(f_c)[z_i,x]}.\]
These sets cover $R$ and lie in $\overline B(z_i,C_b\delta)\subset B(z_i,r)$. Since the model family is finite, their number is bounded by a common $L$.

Set $Q_i:=(f_c^\ell|_Q)^{-1}(K_i)$. Each $Q_i$ is compact, connected, and contained in a component of $f_c^{-\ell}(B(z_i,r))$, so Lemma \ref{lem:koebe2} gives $\diam Q_i\le C\lambda^\ell$. Their intersection graph is connected, since they form a finite closed cover of $Q$, and therefore $\diam Q\le LC\lambda^\ell$.

Now let $X\in\mathbf X^n$, $Y\in\mathbf X^{n+k}$, and $Y\subset X$, with $n\ge1$. Put $Q=f_c^{m_X}(Y)$ and $\ell=m_Y-m_X\ge k$. Then $f_c^\ell$ maps $Q$ homeomorphically onto a normalized model, so $\diam Q\lesssim_f\lambda^k$. In ordinary coordinates, Lemma \ref{lem:koebe} gives $\diam Y/\diam X\lesssim_c\diam Q$.

In square-root coordinates, let $W\subset Q_t$ be the connected lift of $Q$, so that $Y=h_X(W)$ and $\pi_t(W)=Q$. Set $\eta=\diam Q$ and fix $\zeta_0\in W$. For $\zeta\in W$,
\[|\zeta-\zeta_0|\,|\zeta+\zeta_0|=|\pi_t(\zeta)-\pi_t(\zeta_0)|\le\eta.\]
Thus $W\subset\overline B(\zeta_0,\sqrt\eta)\cup\overline B(-\zeta_0,\sqrt\eta)$. If these disks are disjoint, connectedness places $W$ in one of them; otherwise their union has diameter at most $4\sqrt\eta$. Hence
\[\diam W\le4\sqrt{\diam Q},\qquad\frac{\diam Y}{\diam X}\lesssim_f\diam W\lesssim_f\lambda^{k/2}.\]
Taking $\rho=\sqrt\lambda$ proves \eqref{eq:relative-shrinking}. The case $X=J(f_c)$ follows by first passing to a level-one tile and enlarging $C_1$.

Let $C_0\ge1$ be the comparability constant in (5), and choose $K\in\mathbb N$ with $2C_0C_1\rho^K\le1$. Suppose that $X\in\mathbf X^n$, $Y\in\mathbf X^m$, $m\ge n+K$, and $X\cap Y\ne\emptyset$. Let $X'\in\mathbf X^n$ contain $Y$. Then $X\cap X'\ne\emptyset$, so (5) and \eqref{eq:relative-shrinking} give
\[\diam Y\le C_1\rho^{m-n}\diam X'\le C_0C_1\rho^K\diam X\le\tfrac12\diam X.\]
Enlarging the common constant in (5)--(8) to at least $K$ proves (8).
\end{proof}

\bibliographystyle{alpha}
\bibliography{quasitrees-ref}

\begin{thebibliography}{CGIV26}

\bibitem[Ahl63]{Ahl63}
Lars~V. Ahlfors.
\newblock Quasiconformal reflections.
\newblock {\em Acta Math.}, 109:291--301, 1963.

\bibitem[AM40]{ADMaclane}
V.~W. Adkisson and Saunders MacLane.
\newblock Extending maps of plane {P}eano continua.
\newblock {\em Duke Math. J.}, 6:216--228, 1940.

\bibitem[Azz15]{Azzam}
Jonas Azzam.
\newblock Hausdorff dimension of wiggly metric spaces.
\newblock {\em Ark. Mat.}, 53(1):1--36, 2015.

\bibitem[BA56]{BA56}
A.~Beurling and L.~Ahlfors.
\newblock The boundary correspondence under quasiconformal mappings.
\newblock {\em Acta Math.}, 96:125--142, 1956.

\bibitem[BC23]{BaudoinChen1}
Fabrice Baudoin and Li~Chen.
\newblock Sobolev spaces and {P}oincar\'{e} inequalities on the {V}icsek fractal.
\newblock {\em Ann. Fenn. Math.}, 48(1):3--26, 2023.

\bibitem[BC24]{BaudoinChen2}
Fabrice Baudoin and Li~Chen.
\newblock Heat kernel gradient estimates for the {V}icsek set.
\newblock {\em Mathematische Nachrichten}, 297(12):4450--4477, 2024.

\bibitem[Bea91]{Beardon}
Alan~F. Beardon.
\newblock {\em Iteration of {R}ational {F}unctions: {C}omplex {A}nalytic {D}ynamical {S}ystems}, volume 132 of {\em Graduate Texts in Mathematics}.
\newblock Springer-Verlag, New York, 1991.

\bibitem[Ber11]{Berg}
Walter Bergweiler.
\newblock Iteration of {Q}uasiregular {M}appings.
\newblock {\em Comput. Methods Funct. Theory}, 10:455--481, 2011.

\bibitem[Bis14]{Bishop}
Christopher~J. Bishop.
\newblock True trees are dense.
\newblock {\em Invent. Math.}, 197(2):433--452, 2014.

\bibitem[BM20]{BM20}
Mario Bonk and Daniel Meyer.
\newblock Quasiconformal and geodesic trees.
\newblock {\em Fund. Math.}, 250(3):253--299, 2020.

\bibitem[BM22]{BM22}
Mario Bonk and Daniel Meyer.
\newblock Uniformly branching trees.
\newblock {\em Trans. Amer. Math. Soc.}, 375(6):3841--3897, 2022.

\bibitem[Bou97]{Bourdon97}
M.~Bourdon.
\newblock Immeubles hyperboliques, dimension conforme et rigidit\'e{} de {M}ostow.
\newblock {\em Geom. Funct. Anal.}, 7(2):245--268, 1997.

\bibitem[BT01]{BishopTyson}
Christopher~J. Bishop and Jeremy~T. Tyson.
\newblock Conformal dimension of the antenna set.
\newblock {\em Proc. Amer. Math. Soc.}, 129(12):3631--3636, 2001.

\bibitem[BT21]{BT21}
Mario Bonk and Huy Tran.
\newblock The continuum self-similar tree.
\newblock In {\em Fractal geometry and stochastics {VI}}, volume~76 of {\em Progr. Probab.}, pages 143--189. Birkh\"auser/Springer, Cham, [2021] \copyright 2021.

\bibitem[Car62]{Car62}
Lennart Carleson.
\newblock Interpolations by bounded analytic functions and the corona problem.
\newblock {\em Ann. of Math. (2)}, 76:547--559, 1962.

\bibitem[Car67]{Car67}
Lennart Carleson.
\newblock {\em Selected problems on exceptional sets}, volume No. 13 of {\em Van Nostrand Mathematical Studies}.
\newblock D. Van Nostrand Co., Inc., Princeton, N.J.-Toronto, Ont.-London, 1967.

\bibitem[Car74]{Ca}
Lennart Carleson.
\newblock The extension problem for quasiconformal mappings.
\newblock In {\em Contributions to analysis (a collection of papers dedicated to {L}ipman {B}ers)}, pages 39--47. Academic Press, New York, 1974.

\bibitem[CG93]{CG-dynamics}
Lennart Carleson and Theodore~W. Gamelin.
\newblock {\em Complex dynamics}.
\newblock Universitext: Tracts in Mathematics. Springer-Verlag, New York, 1993.

\bibitem[CGIV26]{CGIV}
Efstathios-K. Chrontsios-Garitsis, Fotis Ioannidis, and Vyron Vellis.
\newblock Universal quasiconformal trees.
\newblock {\em Adv. Math.}, 501:Paper No. 111107, 78, 2026.

\bibitem[CGV26]{CGV}
Efstathios-K. Chrontsios-Garitsis and Vyron Vellis.
\newblock Sobolev spaces on snowtrees.
\newblock {\em preprint}, 2026.

\bibitem[Cha80]{Charatonik}
Janusz~J. Charatonik.
\newblock Open mappings of universal dendrites.
\newblock {\em Bull. Acad. Polon. Sci. S\'er. Sci. Math.}, 28(9-10):489--494, 1980.

\bibitem[CJY94]{CJY}
Lennart Carleson, Peter~W. Jones, and Jean-Christophe Yoccoz.
\newblock Julia and {J}ohn.
\newblock {\em Bol. Soc. Brasil. Mat. (N.S.)}, 25(1):1--30, 1994.

\bibitem[CSW11]{CSW}
Sarah Constantin, Robert~S. Strichartz, and Miles Wheeler.
\newblock Analysis of the {L}aplacian and spectral operators on the {V}icsek set.
\newblock {\em Commun. Pure Appl. Anal.}, 10(1):1--44, 2011.

\bibitem[FA48]{FA}
Ralph~H. Fox and Emil Artin.
\newblock Some wild cells and spheres in three-dimensional space.
\newblock {\em Ann. of Math. (2)}, 49:979--990, 1948.

\bibitem[FG23a]{FreeGart1}
David Freeman and Chris Gartland.
\newblock Lipschitz functions on quasiconformal trees.
\newblock {\em Fund. Math.}, 262(2):153--203, 2023.

\bibitem[FG23b]{FreeGart2}
David Freeman and Chris Gartland.
\newblock Lipschitz functions on unions and quotients of metric spaces.
\newblock {\em Studia Math.}, 273(1):29--61, 2023.

\bibitem[Go{\l}29]{golab}
Stanis{\l}aw Go{\l}ab.
\newblock Sur quelques points de la th{\'e}orie de la longueur.
\newblock {\em Ann. Soc. Polon. Math.}, 7:227--241, 1929.

\bibitem[Hei01]{Heinonen}
Juha Heinonen.
\newblock {\em Lectures on analysis on metric spaces}.
\newblock Universitext. Springer-Verlag, New York, 2001.

\bibitem[HM98]{HamblyMetz}
B.~M. Hambly and V.~Metz.
\newblock The homogenization problem for the {V}icsek set.
\newblock {\em Stochastic Process. Appl.}, 76(2):167--190, 1998.

\bibitem[HM25]{HeaM}
Vivian~Olsiewski Healey and Govind Menon.
\newblock Scaling limits of branching {L}oewner evolutions and the {D}yson superprocess.
\newblock {\em Pacific J. Math}, 338(1):87--137, 2025.

\bibitem[Hut81]{Hu81}
John~E. Hutchinson.
\newblock Fractals and self-similarity.
\newblock {\em Indiana Univ. Math. J.}, 30(5):713--747, 1981.

\bibitem[IM01]{IM}
Tadeusz Iwaniec and Gaven Martin.
\newblock {\em Geometric Function Theory and Non-linear Analysis}.
\newblock Oxford mathematical monographs. Clarendon Press, 2001.

\bibitem[JS00]{JonesPeter}
Peter~W. Jones and Stanislav~K. Smirnov.
\newblock Removability theorems for {S}obolev functions and quasiconformal maps.
\newblock {\em Ark. Mat.}, 38(2):263--279, 2000.

\bibitem[Kin17]{Kinneberg}
Kyle Kinneberg.
\newblock Conformal dimension and boundaries of planar domains.
\newblock {\em Trans. Amer. Math. Soc.}, 369(9):6511--6536, 2017.

\bibitem[LR12]{LR}
Joan Lind and Steffen Rohde.
\newblock Space-filling curves and phases of the {L}oewner equation.
\newblock {\em Indiana Univ. Math. J.}, 61(6):2231--2249, 2012.

\bibitem[LR18]{RohdeLin}
Peter Lin and Steffen Rohde.
\newblock Conformal welding of dendrites.
\newblock {\em preprint}, 2018.

\bibitem[Mac99]{Mac99}
Paul MacManus.
\newblock Catching sets with quasicircles.
\newblock {\em Rev. Mat. Iberoamericana}, 15(2):267--277, 1999.

\bibitem[McM94]{McM94}
Curtis~T. McMullen.
\newblock {\em Complex Dynamics and Renormalization}, volume 135 of {\em Annals of Mathematics Studies}.
\newblock Princeton University Press, Princeton, NJ, 1994.

\bibitem[McM98]{McMullen}
Curtis~T. McMullen.
\newblock Kleinian groups and {J}ohn domains.
\newblock {\em Topology}, 37(3):485--496, 1998.

\bibitem[Met93]{Metz}
Volker Metz.
\newblock How many diffusions exist on the {V}icsek snowflake?
\newblock {\em Acta Appl. Math.}, 32(3):227--241, 1993.

\bibitem[Mih11]{Miha}
Nicolae Mihalache.
\newblock Julia and {J}ohn revisited.
\newblock {\em Fund. Math.}, 215(1):67--86, 2011.

\bibitem[MP10]{MaPe}
Gaven Martin and Kirsi Peltonen.
\newblock Sto\"ilow factorization for quasiregular mappings in all dimensions.
\newblock {\em Proc. Amer. Math. Soc.}, 138(1):147--151, 2010.

\bibitem[MR05]{MR}
Donald~E. Marshall and Steffen Rohde.
\newblock The {L}oewner differential equation and slit mappings.
\newblock {\em J. Amer. Math. Soc.}, 18(4):763--778, 2005.

\bibitem[MT10]{MT10}
John~M. Mackay and Jeremy~T. Tyson.
\newblock {\em Conformal dimension}, volume~54 of {\em University Lecture Series}.
\newblock American Mathematical Society, Providence, RI, 2010.
\newblock Theory and application.

\bibitem[NV91]{NV91}
Raimo N\"akki and Jussi V\"ais\"al\"a.
\newblock John disks.
\newblock {\em Exposition. Math.}, 9(1):3--43, 1991.

\bibitem[Pan89]{Pan89}
Pierre Pansu.
\newblock Dimension conforme et sph\`ere \`a{} l'infini des vari\'et\'es \`a{} courbure n\'egative.
\newblock {\em Ann. Acad. Sci. Fenn. Ser. A I Math.}, 14(2):177--212, 1989.

\bibitem[Ric66]{Rickman}
Seppo Rickman.
\newblock Characterization of quasiconformal arcs.
\newblock {\em Ann. Acad. Sci. Fenn. Ser. A I}, 395:30, 1966.

\bibitem[Ste70]{Stein}
Elias~M. Stein.
\newblock {\em Singular integrals and differentiability properties of functions}, volume No. 30 of {\em Princeton Mathematical Series}.
\newblock Princeton University Press, Princeton, NJ, 1970.

\bibitem[Tuk80]{Tukia}
Pekka Tukia.
\newblock The planar {S}ch\"onflies theorem for {L}ipschitz maps.
\newblock {\em Ann. Acad. Sci. Fenn. Ser. A I Math.}, 5(1):49--72, 1980.

\bibitem[TV80]{TV80}
P.~Tukia and J.~V\"{a}is\"{a}l\"{a}.
\newblock Quasisymmetric embeddings of metric spaces.
\newblock {\em Ann. Acad. Sci. Fenn. Ser. A I Math.}, 5(1):97--114, 1980.

\bibitem[TV82]{TukVais}
Pekka Tukia and Jussi V{\"a}is{\"a}l{\"a}.
\newblock Quasiconformal extension from dimension {$n$} to {$n+1$}.
\newblock {\em Ann. of Math. (2)}, 115(2):331--348, 1982.

\bibitem[Vic83]{Vicsek}
T~Vicsek.
\newblock Fractal models for diffusion controlled aggregation.
\newblock {\em Journal of Physics A: Mathematical and General}, 16(17):L647, dec 1983.

\bibitem[Zho09]{Zhou}
Denglin Zhou.
\newblock Spectral analysis of {L}aplacians on the {V}icsek set.
\newblock {\em Pacific J. Math.}, 241(2):369--398, 2009.

\end{thebibliography}

\end{document}